\documentclass[UTF-8,reqno]{amsart}
\usepackage{enumerate}
\usepackage{verbatim,listings}
\usepackage{amssymb,url,color, booktabs}

\usepackage{mathrsfs}
\usepackage{amsmath}
\usepackage{verbatim}
\usepackage{fancyhdr}
\usepackage[nobysame]{amsrefs}
\BibSpec{article}{%
	+{}{\PrintAuthors} {author}
	+{,}{ \textrm} {title}
	+{.}{ \textit} {journal}
	+{,}{ \textbf} {volume}
	+{}{ \parenthesize} {date}
	+{,}{ } {pages}
	+{.}{ arXiv:} {eprint}
	+{.}{} {transition}
}
\BibSpec{book}{%
	+{}{\PrintAuthors} {author}
	+{,}{ \textit} {title}
	+{.}{ \textrm} {series} %+{,}{ \textrm} {series}
	+{,}{ Vol.} {volume}
	+{.}{ } {publisher}
	+{,}{ } {date}
	+{.}{} {transition}
}
\usepackage{color}
\usepackage{tikz}
\usepackage{graphicx}
\usepackage[colorlinks=true]{hyperref}
\hypersetup{
	linkcolor=blue,          % color of internal links
	citecolor=red,        % color of links to bibliography
	filecolor=blue,      % color of file links
	urlcolor=cyan
}

\newenvironment{italictext}
{\par\itshape}
{\par}

\numberwithin{equation}{section}

\newcommand{\be}{\begin{eqnarray}}
	\newcommand{\ee}{\end{eqnarray}}
\newcommand{\ce}{\begin{eqnarray*}}
	\newcommand{\de}{\end{eqnarray*}}
\newtheorem{theorem}{Theorem}[section]

\newtheorem{remark}{Remark}[section]
\newtheorem{problem}{Problem}[section]
\newtheorem{lemma}[theorem]{Lemma}
\newtheorem{definition}[theorem]{Definition}
\newtheorem{proposition}[theorem]{Proposition}
\newtheorem{Examples}[theorem]{Example}
\newtheorem{corollary}[theorem]{Corollary}

\newenvironment{proof of theorem 1.4}{{\it Proof of Theorem 1.4}.}{{\hfill 	
		$\square$\hskip - \parfillskip}}
\newenvironment{proof of theorem 1.5}{{\it Proof of Theorem 1.5}.}{{\hfill 	
		$\square$\hskip - \parfillskip}}
\newenvironment{proof of theorem 1.6}{{\it Proof of Theorem 1.6}.}{{\hfill 	
		$\square$\hskip - \parfillskip}}

\makeatletter

\newcommand{\Rmnum}[1]{\expandafter\@slowromancap\romannumeral #1@}
\makeatother

\usepackage{caption}
\def\w{\mathrm{w}}
\def\p{\partial}

\def\[{{\Big[}}
\def\]{{\Big]}}
\def\<{{\langle}}
\def\>{{\rangle}}
\def\({{\Big(}}
\def\){{\Big)}}

\def\bx{{\mathbf{x}}}

\def\min{{\mathord{{\rm min}}}}
\def\Vol{\mathord{{\rm Vol}}}

\def\={&\!\!=\!\!&}

\def\1{{\mathbf{1}}}

\def\geq{\geqslant}
\def\leq{\leqslant}
\def\ge{\geqslant}

\def\xii{\hat{\xi}}
\def\k{\kappa}

\def\div{\mathord{{\rm div}}}

\def\w{\mathrm{w}}
\def\p{\partial}

\def\ga{\gamma}

\def\[{{\Big[}}
\def\]{{\Big]}}
\def\<{{\langle}}
\def\>{{\rangle}}
\def\({{\Big(}}
\def\){{\Big)}}

\def\bx{{\mathbf{x}}}

\def\W{{\mathcal W}}

\def\K{\mathcal{K}}

\def\cc{\mathcal{C}_{\omega_0}}
\def\du{d\mu_{F}}
\def\pa{\partial}
\def\al{\alpha}
\def\be{\beta}
\def\w{\widetilde{\W}}
\def\s{\tilde{s}}
\def\f{\mathcal{F}}
\def\tt{\tilde{\tau}}
\def\q{\tilde{Q}}

\def\min{{\mathord{{\rm min}}}}
\def\Vol{\mathord{{\rm Vol}}}

\def\={&\!\!=\!\!&}
\def\bt{\begin{theorem}}
	\def\et{\end{theorem}}
\def\bl{\begin{lemma}}
	\def\el{\end{lemma}}
\def\br{\begin{remark}}
	\def\er{\end{remark}}
\def\bx{\begin{Examples}}
	\def\ex{\end{Examples}}
\def\bd{\begin{definition}}
	\def\ed{\end{definition}}
\def\bp{\begin{proposition}}
	\def\ep{\end{proposition}}
\def\bc{\begin{corollary}}
	\def\ec{\end{corollary}}

\def\geq{\geqslant}
\def\leq{\leqslant}
\def\ge{\geqslant}

\def\div{\mathord{{\rm div}}}

 \def\R{\mathbb R}
 \def\R{\mathbb R}

\def\<{\langle} \def\>{\rangle}

\def\bpf{\begin{proof}}
	\def\epf{\end{proof}}

\allowdisplaybreaks

\tikzset{
	pattern size/.store in=\mcSize,
	pattern size = 5pt,
	pattern thickness/.store in=\mcThickness,
	pattern thickness = 0.3pt,
	pattern radius/.store in=\mcRadius,
	pattern radius = 1pt}
\makeatletter
\pgfutil@ifundefined{pgf@pattern@name@_ak5ydppfi}{
	\pgfdeclarepatternformonly[\mcThickness,\mcSize]{_ak5ydppfi}
	{\pgfqpoint{0pt}{0pt}}
	{\pgfpoint{\mcSize+\mcThickness}{\mcSize+\mcThickness}}
	{\pgfpoint{\mcSize}{\mcSize}}
	{
		\pgfsetcolor{\tikz@pattern@color}
		\pgfsetlinewidth{\mcThickness}
		\pgfpathmoveto{\pgfqpoint{0pt}{0pt}}
		\pgfpathlineto{\pgfpoint{\mcSize+\mcThickness}{\mcSize+\mcThickness}}
		\pgfusepath{stroke}
}}
\makeatother
\tikzset{every picture/.style={line width=0.75pt}} %set default line width to 0.75pt

\begin{document}
\title{The Anisotropic Capillary $L_p$-Minkowski Problem}\thanks{\it{The second author was support by NSF of Zhejiang province (No. LQN26A010008).}}\thanks{\it {The third author was support by NSFC (Nos. 11871053 and 12261105).}}
%\date{\today}

\author{Shanwei Ding}
\address{School of Mathematics and Statistics\\                	
	Wuhan University\\                Wuhan 430072, P. R. China}
\email{dingsw@whu.edu.cn}
\author{Jinyu Gao}
\address{Department of Mathematics\\Zhejiang Sci-Tech University\\Hangzhou 310018, P. R. China }
\email{gaojy@zstu.edu.cn}
\author{Guanghan Li}
\address{School of Mathematics and Statistics\\                	
	Wuhan University\\                Wuhan 430072, P. R. China}
\email{ghli@whu.edu.cn}
\author{Mengliang Liu$^{*}$}\thanks{* Corresponding author}
\address{School of Mathematics and Statistics\\                	
	Wuhan University\\                Wuhan 430072, P. R. China}
\email{liumengliang@whu.edu.cn}


\begin{abstract}
This paper introduces the \textit{anisotropic $\omega_0$-capillary $p$-sum} of two hypersurfaces in $\mathbb{R}_+^{n+1}$, and establishes  a theory for anisotropic capillary convex bodies. For a smooth convex hypersurface $\Sigma $ with anisotropic $\omega_0$-capillary boundary, we compute the variation of its anisotropic capillary $k$-th quermassintegral via this $p$-sum, thereby defining the associated anisotropic $\omega_0$-capillary $k$-th $p$-surface area measure on the capillary Wulff shape  $\mathcal{C}_{\omega_{0}}$. This motivates us to propose and solve the anisotropic capillary $L_{p}$-Minkowski problem for $p\geq1$.	
\end{abstract}

\maketitle
\setcounter{tocdepth}{2}
\tableofcontents

\section{Introduction}
The classical Minkowski problem is a central problem in convex geometry which is  concerned with prescribing area measure (or Gauss curvature). It is also a classical problem in differential geometry concerning the existence, uniqueness, regularity, and stability of closed convex hypersurfaces with prescribed Gauss curvature as a function of the outer normal.
Major contributions to this problem were made by Minkowski 	\cite{Mi1,Mi2}, Aleksandrov \cite{Al1,Al2,Al3,Al4}, Fenchel and Jessen \cite{F-J},  Lewy \cite{Le1,Le2},  Nirenberg \cite{Ni},  Calabi \cite{Cal}, Pogorelov \cite{Po1,Po2}, Cheng and Yau \cite{Cheng-Yau}, Caffarelli, Nirenberg, and Spruck \cite{C-N-S}, and so on.

In \cite{Lutwak1993}, Lutwak extended the classical Minkowski problem to $L_p$ form, i.e., the $L_p$-Minkowski problem, which prescribes the $L_p$-surface area measure that is related to the differential of the volume functional over convex bodies with the Firey's $p$- sums.
%When $p=1$, it is the classical Minkowski problem. In the case $p=0$, it is the prescribed cone-volume measure problem and it is called the logarithmic Minkowski problem \cite{Lucwak2013}.
The $L_p$-Minkowski problem has been extensively studied for various ranges of the parameter $p$, with foundational results established in works such as \cite{Chou-Wang, Lutwak1993, Lutwak1995, Lutwak2005, Guan-Lin-preprint, Lutwak2010, Zhu, Chen-Li-Zhu-2017, Zhu-2017}. Further developments and related investigations can be found, for instance, in \cite{Li-Xu, Ding-Li-arxiv-2024, Chen-anisotropic-Lp} and references therein.
Furthermore, the geometry of capillary hypersurfaces, i.e., those which meet other hypersurfaces subject to a constant intersection angle, has received increasing attention during the past years. Many results, which are by now classical for closed hypersurfaces, were recently generalized to the capillary setting. The capillary prescribed area measure problem was recently studied in \cite{M-W-W, M-W-W1, H-I-S, W-Z, H-H-I}. More specifically, they solved the capillary $L_p$-Minkowski problem in case $p\ge n+1$ and $p=1$, as well as the capillary $L_p$-Minkowski problem with even prescribed data in case $1<p<n+1$.

On the other hand, recently, Xia \cite{Xia13} studied the anisotropic Minkowski problem, which is a problem of prescribing the anisotropic Gauss-Kronecker curvature for a closed, strictly convex hypersurface in $\mathbb{R}^{n+1}$ as a function on its anisotropic normals in relative or Minkowski geometry. %And in \cite{M-W-W}, Mei, Wang, and Weng studied the capillary Minkowski Problem.

Inspired by the above-mentioned works, we turn to a more difficult setting, that of anisotropic capillary convex geometry, with particular emphasis on its central problem, the anisotropic capillary $L_p$-Minkowski problem. In contrast to the classical capillary prescribed measure problem, where the prescribed measure is defined on a spherical cap, the anisotropic case involves prescribed measures supported on general subsets of the sphere, arising from the presence of the Wulff shape. This feature leads to substantially greater analytical difficulties than those encountered in the isotropic case. Moreover, due to the introduction of anisotropy, each geometric quantity associated with a single convex body is replaced by a mixed quantity involving two convex bodies, which in itself poses substantial additional challenges.

In this paper, we mainly consider anisotropic capillary hypersurfaces in the half-space
\begin{align*}
	\mathbb{R}_+^{n+1}=\{y\in \mathbb{R}^{n+1} :\<y,E_{n+1}\>>0 \},
\end{align*}
where $E_{i}\ (i=1,\cdots,n+1)$ denote the standard coordimate vector in $\mathbb{R}^{n+1}$, and $E_\al\in \partial \overline{\mathbb{R}_+^{n+1}}$ for $\al=1,\cdots,n$.

Let $\mathcal{W}\subset \mathbb{R}^{n+1}$ be a given smooth, closed, and strictly convex hypersurface enclosing the origin $O$ with a support function $F:\mathbb{S}^{n}\rightarrow\mathbb{R}_+$. The Cahn-Hoffman map associated with $F$ is given by
\begin{align*}
	\Psi:\mathbb{S}^n\rightarrow\mathbb{R}^{n+1},\quad	\Psi(x)=F(x)x+\nabla^{\mathbb{S}} {F}(x),
\end{align*}
where $\nabla^{\mathbb{S}}$ denotes the covariant derivative on $\mathbb{S}^n$ with standard round metric. Let $\Sigma$ be a smooth, compact, and orientable hypersurface in $\overline{\mathbb{R}_+^{n+1}}$ with boundary $\partial\Sigma\subset\partial\overline{\mathbb{R}_+^{n+1}}$, which encloses a bounded domain $\hat{\Sigma}$. Let $\nu$ be the unit normal of $\Sigma$ pointing outward of  $\hat{\Sigma}$. Given a constant  $\omega_0 \in(-F(E_{n+1}), F(-E_{n+1}))$, we say $\Sigma$ is an anisotropic $\omega_0$-capillary hypersurface (see \cite{Jia-Wang-Xia-Zhang2023}) if
\begin{align}
	\label{equ:w0-capillary}
	\<\Psi(\nu),-E_{n+1}\>=\omega_0,\quad \text{on}\ \partial\Sigma.
\end{align}

For the convenience of description, we need a constant vector $E_{n+1}^F \in \mathbb{R}^{n+1}$ (ref. \cite{Jia-Wang-Xia-Zhang2023}) defined as
$$
E_{n+1}^F= \begin{cases}\frac{\Psi\left(E_{n+1}\right)}{F\left(E_{n+1}\right)}, & \text { if } \omega_0<0, \\[5pt]
	E_{n+1},  & \text { if } \omega_0=0,
	\\[5pt]
	-\frac{\Psi\left(-E_{n+1}\right)}{F\left(-E_{n+1}\right)}, & \text { if } \omega_0>0.\end{cases}
$$
Note that $E_{n+1}^F$ is the unique vector in the direction $\Psi\left(E_{n+1}\right)$, whose scalar product with $E_{n+1}$ is 1 (see \cite{Jia-Wang-Xia-Zhang2023}*{Eq. (3.2)}).

The classical example of an anisotropic $\omega_0$-capillary hypersurface is the $\omega_0$-capillary Wulff shape (see \cite{Jia-Wang-Xia-Zhang2023}), defined as a portion of a Wulff shape in $\overline{\mathbb{R}^{n+1}_+}$ satisfying the anisotropic capillary boundary condition \eqref{equ:w0-capillary}.
We denote ${\W_{r_0,\omega_0}}=\W_{r_0}(y_0)\cap\overline{\mathbb{R}_{+}^{n+1}}$ with $y_0=r_0\omega_0E^F_{n+1}$, where $\W_{r_0}(y_0)$ is defined in \eqref{equ:w_r(x)}. Additionally, let $\mathcal{C}_{\omega_0}:=\W_{1,\omega_0}=\mathcal{T}(\W\cap\{x_{n+1}\geq -\omega_0\})$, where $\mathcal{T}:y\rightarrow y+\omega_0E_{n+1}^F$ is the translation by $\omega_0E_{n+1}^F$ in $\mathbb{R}^{n+1}$.  See Fig. \ref{fig:9}.
\begin{figure}[h]
	\centering
	\label{fig:8}
\tikzset{every picture/.style={line width=0.75pt}} %set default line width to 0.75pt

\tikzset{every picture/.style={line width=0.75pt}} %set default line width to 0.75pt

\tikzset{every picture/.style={line width=0.75pt}} %set default line width to 0.75pt

\begin{tikzpicture}[x=0.75pt,y=0.75pt,yscale=-1,xscale=1]
%uncomment if require: \path (0,166); %set diagram left start at 0, and has height of 166

%Straight Lines [id:da3131487622660729]
\draw    (71.86,135.63) -- (71.77,47.2) -- (72.04,15.83) ;
\draw [shift={(72.06,13.83)}, rotate = 90.5] [color={rgb, 255:red, 0; green, 0; blue, 0 }  ][line width=0.75]    (10.93,-3.29) .. controls (6.95,-1.4) and (3.31,-0.3) .. (0,0) .. controls (3.31,0.3) and (6.95,1.4) .. (10.93,3.29)   ;
%Straight Lines [id:da04505495286171601]
\draw [color={rgb, 255:red, 208; green, 2; blue, 27 }  ,draw opacity=1 ] [dash pattern={on 4.5pt off 4.5pt}]  (3.79,79.19) -- (20.22,79.29) -- (95.2,79.75) -- (135.2,80) -- (166.79,80.19) ;
%Straight Lines [id:da1266067157253885]
\draw  [dash pattern={on 4.5pt off 4.5pt}]  (103.39,65.14) -- (120.19,42.74) ;
\draw [shift={(121.39,41.14)}, rotate = 126.87] [color={rgb, 255:red, 0; green, 0; blue, 0 }  ][line width=0.75]    (10.93,-3.29) .. controls (6.95,-1.4) and (3.31,-0.3) .. (0,0) .. controls (3.31,0.3) and (6.95,1.4) .. (10.93,3.29)   ;
%Straight Lines [id:da6403854177941242]
\draw    (187.2,80.4) -- (228.81,79.92) ;
\draw [shift={(230.81,79.89)}, rotate = 179.34] [color={rgb, 255:red, 0; green, 0; blue, 0 }  ][line width=0.75]    (10.93,-3.29) .. controls (6.95,-1.4) and (3.31,-0.3) .. (0,0) .. controls (3.31,0.3) and (6.95,1.4) .. (10.93,3.29)   ;
%Straight Lines [id:da7127760594122439]
\draw    (135.2,80) -- cycle ;
%Curve Lines [id:da9895003989823286]
\draw [color={rgb, 255:red, 208; green, 2; blue, 27 }  ,draw opacity=1 ]   (20.22,79.29) .. controls (29.22,57.29) and (49.2,42.08) .. (66.2,41.08) .. controls (83.2,40.08) and (120.49,62.04) .. (115.2,79.88) ;
%Curve Lines [id:da4700016324980798]
\draw    (20.22,79.29) .. controls (16.99,103.82) and (16.26,116.8) .. (42.99,128.82) .. controls (69.73,140.84) and (94.49,106.04) .. (115.2,79.88) ;
%Straight Lines [id:da4693176529160027]
\draw    (315.86,136.63) -- (315.77,50.2) -- (316.04,18.83) ;
\draw [shift={(316.06,16.83)}, rotate = 90.5] [color={rgb, 255:red, 0; green, 0; blue, 0 }  ][line width=0.75]    (10.93,-3.29) .. controls (6.95,-1.4) and (3.31,-0.3) .. (0,0) .. controls (3.31,0.3) and (6.95,1.4) .. (10.93,3.29)   ;
%Straight Lines [id:da952525604303971]
\draw [color={rgb, 255:red, 74; green, 144; blue, 226 }  ,draw opacity=1 ] [dash pattern={on 4.5pt off 4.5pt}]  (247.79,82.19) -- (264.22,82.29) -- (339.2,82.75) -- (379.2,83) -- (410.79,83.19) ;
%Straight Lines [id:da27811729983402356]
\draw  [dash pattern={on 4.5pt off 4.5pt}]  (347.39,68.14) -- (364.19,45.74) ;
\draw [shift={(365.39,44.14)}, rotate = 126.87] [color={rgb, 255:red, 0; green, 0; blue, 0 }  ][line width=0.75]    (10.93,-3.29) .. controls (6.95,-1.4) and (3.31,-0.3) .. (0,0) .. controls (3.31,0.3) and (6.95,1.4) .. (10.93,3.29)   ;
%Straight Lines [id:da6650445571743265]
\draw    (379.2,83) -- cycle ;
%Curve Lines [id:da5250440016999122]
\draw [color={rgb, 255:red, 74; green, 144; blue, 226 }  ,draw opacity=1 ]   (264.22,82.29) .. controls (273.22,60.29) and (293.2,45.08) .. (310.2,44.08) .. controls (327.2,43.08) and (364.49,65.04) .. (359.2,82.88) ;
%Curve Lines [id:da4314368100211412]
\draw    (264.22,82.29) .. controls (260.99,106.82) and (260.26,119.8) .. (286.99,131.82) .. controls (313.73,143.84) and (338.49,109.04) .. (359.2,82.88) ;

% Text Node
\draw (721,21) node    {$0$};
% Text Node
\draw (701,71) node    {$0$};
% Text Node
\draw (52,99) node  [font=\Large,color={rgb, 255:red, 0; green, 0; blue, 0 }  ,opacity=1 ]  {$\W$};
% Text Node
\draw (144,88) node  [font=\tiny,color={rgb, 255:red, 0; green, 0; blue, 0 }  ,opacity=1 ]  {$x_{n}{}_{+}{}_{1} =-\omega _{0}$};
% Text Node
\draw (97,32.4) node [anchor=north west][inner sep=0.75pt]  [font=\tiny]  {$\W\cap \{x_{n}{}_{+}{}_{1} \geq -\omega _{0}\}$};
% Text Node
\draw (198,85.4) node [anchor=north west][inner sep=0.75pt]  [font=\small]  {$\mathcal{T}$};
% Text Node
\draw (296,102) node  [font=\Large,color={rgb, 255:red, 0; green, 0; blue, 0 }  ,opacity=1 ]  {$\W$};
% Text Node
\draw (388,91) node  [font=\tiny,color={rgb, 255:red, 0; green, 0; blue, 0 }  ,opacity=1 ]  {$x_{n}{}_{+}{}_{1} =0$};
% Text Node
\draw (363,35.4) node [anchor=north west][inner sep=0.75pt]  [font=\tiny]  {$\cc$};
% Text Node
\draw (321,84.4) node [anchor=north west][inner sep=0.75pt]    {$o$};

\end{tikzpicture}

\caption{Translation of $\mathcal{T}$} \label{fig:9}
\end{figure}
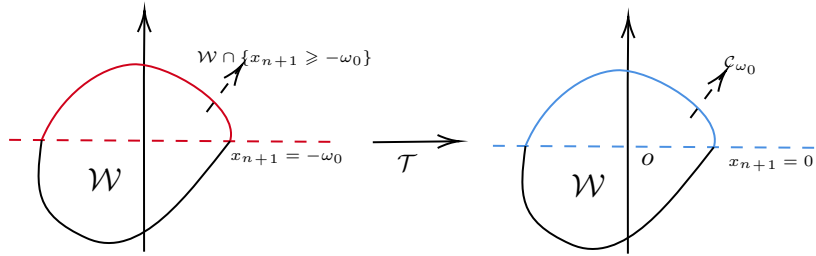 \\

Let $\Sigma$ be an anisotropic strictly convex  $\omega_0$-capillary hypersurface in $\overline{\mathbb{R}^{n+1}_+}$ with boundary on $\partial\overline{\mathbb{R}^{n+1}_+}$. We define the anisotropic capillary Gauss map by $\widetilde{\nu_{F}} = \mathcal{T}\nu_{F} = \nu_{F} + \omega_0 E_{n+1}^F$, and use it to reparametrize the anisotropic support function $\hat{s}$  on $\mathcal{C}_{\omega_0}$, where the anisotropic normal $\nu_{F}$ of $\Sigma$ and support function $\hat{s}$ are defined in Section \ref{sec 2}. The regular anisotropic capillary Minkowski problem in the half-space asks the following question:

\textbf{Anisotropic Capillary Minkowski Problem}: Given a positive function $f$ on $\mathcal{C}_{\omega_0}$, can one find a  strictly convex anisotropic $\omega_0$-capillary hypersurface $\Sigma\subset \overline{\mathbb{R}^{n+1}_+}$ such that
\begin{align}\label{equ:K^-1=f}
	\frac{1}{K(\widetilde{\nu_{{F}}}^{-1}(\xi))}=f(\xi),\ \ \ \forall \xi \in \cc.\
\end{align}
Here $K$ is the anisotropic Gauss-Kronecker curvature of $\Sigma$.

The following energy functional
\begin{align}\label{f1}
	\mathcal{E}(\Sigma)=|\Sigma|_{F}+\omega_{0}|\widehat{\partial\Sigma}|,
\end{align}
was considered by Jia, Wang, Xia and Zhang  \cite{Jia-Wang-Xia-Zhang2023}, and earlier by Philippis and Maggi \cite{P-M}.
Here $\widehat{\partial\Sigma}=\hat{\Sigma}\cap \partial\overline{\mathbb{R}^{n+1}_+}$ denotes the domain enclosed by $\partial\Sigma\subset\mathbb{R}^n=\partial\overline{\mathbb{R}^{n+1}_+}$, $|\cdot |$ denotes the volume of the domain in $ \mathbb{R}^n$ or $ \mathbb{R}^{n+1}$. Note that the anisotropic area element is defined by ${\rm d}{\mu}_{F}:=F(\nu){\rm d}\mu_g$, where ${\rm d}\mu_g$ is the volume form induced by the standard Euclidean inner product. The anisotropic area of $\Sigma$ is given by $$|\Sigma |_F=\int_{\Sigma} {\rm d}{\mu}_{F}.$$ Through direct computation, we find that the functional \eqref{f1} has an equivalent integral form as follows:
\begin{align*}
	\mathcal{E}(\Sigma)=\int_{\Sigma}\left(1+\omega_{0}G(\nu_{F})\(\nu_F, E_{n+1}^{F}\)\right){\rm d}{\mu}_{F},\
\end{align*}
where $G$ is the new metric defined in Section \ref{subsec 2.3}.
This simple observation yields the following area measure $m_{\omega_{0}}$ associated with a given strictly convex anisotropic capillary hypersurface $\Sigma$,
$$m_{\omega_{0}}(B):=\mathcal{E}(\widetilde{\nu_{{F}}}^{-1}(B))= \int_{\widetilde{\nu_{{F}}}^{-1}(B)}\left(1+\omega_{0}G(\nu_{F})\(\nu_F, E_{n+1}^{F}\)\right){\rm d}{\mu}_{F},\ $$
for any Borel set $B \subset \cc$, which is called the anisotropic capillary area measure of $\Sigma$. By a transformation, it can also be written in the following form
$$m_{\omega_{0}}(B):=\int_{B}\frac{F(\nu(\xi))+\omega_{0}\<\nu(\xi), E_{n+1}^{F}\>}{K(\widetilde{\nu_{{F}}}^{-1}(\xi))}{\rm d}\mathcal{H}^n(\xi) ,$$
for any Borel set $B \subset \cc$, where  $F(\nu(\xi))+\omega_{0}\<\nu(\xi), E_{n+1}^{F}\>$ is the support function of $\mathcal{C}_{\omega_{0}}$, and $\mathrm{d}\mathcal{H}^n(\xi)$
denotes the $n$-dimensional Hausdorff measure. %the  Euclidean volume element on $\mathcal{C}_{\omega_0}$ which induced from the ambient Euclidean space.
Therefore, we can formulate the anisotropic capillary Minkowski problem from the perspective of anisotropic capillary surface area measures:
\begin{italictext}
	Given a Borel measure $m$ on $\cc$, what conditions on $m$ characterize the existence and uniqueness of a generalized convex anisotropic capillary hypersurface with constant $\omega_{0}$ whose induced anisotropic capillary area measure satisfies $m_{\omega_{0}} = m$?
\end{italictext}

We reduce the anisotropic capillary Minkowski problem to the existence of the following Monge-Ampère type equation with a naturally Robin boundary value condition for the anisotropic capillary support function $\hat{s}: \cc \rightarrow \mathbb{R}$
\begin{equation}\label{equ:p=1 pde-inv-Gauss}
	\renewcommand{\arraystretch}{1.5}
	\left\{
	\begin{array}{rll}
		\det (\mathring{\nabla}_{ij}\hat{s}-\frac{1}{2}Q_{ijk}\mathring{\nabla}_k\hat{s}+\hat{s}\mathring{g}_{ij})&=f,  \ \ \ \ \ \ \ \ \ \ &\text{in\ } \mathcal{C}_{\omega_0},\\
		 \<{\mathring{\nabla}} \hat{s},E_{n+1}\>&=\omega_0\cdot \hat{s},  \ \ \ \ &\text{on}\ \partial\mathcal{C}_{\omega_0},
	\end{array}	
	\right.
\end{equation}
where $\mathring{g}_{ij}=G(\mathcal{T}^{-1}\xi)(\partial_{i}\xi,\partial_j\xi)$,  $\mathring{\nabla}$ is the Levi-Civita connection of $\mathring{g}$ on $\mathcal{C}_{\omega_0}$, and $Q$ is a $(0,3)$-tensor defined by \eqref{equ:Q}. The boundary condition in \eqref{equ:p=1 pde-inv-Gauss} is equivalent to
\begin{align}
	\label{equ:boundary equ s p=1}
	\mathring{\nabla}_{\mu_F}\hat{s}=\frac{\omega_0}{F(\nu)\<\mu,E_{n+1}\>}\hat{s}, \quad\text{or}\quad \<{\mathring{\nabla}} \hat{s},\mu\>=\frac{\omega_0}{\<\mu,E_{n+1}\>} \hat{s},
\end{align}
where $\nu$ is the unit outward normal vector of  $\mathcal{C}_{\omega_0} $ in $\overline{\mathbb{R}_+^{n+1}}$,  $\mu$ is the unit isotropic outward co-normal of $\partial\mathcal{C}_{\omega_0}$ in $\mathcal{C}_{\omega_0}$, $\mu_{F}=A_F(\nu)\nu$ with $A_F$ defined by \eqref{equ:A_F>0}, and $\mu_F$ is called the anisotropic outward co-normal of $\partial\mathcal{C}_{\omega_0}$ in $\mathcal{C}_{\omega_0}$. For the details of the proof, we refer to Lemma \ref{lemma:Xia-Lemma2.4}. In Section \ref{section5}, a straightforward observation shows that we are allowed to restrict $\hat{s}$ so that it satisfies the following orthogonal condition:
\begin{align}
	\label{equ:int E1--En}
	\int_{\mathcal{C}_{\omega_0}}G(\hat{\xi})(\hat{\xi},E_{\al})\hat{s}\ {\rm d}{\mu}_{F}=0,\quad \forall \al=1,\cdots,n,
\end{align}
where $\hat{\xi}=\mathcal{T}^{-1}\xi\in \W$.  % andthe new metric $G$ and anisotropic measure ${\rm d}{\mu}_{F}$ are defined in Section \ref{sec 2}.

A fundamental difficulty in analyzing equation \eqref{equ:p=1 pde-inv-Gauss} arises from the coupling of its nonlinear elliptic structure with the gradient terms on the left-hand side. This combination makes the derivation of a priori estimates, particularly $C^2$ estimates, a critical challenge from both geometric and analytic perspectives.
In the isotropic setting, the derivation of $C^2$ estimates relies crucially on the convexity of $\partial\mathcal{C}_{\theta}\subset\mathcal{C}_{\theta}$, where the spherical cap $\mathcal{C}_{\theta}$ is defined as in \cite{M-W-W},
$$\mathcal{C}_{\theta}:= \Big\{y \in \overline{\mathbb{R}^{n+1}_+}|\ \  |y+\cos \theta E_{n+1}|=1\Big\},$$
and on employing the classical barrier technique \cite{Lions-Trudinger-1986, Ma-Qiu} for a Neumann boundary condition posed on a uniformly convex domain.
%and by the classical barrier technique (see \cite{Lions-Trudinger-1986,Ma-Qiu})  with a Neumann boundary condition on a uniformly convex domain.
Therefore, naturally, in the anisotropic problem, we need the following Condition.
\begin{definition}
We call the boundary of a $\omega_0$-capillary Wulff shape $\cc \subset \widetilde{\W}=\mathcal{T}(\W)$ is anisotropically convex if $\pa \cc \subset \cc$ satisfies
$\left(\hat{h}_{\al \be}^{\pa\cc }\right)>0,$
where $\hat{h}_{\al \be}^{\pa\cc }:=\tilde{g}(\tilde{e}_{\al}, \tilde{\nabla}_{\tilde{e}_{\be}}\tilde{e}_{n})$, $\tilde{e}_{\al}$ and $\tilde{g}$ are anisotropic quantities with respect to the translated Wulff shape $\widetilde{\W}$ defined in Section \ref{section4}.
\end{definition}
This Condition was first used in \cite{Arxiv1} and is equivalent to
\begin{align}\label{condition}
	\omega_{0}< \frac{Q(z)(Y,Y,\mu_{F})\<E_{n+1},\mu\>F(\nu)}{G(z)(Y,Y)},
\end{align}for all $z \in \mathcal{W}\cap \{x_{n+1}=-\omega_{0}\}$ and $Y \in T_{z}((\mathcal{W}\cap \{x_{n+1}=-\omega_{0}\})$. For details, we refer to Remark \ref{r4.1} and Proposition \ref{g2} in Section \ref{section4}. In the special case where $\W=\mathbb{S}^n$,  the Wulff  spherical cap $\cc$ becomes an isotropic capillary spherical cap, denoted by $\mathcal{C}_{\theta}$, and Condition \eqref{condition} is equivalent to  $\theta\in(0,\frac{\pi}{2})$. In both anisotropic and isotropic cases, we aim to remove this restriction and hope to accomplish this in future work.
We aim to provide a complete extension of the known results for the capillary $L_p$-Minkowski problem with $p\ge1$ to the anisotropic case, while overcoming all the technical difficulties inherent in this extension.
The main results of this paper are summarized in the following theorems. The first theorem is a direct extension of that in \cite{M-W-W}:

\begin{theorem}\label{thm:p=1}
	Let $\omega_0 \in (-F(E_{n+1}), F(-E_{n+1}))$ such that  the boundary of $\omega_0$-capillary Wulff shape $\pa \cc \subset \cc$ is anisotropically convex, %Condition \eqref{condition},
	and $f(\xi) \in C^{2}(\cc)$ be a positive function satisfying
	\begin{align}
		\label{equ:int-fE1--En}
		\int_{\mathcal{C}_{\omega_0}}G(\hat{\xi})(\hat{\xi},E_{\al})f(\xi){~\rm d}{\mu}_{F}=0,\quad \forall \al=1,\cdots,n.
	\end{align}	
		Then
	there is a $C^{3,\ga}\  (0<\ga<1)$ strictly convex $\omega_0$-capillary hypersurface $\Sigma\subset\overline{\mathbb{R}_{+}^{n+1}}$ such that \eqref{equ:K^-1=f} holds. Moreover, $\Sigma$ is unique up to horizontal translations.
	%whose reciprocal anisotropic Gauss-Kronecker curvature is $f$ as a function on its anisotropic capillary normals. Moreover, $\Sigma$ is unique up to horizontal translations.
\end{theorem}
%The volume-preserving flows can be used to prove the isoperimetric inequality.
%We remark that, for the isotropic case, the convexity preserving result for  $\theta\in(\frac{\pi}{2},\pi)$ has not been solved yet. So the corresponding capillary Alexandrov-Fenchel inequalities are only  built for $\theta\in(0,\frac{\pi}{2}]$ in the half-space or the unit Euclidean ball. For the anisotropic capillary hypersurface in the half-space, Theorem \ref{thm:convex preservation} proves the convexity preserves along the flow  under Condition $\ref{condition}$ which enclose the case  for some Wulff shapes and $\theta\in(\frac{\pi}{2},\pi)$ (see Example \ref{exam:x^4T}).

In Section \ref{section3}, we define the anisotropic capillary $p$-sum and  the  anisotropic capillary $p$-surface area measure, and pose the corresponding prescribed measure problem, namely the anisotropic capillary $L_p$-Minkowski problem. This problem can be reduced to establishing the existence of convex solutions to a fully nonlinear elliptic PDE with a natural Robin boundary condition:
\begin{equation}\label{equ:p pde-inv-Gauss}
	\renewcommand{\arraystretch}{1.5}
	\left\{
	\begin{array}{rll}
		 \det \left(\mathring{\nabla}_{ij}\hat{s}-\frac{1}{2}Q_{ijk}\mathring{\nabla}_k\hat{s}+\hat{s}\mathring{g}_{ij}\right)&=f\hat{s}^{p-1}, & \text{in\ } \mathcal{C}_{\omega_0},\\
		  \mathring{\nabla}_{\mu_F}\hat{s}&=\frac{\omega_0}{F(\nu)\<\mu,E_{n+1}\>}\hat{s},& \text{on}\ \partial\mathcal{C}_{\omega_0}.\
	\end{array}	
	\right.
\end{equation}

 For a class of special Wulff shapes, we have the following concept of even functions in the anisotropic capillary setting.
\begin{definition}\label{gd}
Let $\xi=(\xi_{1}, \dots, \xi_{n}, \xi_{n+1}) \in \cc$, and define $\bar{\xi}:=(-\xi_{1}, \dots, -\xi_{n}, \xi_{n+1})$. If the support function $F$ of $\W$ satisfies
\begin{align*}
F(x_{1}, \dots, x_{n}, x_{n+1})&= F(x_{1}, \dots, x_{n}, -x_{n+1}), \\
F(-x_{1}, \dots, -x_{n}, x_{n+1})&= F(x_{1}, \dots, x_{n}, x_{n+1}),
\end{align*}
for all $(x_{1}, \dots, x_{n}, x_{n+1}) \in \mathbb{S}^{n}$, then $\W$ (or equivalently $F$) is called symmetric.
Given a symmetric $\W$, a function $f \in C^{2}(\cc)$ is called an anisotropic capillary even function on $\cc$ if it satisfies
$$f(\xi) = f(\bar{\xi}),\ \forall \xi\in \cc.$$
\end{definition}
The following theorem generalizes the result of Mei, Wang, and Weng \cite{M-W-W1} to the anisotropic capillary setting.
\begin{theorem}\label{thm:p>1}
	Let $\omega_0 \in (-F(E_{n+1}), F(-E_{n+1}))$ such that the boundary of $\omega_0$-capillary Wulff shape $\pa \cc \subset \cc$ is anisotropically convex, and $f \in C^{2}(\cc)$ be a positive function.
	\begin{enumerate}
		\item If $1<p<n+1$, assume that $\W$ is symmetric, $\omega_{0}\leq 0$,  and that $f$ is even  on $\mathcal{C}_{\omega_{0}}$. Then there exists a unique smooth anisotropic capillary even solution $\hat{s}$ to Eq. \eqref{equ:p pde-inv-Gauss}.
		\item If $p=n+1$, there exists a unique (up to a dilation) positive smooth solution $\hat{s}$ and a positive constant $\eta$ solving
		\begin{align}
			\begin{cases}
				\det \left(\hat{s}_{ij}-\frac{1}{2}Q_{ijk}\hat{s}_{k}+\hat{s}\mathring{g}_{ij}\right)&=\eta f\hat{s}^{n},\ \  \ \text{in\ } \mathcal{C}_{\omega_0},\\
				\ \ \ \ \ \ \ \ \ \ \ \ \ \ \ \ \ \ \ \ \<{\mathring{\nabla}} \hat{s},E_{n+1}\>&=\omega_0\cdot \hat{s},\ \ \text{on}\ \partial\mathcal{C}_{\omega_0}.\
			\end{cases}
		\end{align}
		\item If $p>n+1$, there exists a unique smooth  solution $\hat{s}$ to Eq. \eqref{equ:p pde-inv-Gauss}.
	\end{enumerate}
\end{theorem}

\begin{remark}
For readers interested in examples that satisfy the special Wulff shape construction in the first part of Theorem \ref{thm:p>1}, we refer to the appendix of \cite{Arxiv1}.\end{remark}
\textbf{The rest of this paper is organized as follows.}
In Section \ref{sec 2}, we systematically introduce the preliminary notions of anisotropic capillary hypersurfaces, with particular emphasis on a precise definition of the anisotropic support function, which corresponds to the support function of isotropic capillary hypersurfaces in Euclidean space. We also define the mixed quermassintegrals for anisotropic capillary hypersurfaces and establish the well-posedness of these definitions. In Section \ref{section3}, we extend the classical theory of convex bodies to the anisotropic capillary hypersurface case, and define the anisotropic capillary $p$-sum and the $k$-th anisotropic capillary $p$-surface area measure. We propose two prescribed measure problems, which include the anisotropic capillary $L_p$-Minkowski problem. In Section \ref{section4}, we establish the a priori estimates for solutions to problems in Theorem \ref{thm:p=1} and Theorem \ref{thm:p>1}. The $C^1$ and $C^2$ estimates depend on the Condition \eqref{condition}. Nevertheless, obtaining the $C^2$ estimates remains extremely challenging, as we cannot directly apply techniques similar to those used in the isotropic case, such as Lemma 3.3 and Lemma 3.4 in \cite{M-W-W}. In particular, the Monge-Ampère type equation we consider not only involves gradient terms but is also studied in the presence of boundary conditions. Here, we draw inspiration from the works of \cite{H-I-S} and \cite{Xia13} to construct the following auxiliary function:
$$P(\xi)=\log \sigma_{1}+e^{\be(m_{2}-\s)},$$
where $\sigma_{1}$ is the sum of all anisotropic principal curvature radii, $m_{2}$ is the upper bound of the anisotropic support function, and $\s$ is the anisotropic capillary support function reparameterized by the translated Wulff shape. The function balances the interior and boundary terms, enabling effective use of the maximum principle. Besides, Condition \eqref{condition} plays a crucial role in obtaining the $C^1$ and $C^2$ estimates. It is worth noting that we can drop this condition when deriving the $C^0$ estimate. Finally, in Section \ref{section5} and Section \ref{section6}, we employ the continuity method to complete the proof of the main theorems.

\section{Preliminary}\label{sec 2}
\subsection{The Wulff shape and dual Minkowski norm}
Let $F$ be a smooth positive function on the standard sphere $(\mathbb{S}^n, g^{\mathbb{S}^n} ,\nabla^{\mathbb{S}})$ such that the matrix
\begin{equation}\label{equ:A_F>0}
	A_{F}(x)~=~\nabla^{\mathbb{S}}\nabla^{\mathbb{S}} {F}(x)+F(x)g^{\mathbb{S}^n}, %>0 \text{ (positive definite)},
	\quad x\in \mathbb{S}^n,
\end{equation}
is positive definite on $  \mathbb{S}^n$,
where %$\nabla^{\mathbb{S}}$ denotes the covariant derivative on $\mathbb{S}^n$ and,
$g^{\mathbb{S}^n}$ denotes the round metric on $\mathbb{S}^n$.
Then there exists a unique smooth strictly convex hypersurface $\W$ defined by
\begin{align*}
	\W=\{\Psi(x)|\Psi(x):=F(x)x+\nabla^{\mathbb{S}} {F}(x),~x\in \mathbb{S}^n\},
\end{align*}
whose support function is given by $F$. We call $\W$ the Wulff shape determined by the function $F\in C^{\infty}(\mathbb{S}^n)$. If $F$ is constant, the corresponding Wulff shape coincides with a round sphere in the standard Euclidean space.
The smooth function $F$ on $\mathbb{S}^n$ can be extended  to a $1$-homogeneous function on $\mathbb{R}^{n+1}$ by
\begin{equation*}
	F(y)=|y|F({y}/{|y|}), \quad y\in \mathbb{R}^{n+1}\setminus\{0\},
\end{equation*}
and setting $F(0)=0$. Then it is easy to show that $\Psi(x)=DF(x)$ for $x\in \mathbb{S}^n$, where $D$ is the Euclidean gradient.

%In this paper we add an extra condition for $F$ to make the oblique boundary condition of flow  satisfies the non-degeneracy condition (i.e., \eqref{equ:boundary-nondegenracy}).
%\begin{align}\label{equ:condition}
%	\frac{\partial^2 F}{\partial E_{n+1}\partial E_{n+1}}(x)>0, \quad\text{ for } x\in\mathbb{S}^n\setminus\{E_{n+1},-E_{n+1}\}.
%\end{align}
%Specially, it's easy to check for $\W=\mathbb{S}^{n}$,
%$F(x)=\left(x_1^2+\cdots +x_{n+1}^2\right)^{\frac{1}{2}}$ satisfies condition \eqref{equ:condition}.

The homogeneous extension $F$ defines a Minkowski norm on $\mathbb{R}^{n+1}$, that is, $F$ is a norm on $\mathbb{R}^{n+1}$ and $D^2(F^2)$ is uniformly positive definite on $\mathbb{R}^{n+1}\setminus\{0\}$. We can define a dual Minkowski norm $F^0$ on $\mathbb{R}^{n+1}$ by
\begin{align*}%\label{s2:gam0}
	F^0(\zeta):=\sup_{y\neq 0}\frac{\langle y,\zeta\rangle}{{F}(y)},\quad \zeta\in \mathbb{R}^{n+1}.
\end{align*}
The unit Wulff shape can be interpreted by $F^{0}$ as
$$\W=\{y\in \R^{n+1}: F^0(y)=1\}.$$
A Wulff shape of radius $r_0$ centered at $y_0$ is given by
\begin{align}\label{equ:w_r(x)}
	\W_{r_0}(y_0)=\{y\in\mathbb{R}^{n+1}:F^0(y-y_0)=r_0\}.
\end{align}
An $\omega_0$-capillary Wulff shape of radius $r_0$ is given by
\begin{align}\label{equ:crwo}
	\W_{r_0,\omega_0}(E):=\{y\in\overline{\mathbb{R}_+^{n+1}}:F^0(y-r_0\omega_0E)=r_0\},
\end{align}
which is a part of a Wulff shape cut by a hyperplane  $\{x_{n+1}=0\}$, here $E$ satisfies $\<E,E_{n+1}\>=1$ which guarantees that  $\W_{r_0,\omega_0}(E)$ satisfies the capillary condition \eqref{equ:w0-capillary}. If there is no confusion, we just write $\W_{r_0,\omega_0}:=\W_{r_0,\omega_0}(E^F_{n+1})$ in this paper.

%Namely, it is an intersection of a Wulff shape and $\mathbb{R}^{n+1}_+$.
%It is easy to check that the anisotropic normal of $\W_r(x_0)$ is $\nu_F(x)=\frac{x-x_0}{r}$. On $\W_r(x_0)\cap\{x_{n+1}=0\}$, we know $\<\nu_F,-E_{n+1}\>=\<\frac{x-x_0}{r},-E_{n+1}\>=\<\frac{x_0}{r},E_{n+1}\>$, which is a constant.

\subsection{Anisotropic curvature}
Let $(\Sigma,g,\nabla)\subset \overline{\mathbb{R}^{n+1}_+}$ be a $C^2$ hypersurface with $\partial\Sigma\subset\partial\overline{\mathbb{R}^{n+1}_+}$. %, which encloses a bounded domain  $\hat{\Sigma}$. Let $\nu$ be the unit normal of $\Sigma$ pointing outward of  $\hat{\Sigma}$.
The anisotropic Gauss map of $\Sigma$  is defined by $$\begin{array}{lll}\nu_F: &&\Sigma\to  \W\\
	&&X\mapsto \Psi(\nu(X))=F(\nu(X))\nu(X)+\nabla^\mathbb{S} F(\nu(X)).\end{array} $$
%The anisotropic Weingarten map $S_F$ is the derivative of the anisotropic normal $\nu_F$.
The anisotropic principal curvatures $\k^F=(\k^F_1,\cdots, \k^F_n)$ of $\Sigma$ with respect to $\W$ at $X\in \Sigma$  are defined as the eigenvalues of
\begin{align*}%\label{equ:S_F}
	S_F=\mathrm{d}\nu_F=\mathrm{d}(\Psi\circ\nu)=A_F\circ \mathrm{d}\nu : T_X \Sigma\to T_{\nu_F(X)} \W=T_X \Sigma.
\end{align*}
The anisotropic principal curvature radii $\tau=(\tau_1,\cdots, \tau_n)$ of $\Sigma$ with respect to $\W$ at $X\in \Sigma$  are defined as $\tau_i=(\kappa^F_i)^{-1}$, if $\kappa^F_i\ne 0$, $i=1,\cdots,n$.

We define the normalized $k$-th elementary symmetric function $H^F_k=\sigma_k(\kappa^F)/\binom{n}{k}$ of the anisotropic principal curvature $\kappa^F$:
\begin{align}\label{equ:Hk}
	H^F_k:=\binom{n}{k}^{-1}\sum_{1\leq {i_1}<\cdots<{i_k}\leq n} \kappa^F_{i_1}\cdots \kappa^F_{i_k},\quad k=1,\cdots,n,
\end{align}
where $\binom{n}{k}=\frac{n!}{k!(n-k)!}$.
Setting $H^F_{0}=1$, $H^F_{n+1}=0$, and $H_F=nH_1^F$
for convenience.%When $\Sigma=\W\cap\overline{\mathbb{R}^{n+1}_+}\subset\W$, we have $\kappa^F(\Sigma)=(1,\cdots ,1)$ (see \cite{Xia2017}), and then $H^F_k=1$.

\subsection{New metric and anisotropic formulas}\label{subsec 2.3}
There are new metric on $\Sigma$ and $\mathbb{R}^{n+1}$, which were introduced by Andrews in \cite{And01} and reformulated by Xia in \cite{Xia13}.  This new Riemannian metric $G$ with respect to $F^0$ on $\mathbb{R}^{n+1}$ is defined as
\begin{align}
	\label{equ:G}
	G(y)(V,W):=\sum_{i,j=1}^{n+1}\frac{\partial^2 (\frac12(F^0)^2)(y)}{\partial y^i\partial y^j} V^\alpha W^\beta, \quad\hbox{ for } y\in \mathbb{R}^{n+1}\setminus \{0\}, V,W\in T_y{\mathbb{R}^{n+1}}.
\end{align}
The third order derivative of $F^0$ gives a $(0,3)$-tensor
\begin{equation}%\nonumber
	\label{equ:Q}
	Q(y)(U,V,W):=\sum_{i,j,k=1}^{n+1} Q_{ijk}(y)U^i V^j W^k:=\sum_{i,j,k=1}^{n+1} \frac{\partial^3(\frac12(F^0)^2)(y)}{\partial y^i \partial y^j \partial y^k}U^i V^j W^k,
\end{equation}
for $y\in \mathbb{R}^{n+1}\setminus \{0\},$ $U,V,W\in T_y{\mathbb{R}^{n+1}}.$

When we restrict the metric $G$ to $\mathcal{W}$,  the $1$-homogeneity of $F^0$ implies that
\begin{eqnarray*}
	&G(z)(z,z)=1,  \quad G(z)(z, V)=0, \quad \hbox{ for } z\in \W,\  V\in T_z \W.
	\\&Q(z)(z, V, W)=0, \quad \hbox{ for } z\in\W,\  V, W\in \mathbb{R}^{n+1}.
\end{eqnarray*}
For a smooth hypersurface $\Sigma$ in $\overline{\mathbb{R}_+^{n+1}}$, since $\nu_F(X)\in \W$ for $X\in \Sigma$, we have
%\begin{eqnarray*}
\begin{align}
	&G(\nu_F)(\nu_F,\nu_F)=1, \quad G(\nu_F)(\nu_F, V)=0, \quad \hbox{ for } V\in T_X \Sigma,\nonumber% \label{equ:G_0}
	\\
	&Q(\nu_F)(\nu_F, V, W)=0, \quad \hbox{ for } V, W\in \mathbb{R}^{n+1}.\label{equ:Q_0}
\end{align}	
%\end{eqnarray*}
%\end{comment}
This means that $\nu_F(X)$ is perpendicular to $T_X \Sigma$ with respect to the metric $G(\nu_F)$.
Then the Riemannian metric $\hat{g}$ on $\Sigma$  can be defined as
\begin{eqnarray*}
	\hat{g}(X):=G(\nu_F(X))|_{T_X \Sigma}, \quad X\in \Sigma.
\end{eqnarray*}

We denote by $\hat{D}$ and $\hat{\nabla}$ the Levi-Civita connections of $G$ on $\mathbb{R}^{n+1}$ and $\hat{g}$ on $\Sigma$ respectively,
denote by $\hat{g}_{ij}$ and $\hat{h}_{ij}$ the first and second fundamental form of $(\Sigma, \hat{g})\subset (\mathbb{R}^{n+1}, G)$, that is
$$\hat{g}_{ij}=G(\nu_F(X))(\p_i X, \p_j X),\quad \hat{h}_{ij}=G(\nu_F(X))(\hat{D}_{\p_i }\nu_F, \p_j X).$$
Denote by $\{\hat{g}^{ij}\}$ the inverse matrix of $\{\hat{g}_{ij}\}$, we can reformulate $\k^F$ as  the  eigenvalues of $\{\hat{h}_j^i\}=\{\hat{g}^{ik}\hat{h}_{kj}\}$.

The anisotropic Gauss-Weingarten type formula and the anisotropic Gauss-Codazzi type equation are as follows.
\begin{lemma}[\cite{Xia13}*{Lemma 2.5}] \label{lem2-1}
	\begin{eqnarray}%\nonumber%\label{Gauss}
		\partial_i\partial_j X=-\hat{h}_{ij}\nu_F+\hat{\nabla}_{\partial_i } \partial_j+\hat{g}^{kl}A_{ijl}\partial_kX; \;\;\;\hbox{ (Gauss formula)}\label{equ:Gauss-formula}
	\end{eqnarray}
	\begin{eqnarray}\label{Weingarten}
		\partial_i \nu_F=\hat{g}^{jk}\hat{h}_{ij}\partial_k X;\;\; \hbox{ (Weingarten formula) }
	\end{eqnarray}
	$$\hat{R}_{i j k \ell}=\hat{h}_{i k} \hat{h}_{j \ell}-\hat{h}_{i \ell} \hat{h}_{j k}+\hat{\nabla}_{\partial_{\ell}} A_{j k i}-\hat{\nabla}_{\partial_k} A_{j \ell i}+A_{j k}^m A_{m \ell i}-A_{j \ell}^m A_{m k i} ;\quad	\text{(Gauss equation)}$$
	\begin{eqnarray}\label{Codazzi}
		\hat{\nabla}_k\hat{h}_{ij}+\hat{h}_j^lA_{lki}=\hat{\nabla}_j\hat{h}_{ik}+\hat{h}_k^lA_{lji}. \;\;\hbox{ (Codazzi equation) }
	\end{eqnarray}
	Here, $\hat{R}$ is the Riemannian curvature tensor of $\hat{g}$ and  $A$ is a $3$-tensor
	\begin{eqnarray}\label{AA}
		A_{ijk}=-\frac12\left(\hat{h}_i^l Q_{jkl}+\hat{h}_j^l Q_{ilk}-\hat{h}_k^l Q_{ijl}\right),
	\end{eqnarray} where $Q_{ijk}=Q(\nu_F)(\partial_i X, \partial_j X, \partial_k X)$. Note that the $3$-tensor $A$ on $(\Sigma,\hat{g})\to (\mathbb{R}^{n+1},G)$ depends on $\hat{h}_i^j$. It is straightforward to see that $Q$ is totally symmetric in all three indices, while $A$ is only symmetric for the first two indices.
\end{lemma}
It is convenient to regard $X$ and $\nu_F$ as $\mathbb{R}^{n+1}$-valued functions in a fixed Cartesian frame; accordingly, the derivatives $\partial_i\partial_j X$ and $\partial_i \nu_F$ are taken in the classical sense of partial differentiation of vector functions. We also denote $X_i=\partial_iX$ and $X^i=\hat{g}^{ij}X_j$. For $X \in \Sigma \subset \mathbb{R}^{n+1}$, choose local coordinate
$$\{y^{i}\}_{i=1}^{n+1}\ \ in\ \ \mathbb{R}^{n+1} ,$$ such that $\pa/\pa y^{i}, i=1, \cdots ,n$ are tangent to $\Sigma$ and $\pa/\pa y^{n+1}= \nu_{F}$ is the the unit anisotropic outer normal of $\Sigma$. By identification, $\pa_{\al}=\pa/\pa y^{\al}=\pa_{\al}X$ for $\al=1, \cdots, n$. Given the standard volume from $\Omega$ (Lebesgue measure) in $\mathbb{R}^{n+1}$, the anisotropic measure on $\Sigma$ can be interpreted as
\begin{equation}\label{qqp}
d{\mu}_{F}= \Omega(\nu_{F}, \pa_{1}, \cdots, \pa_{n})dy^{1}\cdots dy^{n}=F(\nu)\Omega (\nu, \pa_{1}, \cdots, \pa_{n})dy^{1}\cdots dy^{n},
\end{equation}
see also in \cite{Xia13}.

It is direct to see that for $\Sigma=\W$,  the anisotropic normal $\nu_F$ is just the position vector and the anisotropic principal curvatures are equal to 1. Then for case that $\Sigma=\mathcal{C}_{\omega_0}\subset\mathcal{T}(\W)$, we have $\nu_F(\Sigma)=\mathcal{T}^{-1}\nu_{{F}}=X(\Sigma)-\omega_0E_{n+1}^F$. %Then we use the notations $\mathring{g}, \mathring{\nabla}$ and $\mathring{R}$ to denote the induced metric in the Wulff shape $\cc$ from $(\mathbb{R}^{n+1}, G)$, and its Levi-Civita connection and curvature tensor, respectively.
Notice that $\mathring{g}:=G(\mathcal{T}^{-1}\xi)|_{T_{\xi}\mathcal{C}_{\omega_0}}$  for $\xi\in\mathcal{C}_{\omega_0}$.
 From \cite{Xia-2017-convex}*{Proposition 2.2}, one can easily see the following lemma.
\begin{lemma}\label{lemma:Gauss-Weingarten}
	Let $X \in \mathcal{C}_{\omega_0}\subset\mathcal{T}(\W)=\W+\omega_0E_{n+1}^F$. %and $\Sigma=\cc$. % $\mathring{g}_{ij}=\hat{h}_{ij}$.
	 We also have$$
	A_{i j k}=-\frac{1}{2} Q_{i j k}=-\frac{1}{2} Q(\mathcal{T}^{-1}X)\left(\partial_i X, \partial_j X, \partial_k X\right),
	$$
	and
	\begin{align}\label{equ:Qjikl-syms}
		\mathring{\nabla}_i Q_{j k l}=\mathring{\nabla}_j Q_{i k l} .
	\end{align}
	The Gauss-Weingarten formula and the Gauss equation are as follows:
	\begin{align*}
		\partial_i \partial_j X & =-\mathring{g}_{i j} \mathcal{T}^{-1}X+\mathring{\nabla}_{\partial_i} \partial_j-\frac{1}{2} \mathring{g}^{k l} Q_{i j l} \partial_k X ; \text { (Gauss formula) } \\
		\partial_i \nu_F & =\partial_i X ; \text { (Weingarten formula) }\\
		\mathring{R}_{i j k l} & =\mathring{g}_{i k} \mathring{g}_{j l}-\mathring{g}_{i l} \mathring{g}_{j k}+\frac{1}{4} \mathring{g}^{p m} Q_{j k p} Q_{m l i}-\frac{1}{4} \mathring{g}^{p m} Q_{j l p} Q_{m k i}. \text { (Gauss equation) }
	\end{align*}
	Here $\mathring{R}$ denotes the  curvature tensor induced by $\mathring{g}$.
\end{lemma}

\subsection{Anisotropic capillary support function }
The usual support function of $\Sigma$ is defined by
$$u(X)=\< X,\nu(X) \> .\ $$
From  \cite{Xia2017}, we can define the anisotropic support function of $X \in \Sigma$ with respect to $F$ (or $\W$) as follows
$$\hat{s}(X):=G(\nu_F)(\nu_F,X).$$
It is easy to see that
\begin{equation}
	\label{u-hatu}
	\hat{s}(X)=\frac{u(X)}{F(\nu(X))}.
\end{equation}
Indeed,\begin{align}\label{equ:G(vF,Y)=<v,Y>/F}
	G(\nu_F)(\nu_F,Y)=\<DF^0(DF(\nu)),Y\>=\<\frac{\nu}{F(\nu)},Y\>=\frac{\<Y,\nu\>}{F(\nu)},
\end{align}
for any $Y\in \mathbb{R}^{n+1}$, where we used  $DF^0(DF(x))=\frac{x}{F(x)}$ (see e.g., \cite{Xia-phd}).
Instead of using the usual anisotropic Gauss map $\nu_F$, it is more convenient to use the following map
\begin{eqnarray}\label{cap Gauss map}
	\widetilde {\nu_F}:= \mathcal{T} \circ \nu_F: \Sigma \to \mathcal{C}_{\omega_0},\end{eqnarray}
which we call {\it anisotropic capillary Gauss map} of $\Sigma$.
For anisotropic $\omega_0$-capillary hypersurface $\Sigma$, we define the anisotropic capillary support function $\bar{s}$ by
\begin{align}\label{equ:u}
	\bar{s}(X)=\frac{G(\nu_F(X))(X,\nu_{F}(X))}{G(\nu_F(X))(\widetilde{\nu}_{F}(X),\nu_{F}(X))},  \quad X\in\Sigma,
\end{align}
where the quantity $\bar{s}$ is well-defined thanks to \cite[Proposition 3.2]{Jia-Wang-Xia-Zhang2023}.
Similar to \eqref{u-hatu}, the support function \eqref{equ:u} is equivalently defined by
\begin{align*}%\label{equ:u-}
	\bar{s}(X)=\frac{\hat{s}(X)}{1+\omega_0 G(\nu_F(X))(\nu_F(X),E_{n+1}^F)}=\frac{\<X, \nu(X)\>}{F(\nu(X))+\omega_0 \<\nu(X),E_{n+1}^F\>}.\
\end{align*}

%, and $|\cdot |_F$ denotes the anisotropic area defined by \eqref{equ:F-area}.

%There holds the following anisotropic Minkowski integral formula for $\omega_0$-capillary hypersurface.
%\begin{lemma}[Jia-Wang-Xia-Zhang, 2023, \cite{Jia-Wang-Xia-Zhang2023}*{Theorem 1.3}] \label{Thm1.2}
%	Assume  $(\Sigma,g) \subset \overline{\mathbb{R}_{+}^{n+1}}$ is a $C^2$ compact anisotropic $\omega_0$-capillary hypersurface, where $\omega_0 \in\left(-F\left(E_{n+1}\right)\right.,$ $\left.F\left(-E_{n+1}\right)\right)$.
% Then for  $0 \leq k \leq n-1$, it holds
%\begin{align}\label{equ:Minkow}
%	\int_{\Sigma} H_{k}^F\left(1+\omega_0G(\nu_F)( \nu_F, E_{n+1}^F)\right)-H_{k+1}^F\hat{ u} \mathrm{~d}\hat{\mu}=0.
%\end{align}

%\end{lemma}

%The boundary condition \eqref{equ:w0-capillary} implies that $\omega_0 \in(-F(E_{n+1}), F(-E_{n+1}))$ by the Cauchy-Schwarz inequality $\<x,z\>\leq F^0(x)F(z)$ (see \cite{Jia-Wang-Xia-Zhang2023}).
%This Cauchy-Schwarz inequality also implies that $\bar{ u}$ in \eqref{equ:u} is well-defined, i.e.
%\begin{proposition}[\cite{Jia-Wang-Xia-Zhang2023}*{Prop 3.2}]\label{prop2.1}
%	For $\omega_0 \in(-F(E_{n+1}), %%F(-E_{n+1}))$, it holds that
%	\begin{align*}
	%		F(z)+\omega_0\<z,E_{n+1}^F\>>0,\quad\text{for any}\ z\in \mathbb{S}^n.
	%	\end{align*}
%\end{proposition}
%By Proposition \ref{prop2.1}, we know  definition of $\bar{u }$ in \eqref{equ:u} is meaningful.

\subsection{General quermassintegrals for anisotropic capillary hypersurfaces}

For the given  Wulff shape $\W\subset\mathbb{R}^{n+1}$ determined by  $F$, and for the constant $\omega_0\in(-F(E_{n+1}),F(-E_{n+1}))$, we denote
$\overline{\W}=\partial \mathcal{C}_{\omega_0},$
as an $(n-1)$-dimensional Wulff shape in $\mathbb{R}^n$  with support function $\bar{F}$.%, and $-\omega_0E_{n+1}^F$ as the new origin in $\mathbb{R}^n=\{x_{n+1}=-\omega_0\}$.

 Recently, the following geometric functions were introduced by Ding, Gao and Li \cite{Arxiv1}, which can be considered as the suitable quermassintegrals for anisotropic capillary hypersurfaces in the half-space  $\overline{\mathbb{R}_+^{n+1}}$ with contact constant $\omega_{0}$:
 \begin{align*}
\mathcal{V}_{0,\omega_0}(\Sigma) &:= |\widehat{\Sigma}|, \\
\mathcal{V}_{1,\omega_0}(\Sigma) &:= \frac{1}{n+1}\left(|\Sigma|_F+\omega_0|\widehat{\partial\Sigma}|\right),\\
\mathcal{V}_{k+1,\omega_0}(\Sigma)&:= \frac{1}{n+1}\left(\int_{\Sigma}H^F_kF(\nu)d\mu_g+\frac{\omega_0}{n}\int_{\partial\Sigma} H_{k-1}^{\bar{F}}\bar{F}(\bar{\nu})ds\right),\ 1\leq k\leq n.
\end{align*}	
Here $\bar{\nu}(X)$ is also the  unit co-normal of $\overline{\W}\subset \mathbb{R}^n$, and $H_k^{\bar{F}}$ (with $0\leq k\leq n-1$ and $H_0^{\bar{F}}=1$) denotes the normalized anisotropic $k$-th mean curvature of $\partial\Sigma\subset\mathbb{R}^n$ relative to the Wulff shape $\overline{\mathcal{W}}$.
In particular, one has
\begin{align}\label{equ:Vk=Vkkklll}
	\mathcal{V}_{k+1,\omega_0}(\Sigma)=\frac{1}{n+1}\int_{\Sigma}H_k^F\(1+\omega_0G(\nu_{{F}})(\nu_{{F}},E_{n+1}^F)\) {\rm d}\mu_F, \;\;k=0,\cdots, n,
\end{align}	
and $\mathcal{V}_{k,\omega_0}(\mathcal{C}_{\omega_0})=\Vol(\mathcal{C}_{\omega_0}):=|\mathcal{C}_{\omega_0}|$ for $k=0,\cdots,n+1$, see \cite{Arxiv1}*{Lemma 7.6}.

%We remark that, if $\W=\mathbb{S}^n$, then $F(\mathbb{S}^n)=1,H_k^F=H_k,\nu_{{F}}=\Psi(\nu)=\nu,E_{n+1}^F=E_{n+1},\cos\theta=-\omega_0,\<\nu,\bar{\nu}\>=\sin\theta,\bar{F}(\mathbb{S}^{n-1})=\sin\theta, \overline{\W}=\sin\theta\cdot\mathbb{S}^{n-1},S_{\bar{F}}=\sin\theta\bar{h}$,   and $ H^{\bar{F}}_{k-1}=\sin^{k-1}\theta H^{\partial\Sigma}_{k-1}$, where $H_k$ is normalized $k$-mean curvature of $(\Sigma,g)\subset(\mathbb{R}^{n+1}, \<\cdot,\cdot\>)$, $\bar{h}$ and $H^{\partial\Sigma}_{k-1}$ are the Weingarten map and  normalized $k$-mean curvature of $\partial\Sigma\subset\mathbb{R}^n$, and $\theta$ is the contact angle. Thus we have
%\begin{align*}
	%\mathcal{V}_{k+1,\omega_0}(\Sigma)=
	%\frac{1}{n+1}
	%\left(
	%\int_{\Sigma}H_kd\mu_g
	%-\frac{\cos\theta\sin^k\theta}{n}
	%\int_{\partial\Sigma} H_{k-1}^{\partial\Sigma}ds
	%\right), \quad k=1,\cdots,n,
%\end{align*}
%which matches the definition of  isotropic quermassintegrals \cite{Wang-Weng-Xia}*{eq.(1.8)} for  capillary hypersurfaces in the half-space.

The quermassintegrals defined above are well-defined, as they satisfy the following variational formulas:

\begin{proposition}[\cite{Arxiv1}*{Theorem 6.2}]\label{Thm:p_t-Vk}
	Let $\Sigma_t \subset \overline{\mathbb{R}_{+}^{n+1}}$ be a family of smooth anisotropic $\omega_0$-capillary hypersurfaces supported by $\partial \overline{\mathbb{R}_{+}^{n+1}}$ with a constraint condition $\omega_0 \in\left(-F(E_{n+1}),F(-E_{n+1})\right)$, given by the embedding $X(\cdot, t): \Sigma \rightarrow \overline{\mathbb{R}_{+}^{n+1}}$ satisfying
	\begin{align}\label{equ:p_tX=fv}
		\partial_t X=f \nu_F+T,
	\end{align}
	for a smooth function $f$ and vector field $T\in T\Sigma_t$. Then for $-1 \leq k \leq n$,
	\begin{align}
		\frac{d}{d t} \mathcal{V}_{k+1, \omega_0}\left({\Sigma_t}\right)=\frac{n-k}{n+1} \int_{\Sigma_t} f H^F_{k+1} d \mu_F.
		\label{equ:Thm5}
	\end{align}
\end{proposition}

%There exists a fundamental relationship between quermassintegrals and mixed volumes in convex geometry, see \cite{arxiv2}*{Lemma 3.11}:
For the combination of anisotropic capillary convex bodies $K:=\sum_{i=1}^m \lambda_i K_i$  where $K_1, \cdots, K_m \subset {\mathbb{R}^{n+1}}$, $m\in \mathbb{N}$, $\lambda_{1}, \cdots,\lambda_{m}\geq 0$, the mixed volume $V\left(K_{i_0}, \cdots, K_{i_{n}}\right)$ is defined by
\begin{align*}
	|K|=\sum_{i_0, \cdots, i_{n}=1}^m \lambda_{i_0} \cdots \lambda_{i_{n}} V\left(K_{i_0}, \cdots, K_{i_{n}}\right).
\end{align*}
A fundamental relationship between quermassintegrals $\mathcal{V}_{k,\omega_{0}}$ and mixed volumes in convex geometry is given by the following lemma:
\begin{lemma}[\cite{arxiv2}*{Lemma 3.11}]\label{lemma:V-k=V(KKK,LLL)}
	Let ${\Sigma}$ be an anisotropic capillary hypersurface in $\overline{\mathbb{R}^{n+1}_+}$. For $-1\leq k\leq n$, we have
	\begin{align*}
		\mathcal{V}_{k+1,\omega_0}(\Sigma)	=V\(\underbrace{\widehat{\Sigma}, \cdots, \widehat{\Sigma}}_{(n-k) \text { copies }}, \underbrace{\widehat{\mathcal{C}}_{\omega_0}, \cdots, \widehat{\mathcal{C}}_{\omega_0}}_{(k+1) \text { copies }}\).
	\end{align*}
\end{lemma}
%For mixed volumes in convex geometry, we have proposition:
A key proposition concerning mixed volumes in convex geometry is given as follows:
\begin{proposition}[\cite{book-convex-body}*{Theorem 7.2.3}]\label{prop:convex bodies V1}
	 Let $K, L \subset \mathbb{R}^{n+1}$ be convex bodies, and let $\mathcal{A} \subset \mathbb{S}^{n}$ be a closed set such that
$$
K=\bigcap_{u \in \mathcal{A}} H^{-}(K, u) .
$$
If $\bar{L}$ is defined by
$$
\bar{L}:=\bigcap_{u \in \mathcal{A}} H^{-}(L, u),
$$
then
\begin{align}\label{equ:convex-bodies-V1}
	V_1(K, L)^{\frac{n+1}{n}}-\Vol(K) \Vol(\bar{L})^{\frac{1}{n}} \geq\left[V_1(K, L)^{\frac{1}{n}}-r(K, L) \Vol(\bar{L})^{\frac{1}{n}}\right]^{n+1},
\end{align}
where $r(K, L)$ denotes the inradius of $K$ relative to $L$, $V_1(K,L)=V\(\underbrace{K, \cdots, K}_{n \text { copies }}, L\)$, $H^-(K,u):=\left\{x\in \mathbb{R}^{n+1}:\<x,u\>\leq h(K,u)\right\}$, and $h(K,u):=\sup \left\{ \<x,u\>:x\in K
\right\}$ for $u\in\mathbb{R}^{n+1}$.
\end{proposition}
For the strictly  convex anisotropic $\omega_0$-capillary  hypersurface $\Sigma\subset\overline{\mathbb{R}^{n+1}_+}\subset \mathbb{R}^{n+1}$,  we take $\mathcal{A}=\left\{\nu(X):X\in \Sigma\right\}\bigcup\left\{-E_{n+1}\right\}$, then from Lemma \ref{Lemma:gaussmap}, we can check that
$$
\widehat{\Sigma}=\bigcap_{u \in \mathcal{A}} H^{-}(\widehat{\Sigma}, u),\quad \widehat{\cc}=\bigcap_{u \in \mathcal{A}} H^{-}(\widehat{\cc}, u) .
$$
Putting together the above facts, Lemma \ref{lemma:V-k=V(KKK,LLL)}, and the definition of $r(\Sigma)$ from \eqref{equ:inner radius}, the inequality \eqref{equ:convex-bodies-V1} takes the form:
%Combining with above and  Lemma \ref{lemma:V-k=V(KKK,LLL)} and  the definite of  anisotropic inner radius $r(\Sigma)$ of $\Sigma$ relative to $\mathcal{C}_{\omega_0}$ in \eqref{equ:inner radius}, the inequality \eqref{equ:convex-bodies-V1} can become:
\begin{align}\label{equ:capillary-V1-Vol}
	\mathcal{V}_{1,\omega_{0}}(\Sigma)^{\frac{n+1}{n}}-\Vol(\Sigma) \Vol(\cc)^{\frac{1}{n}} \geq\left[\mathcal{V}_{1,\omega_{0}}(\Sigma)^{\frac{1}{n}}-r(\Sigma) \Vol(\cc)^{\frac{1}{n}}\right]^{n+1}.
\end{align}
%Following a proof similar to that of Proposition 5.1 in \cite{And01} and employing equation \eqref{equ:capillary-V1-Vol}, we obtain following Proposition
Combining a proof analogous to \cite[Proposition 5.1]{And01} with inequality \eqref{equ:capillary-V1-Vol}, we obtain the following proposition.
\begin{proposition}
	\label{prop:Vol<r}
	For a strictly convex  anisotropic  $\omega_{0}$-capillary hypersurface $\Sigma$, we have
	\begin{align}
		\label{equ:Vol<r}
		\frac{\Vol(\Sigma)}{\mathcal{V}_{1,\omega_{0}}(\Sigma)}\leq r(\Sigma).
	\end{align}
	Here $r(\Sigma)$ is  defined by \eqref{equ:inner radius}.
\end{proposition}
\section{Anisotropic $\omega_0$-capillary convex bodies in $\overline{\mathbb{R}^{n+1}_+}$}\label{section3}
In this section, we define the anisotropic capillary $p$-sum and the $k$-th anisotropic capillary $p$-surface area measure, and propose two prescribed measure problems including the anisotropic capillary $L_p$-Minkowski problem.% following Li and Xu \cite{Li-Xu}. %Additionally, we derive several inequalities for later use, drawing from \cite{Xia-arxiv,book-convex-body,Shenfeld19,Shenfeld22}.

\subsection{Anisotropic capillary convex bodies and convex functions}
\begin{definition}
	Given a Wulff shape $\W$ (or a Minkowski norm $F$). For $\omega_0\in\left(-F(E_{n+1}),F(-E_{n+1})\right)$, we call $\hat{\Sigma}$ an anisotropic $\omega_0$-capillary convex body if $\hat{\Sigma}$ is a convex body (a compact convex set with non-empty interior) in $\mathbb{R}^{n+1}$ and $\Sigma$ is an anisotropic $\omega_0$-capillary hypersurface in $\overline{\mathbb{R}_+^{n+1}}$.  The class of anisotropic $\omega_0$-capillary convex bodies in $\overline{\mathbb{R}_{+}^{n+1}}$ is denoted by $\mathcal{K}_{\omega_0}$. For $\hat{\Sigma}\in\mathcal{K}_{\omega_0}$, if  $\widehat{\partial\Sigma}\subset \partial\overline{\mathbb{R}_{+}^{n+1}}$ contains the origin $O$, we denote $\hat{\Sigma}\in \mathcal{K}^O_{\omega_0}$.
	
\end{definition}

The simplest capillary convex body is given by $\widehat{\mathcal{C}_{\omega_0}}\in \K_{\omega_0}$.  For any $\widehat{\Sigma}\in \mathcal {K}_{\omega_0}$ we give a parametrization of $\Sigma$   via the anisotropic capillary Gauss map $\widetilde {\nu_F}$.
\begin{lemma}[\cite{arxiv2}, Lemma 3.1]\label{Lemma:gaussmap}
	Given a Wulff shape $\W$ (or a Minkowski norm $F$), and a constant  $\omega_0\in\left(-F(E_{n+1}),F(-E_{n+1})\right)$. Denote  by $\mathcal{A}:=\{x\in\mathbb{S}^n:\<DF(x),E_{n+1}\>\geq -\omega_0\}$.  %is a connecte domain in $			\mathbb{S}^{n}$.
	For any  $\widehat{\Sigma}\in\mathcal{K}_{\omega_0}$, the following statements hold.
	\begin{itemize}
		\item [(i)]  The Gauss map $\nu$ of $\Sigma$ has its spherical image $\mathcal{A}\subset\mathbb{S}^n$, and $\nu:\Sigma \to \mathcal{A}$ is a diffeomorphism.
		
		\item [(ii)] The aniosptrpic capillary Gauss map $\widetilde{\nu_{{F}}}=\nu_{{F}}+\omega_0E_{n+1}^F$ of $\Sigma$ has its image in $\cc$, %$\mathcal{C}_{\omega_0}=\(\W+\omega_0E_{n+1}^F\)\bigcap\overline{\mathbb{R}^{n+1}_+}$,
		and  $\widetilde {\nu_F}: \Sigma\to \mathcal{C}_{\omega_0}$  %has its image in
		%\begin{eqnarray*}
		%	\mathbb{S}^n_\theta :=\{ y\in \mathbb{S}^n \,|\,  y_{n+1} \geq \cos \theta\},
		%\end{eqnarray*}
		%	and $\nu: \Sigma\to \mathbb{S}^n_\theta$
		is a diffeomorphism.
	\end{itemize}
\end{lemma}

Lemma \ref{Lemma:gaussmap} implies that we can parametrize the strictly convex anisotropic capillary hypersurface $\Sigma$ by the inverse anisotropic capillary Gauss map, i.e., $X:\mathcal{C}_{\omega_0}\to \Sigma$,  given by
\begin{eqnarray*}
	X(\xi)= \widetilde {\nu_F} ^{-1} (\xi) =\nu_F^{-1}\circ \mathcal{T}^{-1}(\xi)=\nu_F^{-1}(\xi-\omega_0E_{n+1}^F).
\end{eqnarray*}

Now we view \eqref{equ:u} and \eqref{u-hatu} as function defined in $\cc$ by our parametrization, namely

\begin{align}\label{support}
	\hat{s}(\xi): =G(\mathcal{T}^{-1}\xi)\(X(\xi),\mathcal{T}^{-1}\xi\)=G(\mathcal{T}^{-1}\xi)\(\widetilde{\nu_{{F}}}^{-1}(\xi),\xi-\omega_0E_{n+1}^F\),%\nonumber
	%\\=&\sup_{x\in\widehat{\Sigma}}G(\mathcal{T}^{-1}\xi)\(\mathcal{T}^{-1}\xi,x\).
\end{align}
and $$\bar{s}(\xi):=\frac{\hat{s}(\xi)}{1+\omega_0G(\mathcal{T}^{-1}\xi)\(E_{n+1}^F,\mathcal{T}^{-1}\xi\)},\quad  \text{for}\  \xi\in \mathcal{C}_{\omega_0},$$
 which has the property $\mathring{\nabla}_{\mu_F} \bar{s}=0$ on $\partial\mathcal{C}_{\omega_0}$, see \cite[Lemma 3.4]{Arxiv1}.

{To distinguish with the anisotropic support function of the convex body $\widehat{\Sigma}$, we call $\hat{s}$ in \eqref{support} {\it anisotropic capillary support function} of $\Sigma$ (or of $\widehat{\Sigma}$).  It is clear that  the anisotropic capillary Gauss map for $ \mathcal{C}_{\omega_0}$  is the identity map
	from $\mathcal{C}_{\omega_0} \to \mathcal{C}_{\omega_0}$ and	\begin{align*}
		\hat{s}\vert_{\cc}(\xi)=&G(\xi-\omega_0E_{n+1}^F)\(\xi  , \xi -\omega_0E_{n+1}^F\)
		\\
		=&
		G(\xi-\omega_0E_{n+1}^F)\(\xi-\omega_0E_{n+1}^F+\omega_0E_{n+1}^F  , \xi -\omega_0E_{n+1}^F\)
		\\
		=&1+\omega_0G(\xi-\omega_0E_{n+1}^F)\(E_{n+1}^F  , \xi -\omega_0E_{n+1}^F\).
	\end{align*}
	For simplicity, we denote
	\begin{eqnarray}\label{ell}
		\ell(\xi):=1+\omega_0G(\xi-\omega_0E_{n+1}^F)\(E_{n+1}^F  , \xi -\omega_0E_{n+1}^F\).
	\end{eqnarray} Therefore $\ell$ is the anisotropic capillary support function of $\mathcal{C}_{\omega_0}$ (or of $\widehat{\mathcal{C}_{\omega_0}}$) in the latter context.
	
	By computing the geometric quantities of $X$ in terms of $\hat{s}$, we have the following lemma.
	\begin{lemma}[\cite{arxiv2}*{Lemma 3.2}]\label{lemma:Xia-Lemma2.4}
		For the parametrization $X:\left(\mathcal{C}_{\omega_0},\mathring{g},\mathring{\nabla }\right)\rightarrow\Sigma$ of anisotropic $\omega_0$-capillary convex body $\hat{\Sigma}$, we have
		\begin{itemize}
			\item[(1)] $X(\xi)=\mathring{\nabla} \hat{s}(\xi) +\hat{s}(\xi)\mathcal{T}^{-1}\xi$.
			
			\item[(2)]  $\mathring{\nabla}_{\mu_F}\hat{s}=\frac{\omega_0}{F(\nu)\<\mu,E_{n+1}\>}\hat{s}$ along $\partial\mathcal{C}_{\omega_0}$. %, where $\mu_F$ is anisotropic conormal to $\pa \cc \subset \cc$ . 				
			\item[(3)] The anisotropic principal radii of $\Sigma$ at $X(\xi)$ are given by the eigenvalues of $\tau_{ij}(\xi):=\mathring{\nabla}_{e_i}\mathring{\nabla}_{e_j}\hat{s}(\xi)+\mathring{g}_{i j} \hat{s}(\xi)-\frac{1}{2} Q_{i j k} \mathring{\nabla}_{e_k} \hat{s}(\xi)$  with respect to the metric $\mathring{g}$. In particular, $(\tau_{ij})>0$ on $\mathcal{C}_{\omega_0}$.
		\end{itemize}	
	\end{lemma}
\begin{remark}
Along the boundary of $(\cc,\mathring{g})$, we can choose an orthonormal frame $\{e_{i}\}_{i=1}^{n}$ with
\begin{equation}\label{pwf1}
e_{n} = \frac{\mu_{F}}{|\mu_{F}|_{\mathring{g}}} = \frac{F(\nu)^{\frac{1}{2}} A_{F}(\nu)\mu}{\langle A_{F}(\nu)\mu, \mu\rangle^{\frac{1}{2}}}.
\end{equation}
For details, we refer to \cite{arxiv2}*{Lemma 3.2} and \cite{Arxiv1}*{Remark 3.1}.
\end{remark}

	\begin{definition}\label{defn-convex capillary fun}
		For $\omega_0\in (-F(E_{n+1}),F(-E_{n+1}))$, a function $f\in  C^{2}(\mathcal{C}_{\omega_0})$ is called an \emph{anisotropic capillary function} if it satisfies the Robin-type boundary condition
		\begin{eqnarray}
			\label{robin}
			\mathring{\nabla}_{\mu_F} f=\frac{\omega_0}{F(\nu)\langle\mu,E_{n+1}\rangle} f\quad \text{on}\quad \partial \mathcal{C}_{\omega_0}.
		\end{eqnarray}	
		If in addition, the tensor
		$$\tau_{ij}[f](\xi):=\mathring{\nabla}_{e_i}\mathring{\nabla}_{e_j}f(\xi)+\delta_{i j} f(\xi)-\frac{1}{2} Q_{i j k} \mathring{\nabla}_{e_k} f(\xi),$$
		is positive definite on $\mathcal{C}_{\omega_0}$, then $f$ is called an \emph{anisotropic capillary convex function} on $\mathcal{C}_{\omega_0}$.
	\end{definition}
	%where $\nu$ is the Gauss map of $\Sigma$.
	The next proposition establishes a one-to-one correspondence between anisotropic capillary convex functions on $\mathcal{C}_{\omega_0}$ and anisotropic $\omega_0$-capillary convex bodies in $\overline{\mathbb{R}^{n+1}_+}$. This extends the classical correspondence between convex bodies in $\mathbb{R}^{n+1}$ and convex functions on $\mathbb{S}^n$ (or a Wulff shape $\mathcal{W}$) as discussed in \cite{Xia13}. 	
	\begin{proposition}[\cite{arxiv2}*{Proposition 3.4}]\label{equ pro}
		Let $\hat{s}\in C^{2}(\mathcal{C}_{\omega_0})$. Then $\hat{s}$ is an anisotropic capillary convex function if and only if $\hat{s}$ is the anisotropic capillary support function of an anisotropic capillary convex body $\widehat{\Sigma}\in \mathcal{K}_{\omega_0}$. %hypersurface $\Sigma\subset \overline{\RR^{n+1}_+}$.
	\end{proposition}

\subsection{Anisotropic capillary $p$-sum for capillary convex bodies} Inspired by \cite[Section 9.1]{book-convex-body}, \cite{Li-Xu}, and \cite{Xia-arxiv}, this subsection defines the anisotropic capillary $p$-sum for both convex functions and convex bodies, and establishes the relationship (Proposition \ref{prop:compare}) between Definition \ref{def:p-sum-fcn} and Definition \ref{def:p-sum-cb}.	
	%Before define the anisotropic capillary $p$-sum, we  need some
	
	\begin{lemma}
		\label{lemma:p-sum}
		Let  $a,b\geq 0$ (not both zero), and $p\geq 1$ be real numbers.  Given two anisotropic capillary convex function $f_1,f_2\in C^{\infty}(\mathcal{C}_{\omega_0})$, consider the $p$-linear combination of $f_1$ and $f_2$  defined by
		\begin{align}\label{equ:fp=f1p+f2p}
			f=\left(a\cdot f_1^p+b\cdot f_2^p\right)^{\frac{1}{p}}.
		\end{align}
	Then $f\in C^{\infty}(\mathcal{C}_{\omega_0})$  is also an anisotropic capillary convex function provided that $p>1$ and $f_1, f_2>0$, or
	$p=1$.
	\end{lemma}
	\begin{proof}
		Since $f_1,f_2$ are two anisotropic capillary convex function, we have
		\begin{align}
			\mathring{\nabla}_{\mu_F} f_{k}=\frac{\omega_0}{F(\nu)\<\mu,E_{n+1}\>} f_{k}\quad \text{on}\quad \partial \mathcal{C}_{\omega_0},\quad k=1,2, \label{equ:Dnf}
		\end{align}
		and
		\begin{align}
			\(\tau_{ij}[f_{k}]\)>0,\quad \text{in}\quad  \mathcal{C}_{\omega_0},\quad k=1,2.\label{tua[f]}
		\end{align}
		By \eqref{equ:fp=f1p+f2p}, we have
		\begin{align}
			\label{equ:Df}
			\mathring{\nabla}f=&\frac{1}{p}\(a f_1^p+b f_2^p\)^{\frac{1}{p}-1}\cdot
			\(ap\cdot f_1^{p-1}\mathring{\nabla}f_1 +bp\cdot f_2^{p-1}\mathring{\nabla}f_2\)\nonumber
			\\
			=&f^{1-p}\(af_1^{p-1}\mathring{\nabla}f_1 +bf_2^{p-1}\mathring{\nabla}f_2\),
		\end{align}
		and
		\begin{align}
			\mathring{\nabla}_{ij}f=&(1-p)f^{-p}\cdot f^{1-p}\(af_1^{p-1}\mathring{\nabla}_if_1 +bf_2^{p-1}\mathring{\nabla}_if_2\)\cdot \(af_1^{p-1}\mathring{\nabla}_jf_1 +bf_2^{p-1}\mathring{\nabla}_jf_2\)\nonumber
			\\
			&+f^{1-p}\(a(p-1)\mathring{\nabla}_if_1\mathring{\nabla}_jf_1\cdot f_1^{p-2}+b(p-1)\mathring{\nabla}_if_2\mathring{\nabla}_jf_2\cdot f_2^{p-2}\)\nonumber
			\\
			&+f^{1-p}\(af_1^{p-1}\mathring{\nabla}_{ij}f_1+bf_2^{p-1}\mathring{\nabla}_{ij}f_2\)\nonumber
			\\
			:=&(I)_{ij}+(II)_{ij},\label{equ:DDf}
		\end{align}
		where we denote
		\begin{align*}
			(I)_{ij}=&(1-p)f^{-p}\cdot f^{1-p}\(af_1^{p-1}\mathring{\nabla}_if_1 +bf_2^{p-1}\mathring{\nabla}_if_2\)\cdot \(af_1^{p-1}\mathring{\nabla}_jf_1 +bf_2^{p-1}\mathring{\nabla}_jf_2\)\nonumber
			\\
			&+f^{1-p}\(a(p-1)\mathring{\nabla}_if_1\mathring{\nabla}_jf_1\cdot f_1^{p-2}+b(p-1)\mathring{\nabla}_if_2\mathring{\nabla}_jf_2\cdot f_2^{p-2}\)\nonumber
			,
		\end{align*}
		and
		\begin{align*}
			(II)_{ij}=	f^{1-p}\(af_1^{p-1}\mathring{\nabla}_{ij}f_1+bf_2^{p-1}\mathring{\nabla}_{ij}f_2\).
		\end{align*}
		We can check that $(I)_{ij}\geq 0$, i.e., $(I)$ is a semi-positive definite matrix, since for any  non-zero vector $Y=(Y^1,\cdots,Y^{n})\in\mathbb{R}^{n}$, we have
		\begin{align}
			\sum_{i,j=1}^{n}(I)_{ij}Y^iY^j=&(p-1)f^{1-2p}\sum_{i,j=1}^{n}\(af^pf_1^{p-2}\mathring{\nabla}_if_1\mathring{\nabla}_jf_1+bf^pf_2^{p-2}\mathring{\nabla}_if_2\mathring{\nabla}_jf_2\nonumber				\\
			&-(af_1^{p-1}\mathring{\nabla}_if_1+bf_2^{p-1}\mathring{\nabla}_if_2)\cdot (af_1^{p-1}\mathring{\nabla}_jf_1+bf_2^{p-1}\mathring{\nabla}_jf_2)\)Y^iY^j\nonumber
			\\
			=&(p-1)f^{1-2p}\(a(f^pf_1^{p-2}-af_1^{2p-2})\cdot (\sum_{i=1}^{n}\mathring{\nabla}_if_1Y^i)^2\nonumber
			\\
			&+b(f^pf_2^{p-2}-bf_2^{2p-2})\cdot (\sum_{i=1}^{n}\mathring{\nabla}_if_2Y^i)^2\nonumber
			\\
			&-2abf_1^{p-1}f_2^{p-1}(\sum_{i=1}^{n}\mathring{\nabla}_if_1Y^i)(\sum_{i=1}^{n}\mathring{\nabla}_if_2Y^i)\)\nonumber
			\\
			\overset{\eqref{equ:fp=f1p+f2p}}{=}&ab(p-1)f^{1-2p}f_1^{p-2}f_2^{p-2}\(f_2\sum_{i=1}^{n}\mathring{\nabla}_if_1Y^i-f_1\sum_{i=1}^{n}\mathring{\nabla}_if_2Y^i\)^2\geq 0. \label{equ:I>0}
		\end{align}
		On the other hand, by \eqref{equ:Df} and \eqref{equ:fp=f1p+f2p}, we have
		\begin{align}
	&(II)_{ij}-\frac{1}{2}Q_{ijk}\mathring{\nabla}_kf+f\delta_{ij}\nonumber
			\\
=&f^{1-p}\(af_1^{p-1}\mathring{\nabla}_{ij}f_1\!-\!\frac{a}{2}f_1^{p-1}Q_{ijk}\mathring{\nabla}_kf_1\!+\!af_1^p\delta_{ij}
			\!+\!
			bf_2^{p-1}\mathring{\nabla}_{ij}f_2\!-\!\frac{b}{2}f_2^{p-1}Q_{ijk}\mathring{\nabla}_kf_2\!+\!bf_2^p\delta_{ij}
			\)\nonumber
			\\
			=&f^{1-p}\(af_1^{p-1}\tau_{ij}[f_1]+bf_2^{p-1}\tau_{ij}[f_2]\)
			>0,\label{equ:(II)+>0}
		\end{align}
		where we use \eqref{tua[f]},  and  $a, b\ge0$ and are not both zero.
		By \eqref{equ:DDf}$\sim$\eqref{equ:(II)+>0}, we have
		\begin{align*}
			\(\tau_{ij}[f]\)=\(\mathring{\nabla}_{ij}f-\frac{1}{2}Q_{ijk}\mathring{\nabla}_kf+f\delta_{ij}\)>0.
		\end{align*}
		Besides,	from \eqref{equ:Dnf}, \eqref{equ:Df}, and \eqref{equ:fp=f1p+f2p}, we obtain
		\begin{align*}
			\mathring{\nabla}_{\mu_F} f=f^{1-p}\(af_1^{p-1}\mathring{\nabla}_{\mu_F}f_1 +bf_2^{p-1}\mathring{\nabla}_{\mu_F}f_2\)=\frac{\omega_0}{F(\nu)\<\mu,E_{n+1}\>} f\quad \text{on}\quad \partial \mathcal{C}_{\omega_0}.
		\end{align*}
		Then we complete the proof.
	\end{proof}
	We denote by $\hat{s}_K(\xi)$ ($\xi\in\cc$) the anisotropic capillary support function of  $K\in\mathcal{K}_{\omega_{0}}$. By Lemma \ref{lemma:Xia-Lemma2.4}, this is an anisotropic capillary convex function. Moreover, Lemma \ref{lemma:p-sum} implies that the $p$-linear combination of such support functions remains an anisotropic capillary convex function (Definition \ref{defn-convex capillary fun}). Therefore, such a function determines an anisotropic capillary convex body. Consequently, the following definition is well-posed.
	
		\begin{definition}
		\label{def:p-sum-fcn}
		(1) For $p=1$, $K,L\in\mathcal{K}_{\omega_0}$, and $a,b\geq 0$ (not both zero), the {anisotropic Minkowski linear combination} $a\cdot K+b\cdot L\in \mathcal{K}_{\omega_0}$ is defined by
		\begin{align}\label{equ:s=sK+sL}
			\hat{s}_{a\cdot K+ b\cdot L}(\xi)=a\cdot  \hat{s}_K(\xi)+b\cdot  \hat{s}_L(\xi).
		\end{align}
		
		(2) For $p>1$,  $K,L\in\mathcal{K}^O_{\omega_0}$, and $a,b\geq 0$ (not both zero), the {anisotropic $L_p$ Minkowski linear combination} $a\cdot K+_p b\cdot L\in \mathcal{K}_{\omega_0}^O$ is defined by 		
		%Let $p\geq 1$, $a,b>0$ be real numbers.  Given two anisotropic capillary convex bodies $K,L\in\mathcal{K}_{\omega_0}^O$ with anisotropic capillary support function $\hat{s}_K(\xi),\hat{s}_L(\xi)\in C^{\infty}(\mathcal{C}_{\omega_0})$, considering $p$-linear combination of $f_1$ and $f_2$  which is given by
		\begin{align}\label{equ:sp=sKp+sLp}
			\hat{s}_{a\cdot K+_p b\cdot L}(\xi)=\left(a\cdot \hat{s}_K^p(\xi)+b\cdot \hat{s}_L^p(\xi)\right)^{\frac{1}{p}}.
		\end{align}
	\end{definition}

	In line with the approach in \cite{Li-Xu} and \cite{Lutwark2012}, we define the following pointwise addition for nonconvex set in $\overline{\mathbb{R}^{n+1}_+}$.
	\begin{definition}
		\label{def:p-sum-cb}
		For $p\geq 1, a,b\geq 0$ and $K,L\subset\overline{\mathbb{R}^{n+1}_+}$, the sum $a\cdot K\widetilde{+}_p b\cdot L\subset\overline{\mathbb{R}_{+}^{n+1}}$ is  defined by
		\begin{align*}
			a\cdot K\widetilde{+}_p b\cdot L=\left\{
			a^{\frac{1}{p}}(1-t)^{\frac{1}{q}}X_{1}+
			b^{\frac{1}{p}}t^{\frac{1}{q}}X_{2}:X_{1}\in K, X_{2}\in L, t\in[0,1]
			\right\},
		\end{align*}
		where $q$ is the H\"{o}lder conjugate of $p$, i.e., $\frac{1}{p}+\frac{1}{q}=1$.
	\end{definition}

	We can check that Definition \ref{def:p-sum-cb} is compatible with Definition \ref{def:p-sum-fcn}:
	\begin{proposition}
		\label{prop:compare}
		Let $a,b\ge 0$, and $K,L\in\mathcal{K}_{\omega_0}$ be smooth. If  $p=1$, then
		\begin{align*}
			a \cdot K\widetilde{+} b\cdot  L=a\cdot  K+ b\cdot L.
		\end{align*}
		If $p>1$ and $K,L\in \mathcal{K}^O_{\omega_0}$, then
		\begin{align*}
			a\cdot K \widetilde{+}_p b\cdot L=a\cdot K+_p b\cdot L.
		\end{align*}
	\end{proposition}
	\begin{proof}
		As the proof closely follows \cite{Li-Xu}, it suffices to check that the boundary condition holds for $a \cdot K \widetilde{+}_p b \cdot L$.
		
		By Lemma \ref{Lemma:gaussmap}, the anisotropic capillary  Gauss maps $\widetilde{\nu_F}_i$ of $\omega_0$-capillary hypersurfaces $\Sigma_i\ (i=1,2)$ are diffeomorphisms and $\widetilde{\nu_F}_i(\partial\Sigma_i)=\partial\mathcal{C}_{\omega_0}$ for $i=1,2$, which implies that, for any $\xi\in\partial\mathcal{C}_{\omega_0}$, there exists $X_{1}\in \partial\Sigma_1$ and $X_{2}\in\partial\Sigma_2$, such that $\widetilde{\nu_F}_1(X_{1})=\xi$ and $\widetilde{\nu_F}_2(X_{2})=\xi$, i.e., ${\nu_F}_1(X_{1})=\mathcal{T}^{-1}\xi, {\nu_F}_2(X_{2})=\mathcal{T}^{-1}\xi$. Then by \eqref{support}
		\begin{align*}
			\hat{s}_K(\xi):&=G(\hat{\xi})\(\hat{\xi},X_{1}\)=\sup_{Z_{1}\in\Sigma_1}G(\mathcal{T}^{-1}\xi)\(\mathcal{T}^{-1}\xi,Z_{1}\)=\sup_{Z_{1}\in K}G(\mathcal{T}^{-1}\xi)\(\mathcal{T}^{-1}\xi,Z_{1}\),\\
			\hat{s}_L(\xi):&=G(\hat{\xi})\(\hat{\xi},X_{2}\)=\sup_{Z_{2}\in\Sigma_2}G(\mathcal{T}^{-1}\xi)\(\mathcal{T}^{-1}\xi,Z_{2}\)=\sup_{Z_{2}\in L}G(\mathcal{T}^{-1}\xi)\(\mathcal{T}^{-1}\xi,Z_{2}\).
		\end{align*}
		For $p\geq 1$, we have
		\begin{align}
			\hat{s}_{\Sigma_3}(\xi):=	&\sup_{Z_{3}\in\Sigma_3}G(\mathcal{T}^{-1}\xi)\(\mathcal{T}^{-1}\xi,Z_{3}\)=\sup_{Z_{3}\in K\widetilde{+}_pL}G(\mathcal{T}^{-1}\xi)\(\mathcal{T}^{-1}\xi,Z_{3}\)
			\nonumber
			\\
			=&\sup_{t\in[0,1],Z_{1}\in K,Z_{2}\in L}G(\mathcal{T}^{-1}\xi)\(\mathcal{T}^{-1}\xi, 	(1-t)^{\frac{1}{q}}Z_{1}^p+
			t^{\frac{1}{q}}Z_{2}^p \)
			\nonumber
			\\
			=&\sup_{t\in[0,1]}	(1-t)^{\frac{1}{q}}\hat{s}_K^p+
			t^{\frac{1}{q}}\hat{s}_L^p =(1-t_*)^{\frac{1}{q}}\hat{s}_K^p+
			{t_*}^{\frac{1}{q}}\hat{s}_L^p
			\nonumber
			%=(\alpha^p+\beta^p)^{\frac{1}{p}}
			\\
			=&G(\mathcal{T}^{-1}\xi)\(\mathcal{T}^{-1}\xi,(1-t_*)^{\frac{1}{q}}{X_{1}}^p+
			{t_*}^{\frac{1}{q}}{X_{2}}^p\),\label{equ:v(w)=xi}
		\end{align}
		where $t_*\in[0,1]$. We denote $X_{3}:=(1-t_*)^{\frac{1}{q}}{X_{1}}^p+
		{t_*}^{\frac{1}{q}}{X_{2}}^p\in \(K\widetilde{+}_p L\)\cap \partial\overline{\mathbb{R}^{n+1}_+}$. Since  $K\widetilde{+}_p L$ is convex, we have $
		X_{3}\in \partial\(K\widetilde{+}_p L\)\cap\partial\overline{\mathbb{R}^{n+1}_+}=\partial\Sigma_3$, and by \eqref{equ:v(w)=xi} we have $\widetilde{\nu_F}_3(X_{3})=\xi$, where $\widetilde{\nu_F}_3$ is the anisotropic  capillary Gauss map of $\Sigma_3$. That means $\Sigma_3$ is also an $\omega_0$-capillary hypersurface.  		
	\end{proof}
	Following the classical approach, the Alexandrov-Fenchel inequalities for anisotropic $\omega_0$-capillary convex bodies (established in \cite{arxiv2}) yield the corresponding Minkowski inequality. Using the H\"older inequality, we then obtain the corresponding $L^p$-Minkowski inequality for $p>1$, and consequently the corresponding $L^p$-Brunn-Minkowski inequality (cf. proof of \cite{book-convex-body}*{Theorem 9.1.3}):
%	Analogously to the classical case, using the Alexandrov–Fenchel inequalities for anisotropic $\omega_0$-capillary convex bodies established in \cite{arxiv2}, we can derive the corresponding Minkowski inequality (cf. proof of \cite{book-convex-body}*{Theorem 7.2.1}), then by holder inequality, we can derive the corresponding $L^p$-Minkowski inequality for $p>1$, then we can derive the corresponding $L^p$-Brunn-Minkowski inequality (cf. proof of \cite{book-convex-body}*{Theorem 9.1.3}):
	\begin{proposition}
		Let $\omega_0\in\left(-F(E_{n+1}),F(-E_{n+1})\right)$,  $p\geq 1$,  $K,L\in\mathcal{K}_{\omega_0}^{O}$ (or $K, L \in \mathcal{K}_{\omega_0}$ for $p=1$), then we have %$L^p$-Minkowski inequality:
		\begin{align*}
			\frac{1}{n+1}\int_{ \mathcal{C}_{\omega_0}} \hat{s}_L^p\hat{s}_K^{1-p}\sigma_n(\tau_{ij}[\hat{s}_K])~\mathrm{d} {\mu}_F(\cc)\geq \Vol(K)^{\frac{n+1-p}{n+1}} \Vol(L)^{\frac{p}{n+1}} .
		\end{align*}		
		For $t \in(0,1)$, the following  holds:
		\begin{align*}
				\Vol\left((1-t) \cdot K+_p t \cdot L\right) &\geq\left((1-t) \Vol(K)^{p / (n+1)}+t \Vol(L)^{p / (n+1)}\right)^{(n+1) / p}.
				%\\
				%&\geq \Vol(K)^{1-t}\Vol(L)^t.
		\end{align*}
		Equality holds if and only if $K$ and $L$ are dilates.
	\end{proposition}
	
	\subsection{Anisotropic capillary $p$-surface area measure and prescribed measure problems}
	%$L_p$ Minkowski Problem}

We introduce the anisotropic capillary mixed  quermassintegrals which are defined by the variation of anisotropic capillary quermassintegrals along the anisotropic capillary $p$-sum in Definition \ref{def:p-sum-fcn}.

\begin{definition}
	\label{def:Wp,k}
	Let $0 \leq k \leq n$ be an integer and $ p \geq 1$ be a real number. We define the anisotropic $\omega_0$-capillary  $p$-mixed $k$-th  quermassintegral of two smooth  anisotropic $\omega_0$-capillary convex bodies  $K, L \in \mathcal{K}_{\omega_0}^O$ (or $K, L \in \mathcal{K}_{\omega_0}$ for $p=1$) by		
	\begin{align*}
		W_{p, k,\omega_0}(K, L)=\lim _{t \rightarrow 0^{+}} \frac{\mathcal{V}_{k,\omega_0}\left(1\cdot K+{ }_p t \cdot L\right)-\mathcal{V}_{k,\omega_0}(K)}{t}.\
	\end{align*}	
\end{definition}

To calculate the integral representation of $W_{p, k,\omega_0}(K, L)$, we should study the scalar evolution equation of the anisotropic capillary support function $\hat{s}(\xi)$ with $
\xi\in\mathcal{C}_{\omega_0}$ along a general flow	\eqref{equ:p_tX=fv}, i.e., flow
\begin{align}	
	\partial_tX(\xi,t)=f(\xi,t)\cdot \mathcal{T}^{-1}\xi+T(\xi,t),\label{equ:ptX}
\end{align}
such that $G(\mathcal{T}^{-1}\xi)\(T(\xi,t),\mathcal{T}^{-1}\xi\)=0$. Now we assume that, on a small time interval, the solution to the above flow is given by a vector-valued function $X(\xi, t)$ defined on $\mathcal{C}_{\omega_0} \times[0, \varepsilon)$. %By  DeTurck's trick, we have
We can obtain the scalar evolution equation of $\hat{s}(\xi,t)$ as follows
\begin{align}
	\partial_t\hat{s}(\xi,t)=&\partial_t\left(G(\mathcal{T}^{-1}\xi)\(X,\mathcal{T}^{-1}\xi\)\right)\nonumber
	\\
	=&G(\mathcal{T}^{-1}\xi)\(\partial_tX,\mathcal{T}^{-1}\xi\)+G(\mathcal{T}^{-1}\xi)\(X,\partial_t(\xi-\omega_0E_{n+1}^F)\)\nonumber\\
	&+Q(\mathcal{T}^{-1}\xi)\(X,\mathcal{T}^{-1}\xi,\partial_t(\mathcal{T}^{-1}\xi)\)\nonumber\\
	=&f(\xi,t)\label{equ:pts}.
\end{align}
Conversely, for a family of anisotropic capillary convex functions $\hat{s}(\cdot, t)$ on $\mathcal{C}_{\omega_0}$ satisfying \eqref{equ:pts}  for $t\in [0,\varepsilon)$, they can determine a family of  anisotropic $\omega_0$-capillary convex hypersurfaces $\Sigma_t$. Let  $X(\xi,t)$ be the parametrization of $\Sigma_t$ given by its inverse anisotropic capillary Gauss map for $\xi\in\mathcal{C}_{\omega_0}, t\in [0,\varepsilon)$.  By DeTurck's trick, there exist diffeomorphisms $\varphi(\cdot,t):\mathcal{C}_{\omega_0}\to \mathcal{C}_{\omega_0}$ such that $\bar{X}(\xi,t):=X(\varphi(\xi,t),t)$ satisfies \eqref{equ:ptX}.

\begin{lemma}\label{lemma:Wpk}
	Let $k, p, K$ and $L$ satisfy the assumptions in Definition \ref{def:Wp,k}. Then the integral representation of the anisotropic capillary p-mixed $k$-th  quermassintegral $W_{p, k,\omega_0}(K, L)$ is given by	
	$$
	W_{p, k,\omega_0}(K, L)=\frac{n-k+1}{p(n+1)\binom{n}{k}}\int_{\mathcal{C}_{\omega_0}} \hat{s}_L^{p}(\xi)\hat{s}_K^{1-p}(\xi) \sigma_{n-k}(\tau_{ij}[\hat{s}_K](\xi))\mathrm{d} {\mu}_F(\cc).\ %\mathrm{d}\hat{\mu}
	$$
\end{lemma}
\begin{proof}
	Let $\widehat{\Sigma}_t=1\cdot K+{ }_p t \cdot L\in\mathcal{K}_{\omega_0}$. Then the anisotropic capillary support function of $\Sigma_t$ is given by		
	$$
	\hat{s}(\xi, t)= \(\hat{s}_K^p(
	\xi)+t \hat{s}_L^p(\xi)\)^{\frac{1}{p} }.
	$$		
	The initial speed of the above variation is given by			
	\begin{align}
		\frac{\partial}{\partial t} \hat{s}(\xi, 0)=\frac{1}{p} \hat{s}_K^{1-p}\hat{s}_L^{p}.\label{equ:pts0}
	\end{align}		
	By Definition \ref{def:Wp,k} and  Proposition \ref{Thm:p_t-Vk}, we have
	\begin{align*}
		W_{p, k,\omega_0}(K, L)=&\left.\frac{d}{d t}\right|_{t=0} \mathcal{V}_{ k,\omega_0}\left(\Sigma_t\right)
		\\
		=&\frac{n-k+1}{n+1}\int_{\Sigma_0}\left(\left.\frac{\partial}{\partial t}\right|_{t=0}\hat{s}(\widetilde{\nu_{{F}}}(\Sigma_t),t)\right)\cdot H^F_{k}(\Sigma_0)\mathrm{d}\mu_{{F}}
		\\
		=&\frac{n-k+1}{(n+1)\binom{n}{k}}\int_{\mathcal{C}_{\omega_0}}\frac{1}{p} \hat{s}_K^{1-p}\hat{s}_L^{p} \sigma_{n-k}(\tau_{ij}[\hat{s}_K])\mathrm{d} {\mu}_F(\cc),\ %\mathrm{d}\hat{\mu}
	\end{align*}
	where we use \eqref{equ:pts0},  $\widehat{\Sigma}_0=K$, $\binom{n}{k}H^F_{k}(\Sigma_0)=\frac{\sigma_{n-k}(\tau_{ij}[\hat{s}_K])}{\sigma_{n}(\tau_{ij}[\hat{s}_K])}$, and $F(\nu(\Sigma_0))\mathrm{d}\mu(\Sigma_0)=\mathrm{d}{\mu}_F(\Sigma_0)=\det(\tau_{ij}[\hat{s}_K])\mathrm{d}{\mu}_F(\mathcal{C}_{\omega_0})$.
\end{proof}

\begin{definition}\label{def:p-surface-area-measure}
	Let $0 \leq k \leq n$ be an integer and $p$ be a real number. The $k$-th anisotropic capillary p-surface area measure of a smooth anisotropic $\omega_0$-capillary convex body $K$ is defined by	
	$$
	\mathrm{d }S_{p, k}(K, \xi)=\hat{s}_K^{1-p} \sigma_{n-k}\left(\tau_{ij}\left[\hat{s}_K(\xi)\right]\right)\mathrm{d} {\mu}_F(\cc).\  % \mathrm{d}\hat{\mu} \ .
	$$
Particularly, we call $d S(K, \xi):=d S_{1,0}(K, \xi)$ the anisotropic capillary surface area measure of $K$.
\end{definition}
By Lemma \ref{lemma:Wpk} and Definition \ref{def:p-surface-area-measure}, we have
$$
W_{p, k,\omega_0}(K, L)=\frac{n-k+1}{p(n+1)\binom{n}{k}}\int_{\mathcal{C}_{\omega_0}} \hat{s}_L^{p}(\xi)\mathrm{d }S_{p, k}(K, \xi).\
$$	
This motivates us to propose the following two prescribed measure problems by using Definition \ref{def:p-surface-area-measure}, which are nonlinear elliptic PDEs on $\mathcal{C}_{\omega_0}$.\

\begin{problem}[Anisotropic capillary $p$-Minkowski problem]
	\label{problem:Anisotropic capillary $p$-Minkowski problem}
	Let $n \geq 1$ be an integer and $p$ be a real number. Given a smooth positive function $f(\xi)$ defined on $\mathcal{C}_{\omega_0}$, what are necessary and sufficient conditions for $f(\xi)$, such that there exists a smooth anisotropic capillary convex body $K\in \mathcal{K}_{\omega_0}^O$ satisfying	
	$$
	\mathrm{d} S_{p, 0}(K, \xi)=f(\xi) \mathrm{d} {\mu}_F(\cc). %\mathrm{d} \hat{\mu}\ .
	$$
	That is, finding a smooth positive  anisotropic capillary convex function  $\hat{s}(\xi)$ satisfies 	
	$$
	\hat{s}^{1-p}(\xi) \sigma_n(\tau_{ij}[\hat{s}(\xi)])=f(\xi), \quad \forall\xi\in\mathcal{C}_{\omega_0}.
	$$
\end{problem}	
\begin{proposition}\label{rk:int f=0}
If \eqref{equ:p=1 pde-inv-Gauss} is solvable for a given function $f : \cc \rightarrow \R$, then $f$ must satisfy:

	\begin{align*}
	\int_{\mathcal{C}_{\omega_0}}G(\hat{\xi})(\hat{\xi},E_{\al})f{~\rm d}{\mu}_{F}=0,\quad \forall \al=1,\cdots,n.
\end{align*}	
\end{proposition}
\begin{proof}
For any $1\leq \al \leq n$, Lemma \ref{Lemma:gaussmap} and \eqref{qqp} implies
	
		\begin{align*}
		0=&\int_{\widehat{\Sigma}}\div (E_{\al}) {\,\rm d}\mathcal{H}^{n+1}=\int_{\Sigma}\<E_{\al},\nu\> ~d\mu_g+\int_{\widehat{\partial\Sigma}}\<E_{\al},-E_{n+1}\> ~d\mu_{\widehat{\partial\Sigma}}\\
		=&\int_{\Sigma}\Omega(E_{\al}, \pa_{1}, \cdots, \pa_{n})dy^{1}\cdots dy^{n}\\
		=&\int_{\widetilde{\nu}^{-1}(\cc)}\Omega(E_{\al}, \pa_{1}, \cdots, \pa_{n})dy^{1}\cdots dy^{n}\\
		=&\int_{\widetilde{\nu}^{-1}(\cc)}G(\xii)(E_{\al},\xii)\Omega(\xii,\pa_{1}\cdots,\pa_{n})dy^{1}\cdots dy^{n}\\
		=&\int_{\mathcal{C}_{\omega_0}}{G(\hat{\xi})(E_{\al},\hat{\xi})}f~d\mu_F,\quad \al=1,\cdots,n.
	\end{align*}
	Here, $d{\mu}_{g} =\Omega(\nu,\pa_{1},\cdots,\pa_{n})dy^{1}\cdots dy^{n}$ is the induced volume from of the Euclidean metric in $\mathbb{R}^{n+1}$.
	\end{proof}

\begin{problem}[Anisotropic capillary $p$-Christoffel-Minkowski problem]
	Let $n \geq 2$ and $1\leq k\leq n-1$ be integers, and let $p$ be a real number. Given a smooth positive function $f(\xi)$ defined on $\mathcal{C}_{\omega_0}$, what are necessary and sufficient conditions for $f(\xi)$, such that there exists a smooth anisotropic capillary convex body $K\in \mathcal{K}_{\omega_0}^O$ satisfying	
	$$
	\mathrm{d} S_{p, k}(K, \xi)=f(\xi) \mathrm{d} {\mu}_F(\cc).
	$$
	That is, finding a smooth positive  anisotropic capillary convex function  $\hat{s}(\xi)$ satisfies 	
	$$
	\hat{s}^{1-p}(\xi) \sigma_{n-k}(\tau_{ij}[\hat{s}(\xi)])=f(\xi), \quad \xi\in\mathcal{C}_{\omega_0}.
	$$
\end{problem}
To maintain our focus on the $L_{p}$-Minkowski problem and to avoid diverging from the main theme of this paper, we postpone the study of this problem to a forthcoming work.

\section{The regularity of solutions to anisotropic capillary $L_p$-Minkowski problem}\label{section4}
\subsection{$C^{0}$ estimate}
We first establish a uniform positive lower bound and upper bound for $\hat{s}$ with $p\geq n+1$.
 \begin{definition}
We call a solution $\hat{s}$ to Eq. \eqref{equ:p pde-inv-Gauss} an admissible solution if 
\begin{enumerate}
\item $p=1$ or $p\geq n+1$, the $n\times n$ matrix $\hat{s}_{ij}-\frac{1}{2}Q_{ijk}\hat{s}_{k}+\delta_{ij}\hat{s}$ is positive definite;
\item $1<p<n+1$, the $n\times n$ matrix $\hat{s}_{ij}-\frac{1}{2}Q_{ijk}\hat{s}_{k}+\delta_{ij}\hat{s}$ is positive definite and $\hat{s}$ is an anisotropic capillary even function.
\end{enumerate}
The positive definiteness of this matrix is equivalent to that of
$$\tilde{s}_{ij}-\frac{1}{2}\tilde{Q}_{ijk}\tilde{s}_{k}+\tilde{s}\tilde{g}_{ij},$$
as shown in \eqref{equ:tau_ij}.
\end{definition}

\begin{lemma}\label{C01}
	Let $p> n+1$ and $\omega_0 \in\left(-F\left(E_{n+1}\right), F\left(-E_{n+1}\right)\right)$. Suppose $\hat{s}$ is an admissible solution to Eq. \eqref{equ:p pde-inv-Gauss}. Then there holds
	\begin{align*}
		\frac{(\underset{{\mathcal{C}_{\omega_0}}}{\min}\ \ell)^{p-n-1}}{\underset{{\mathcal{C}_{\omega_0}}}{\max}\ f\cdot(\underset{{\mathcal{C}_{\omega_0}}}{\max}\ell)^{p-1}}\leq \hat{s}^{p-n-1} \leq \frac{(\underset{{\mathcal{C}_{\omega_0}}}{\max}\ \ell)^{p-n-1}}{\underset{{\mathcal{C}_{\omega_0}}}{\min}\ f\cdot(\underset{{\mathcal{C}_{\omega_0}}}{\min}\ell)^{p-1}}.
	\end{align*}
\end{lemma}
\begin{proof}
	Consider the anisotropic capillary support function $\bar{s}=\ell^{-1}\hat{s}$. Then, in view of \eqref{equ:p pde-inv-Gauss}, we have that
	\begin{equation}\label{b1}
		\renewcommand{\arraystretch}{1.5}
		\left\{
		\begin{array}{rll}
			\det \left(\bar{s}_{ij}\ell-\frac{1}{2}Q_{ijk}\bar{s}_{k}\ell+\bar{s}_{i}\ell_{j}+\bar{s}_{j}\ell_{i}+\bar{s}\mathring{g}_{ij}\right)&=f(\bar{s}\ell)^{p-1}, &\text{in\ } \mathcal{C}_{\omega_0},\\
		{\mathring{\nabla}}_{e_{n}}\bar{s}&=0,&\text{on}\ \partial\mathcal{C}_{\omega_0},
		\end{array}	
		\right.
	\end{equation}
	where we used \eqref{pwf1} and \cite{Arxiv1}*{Lemma 3.4}.
	Suppose $\bar{s}$ attains the maximum at some point $\xi_{0}\in\mathcal{C}_{\omega_0}$. The boundary condition in \eqref{b1} implies that
	\begin{align}\label{l1}
		{\mathring{\nabla}}\bar{s}(\xi_{0})=0,\ {\mathring{\nabla}}^{2}\bar{s}(\xi_{0})\leq 0.
	\end{align}
	Substituting \eqref{l1} into \eqref{b1}, we obtain that
	\begin{align*}
		f({\bar{s}\ell})^{p-1}\leq \bar{s}^{n}\ \Leftrightarrow \bar{s}^{p-n-1}(\xi_{0})\leq f^{-1}\ell^{1-p}\leq\frac{1}{\underset{\mathcal{C}_{\omega_0}}{\min}f\cdot(\underset{\mathcal{C}_{\omega_0}}{\min}\ell)^{p-1}}.\
	\end{align*}
	Then for all $\xi \in \mathcal{C}_{\omega_0}$,
	\begin{align}\label{l2}
		\hat{s}^{p-n-1}(\xi)=(\bar{s}\ell)^{p-n-1}(\xi)\leq (\underset{{\mathcal{C}_{\omega_0}}}{\max}\ \ell)^{p-n-1}\cdot \bar{s}^{p-n-1}(\xi_{0})\leq \frac{(\underset{{\mathcal{C}_{\omega_0}}}{\max}\ \ell)^{p-n-1}}{\underset{{\mathcal{C}_{\omega_0}}}{\min}\ f\cdot(\underset{{\mathcal{C}_{\omega_0}}}{\min}\ell)^{p-1}}.\
	\end{align}
	Similarly, we also have that
	\begin{align*}
		\hat{s}^{p-n-1}\geq \frac{(\underset{{\mathcal{C}_{\omega_0}}}{\min}\ \ell)^{p-n-1}}{\underset{{\mathcal{C}_{\omega_0}}}{\max}\ f\cdot(\underset{{\mathcal{C}_{\omega_0}}}{\max}\ell)^{p-1}}.
	\end{align*}
\end{proof}
Next, we deal with the case $1\leq p < n+1$. In the case $p=1$, it follows from Remark \ref{rk:O-inter} that we can restrict to the situation where $\widehat{\partial\Sigma}$ contains the origin in its interior.
For $1<p<n+1$, from Definition \ref{gd}, we have
$$
\int_{\mathcal{C}_{\omega_{0}}} G(\mathcal{T}^{-1}\xi)\left(\mathcal{T}^{-1}\xi, E_i\right) \cdot \hat{s}(\xi) ~d \mu_F=0, \quad \forall i=1,2, \ldots, n .
$$
Since $\omega_0\leq 0$, and Wulff shape $\W$ is a symmetric convex hypersurface,  the Gauss map image of
$\mathcal{C}_{\omega_0}$
lies entirely in the upper hemisphere $\mathbb{S}^n_+=\{x\in\mathbb{S}^n:\<x,E_{n+1}\>\geq 0\}$. Consequently, the Gauss map image of $\Sigma$ also lies in $\mathbb{S}^n_+$. This implies that the union $$\Omega:=\widehat{\Sigma}\cup(-\widehat{\Sigma})$$
 is strictly convex and possesses an even anisotropic support function on  $\W$ denoted by $\hat{s}_{\Omega}(z)$ for $z\in \W$. If $\xi \in\cc$, we have $\hat{s}_{\Omega}(\mathcal{T}^{-1}\xi)=\hat{s}(\xi)$. Such an observation allows us to restrict $ \hat{s}_{\Omega}(z)$ to satisfy the following orthogonal condition:
$$
\int_{\W} G(z)\left(z, E_i\right) \cdot \hat{s}_{\Omega}(z) ~d \mu_F=0, \quad \forall i=1,2, \ldots, n+1.
$$
This orthogonal condition means that the origin lies in the interior of the convex body $\Omega$, which implies  the origin lies in the interior of $\widehat{\partial\Sigma}$. For closed hypersurfaces and isotropic capillary hypersurfaces, one can refer to  \cite{Xia13}*{Eq. (4.2)} and \cite{M-W-W}*{Lemma 3.1}, respectively. %enclosed by $\partial\Sigma\subset \mathbb{R}^n=\partial\overline{\mathbb{R}^{n+1}_+}$.

%Besides, we will use the following Proposition \cite{book-convex-body}, which will be used to derive a lower bound for the $C^0$ estimate in the case $p=1$.
%\begin{proposition}\label{prop:John ellipsoid}
	% Let convex body $K \in \mathbb{R}^{n+1}$, let $E$ be the John ellipsoid of $K$ and let $c$ be its centre. Then $E \subset K \subset n(E-c)+c$.

%If $K$ is o-symmetric, then $E \subset K \subset \sqrt{n} E$. $\underset{\mathcal{C}_{\omega_0}}{\min} f$ and $\underset{\mathcal{C}_{\omega_0}}{\max} f$.
%\end{proposition} $\underset{\mathcal{C}_{\omega_0}}{\min} f$ and $\underset{\mathcal{C}_{\omega_0}}{\max} f$
\begin{lemma}\label{C02}
	Consider an admissible solution $\hat{s}$ to Eq. \eqref{equ:p pde-inv-Gauss}. Let the exponent $p$ and the contact constant $\omega_0$ satisfy one of the following conditions:
	\begin{enumerate}
		\item$p=1$ and $\omega_0\in\bigl(-F(E_{n+1}),\,F(-E_{n+1})\bigr)$;
		\item$1<p<n+1$ and $\omega_0\in\bigl(-F(E_{n+1}),\,0\bigr]$.
	\end{enumerate}
	Then $\hat{s}$ admits uniform bounds
	\begin{align*}
		0<m_{1}\leq\hat{s}\leq m_{2},
	\end{align*}
	where the constants $m_{1},m_{2}>0$ depending only on $n$, $p$, $\underset{\mathcal{C}_{\omega_0}}{\min} f$ and $\underset{\mathcal{C}_{\omega_0}}{\max} f$.
	\end{lemma}
	\begin{proof}
	The anisotropic inner radius and outer radius of closed hupersurface relative to Wulff shape are defined in \cite{Xia-2017-convex},
	Since $\hat{\Sigma}\in \mathcal{K}^O_{\omega_0}$, in order to obtain the bounds for $\hat{s}$, it is sufficient to establish bounds on the anisotropic inner and outer radii of $\Sigma$ relative to $\mathcal{C}_{\omega_0}$.
	The anisotropic inner radius of $\Sigma$ relative to $\mathcal{C}_{\omega_0}$ is defined as
	\begin{align}\label{equ:inner radius}
		r(\Sigma):= \sup\{t > 0 | t\widehat{\mathcal{C}_{\omega_0}} + y \subset \hat{\Sigma}\  \text{for some}\ y\  \in \partial\overline{\mathbb{R}_+^{n+1}}\},
	\end{align}
	and the anisotropic outer radius of $\hat{\Sigma}$ relative to $\mathcal{C}_{\omega_0}$ is defined as
	\begin{align*}
		R(\Sigma):= \inf\{t > 0 | \hat{\Sigma} \subset t \widehat{\mathcal{C}_{\omega_0}} + y\ \text{for  some} \ y\  \in \partial\overline{\mathbb{R}_+^{n+1}}\}.
	\end{align*}
	
	\noindent
	{\bf Step \uppercase\expandafter{\romannumeral1}$: \bf\hat{s}\leq m_{2}$}.
	
	There exists some $R>0$ such that
	\begin{equation}\label{l3}
		\widehat{\Sigma} \subset  \widehat{\mathcal{W}}_{R,\omega_0}.
	\end{equation}
%	where
%	$$ \hat{\mathcal{C}}_{R, \omega_0}:=\{x\in\overline{\mathbb{R}_+^{n+1}}:F^0(x-R\omega_0E_{n+1}^F)\leq R\}\ .$$
	Let $R_{0}$ denote the smallest positive constant $R$ satisfying \eqref{l3}. It is clear that there exists $X_{0} \in \Sigma\bigcap R_{0}\mathcal{C}_{\omega_0}$. Set $\hat{X_{0}}=\frac{X_{0}}{R_{0}} \in \mathcal{C}_{\omega_0}$, for any $\xi \in \mathcal{C}_{\omega_0}$ such that $\hat{\xi}=\mathcal{T}^{-1}\xi$, we get that
	$$\hat{s}(\xi)=\underset{y \in \Sigma}{\sup}G({\xii})(\xii, y)\geq R_{0}\max\{0, G(\xii)(\xii, \hat{X_{0}})\}:=R_{0}A_{1}.$$
	Integrating over $\xi$ in $\cc$ yields
	\begin{align}\label{l4}
		\int_{\cc}\hat{s}^{p}d\mu_{F}\geq R_{0}^{p}\int_{\cc}A_{1}^{p} d\mu_{F},
	\end{align}
	and
	\begin{align}\label{l7}
		\int_{\cc}\hat{s}^{p} fd\mu_{F}\geq R_{0}^{p}\underset{\cc}{\min} f\int_{\cc}A_{1}^{p} d\mu_{F}.\
	\end{align}
	On the other hand, we have a Minkowski formula
	\begin{align}\nonumber
		\int_{\cc}f \hat{s}^{p}\du&=\int_{\Sigma}\hat{s}F(\nu)d\mu_{\Sigma}=\int_{\Sigma}G(\xii)(\xii, X)\Omega(\xii,\pa_{1}, \cdots, \pa_{n})dy^{1}\cdots dy^{n}\\ \label{l5}
		&=\int_{\Sigma}\Omega(X,\pa_{1},\cdots,\pa_{n})dy^{1}\cdots dy^{n}=\int_{\Sigma}\<X , \nu\>d\mu_{\Sigma}=(n+1)|\hat{\Sigma}|.
	\end{align}
	From the anisotropic capillary isoperimetric inequality in \cite{Arxiv}*{Theorem 1.3}, we have that
	\begin{align}\label{l6}
		|\hat{\Sigma}|^{\frac{n}{n+1}}\cdot A_{2}\leq |\Sigma|_{F}+\omega_{0}{|\widehat{\pa{\Sigma}}|}=\int_{\cc}\ell\hat{s}^{p-1}f \du,
	\end{align}
	where $A_{2}:=\frac{|\cc|_{F}+\omega_{0}|\widehat{\pa\cc}|}{|\widehat{\cc}|^{\frac{n}{n+1}}}$.  \eqref{l5}, \eqref{l6} and H\"older's inequality together yield
	\begin{align*}
		\int_{\cc}f\hat{s}^{p}\du &\leq \Big(\int_{\cc}\hat{s}^{p-1}f\du\Big)^{\frac{n+1}{n}}\left(\frac{\underset{\cc}{\max}\ \ell}{A_{2}}\right)^{\frac{n+1}{n}}(n+1)\\
		&\leq \left(\frac{\underset{\cc}{\max}\ \ell}{A_{2}}\right)^{\frac{n+1}{n}}(n+1)\Big(\int_{\cc}f\hat{s}^{p}\du\Big)^{\frac{(p-1)(n+1)}{pn}}\Big(\int_{\cc}f\du\Big)^{\frac{n+1}{np}}\\
		&:=A_{3}\Big(\int_{\cc}f\hat{s}^{p}\du\Big)^{\frac{(p-1)(n+1)}{pn}}\Big(\int_{\cc}f\du\Big)^{\frac{n+1}{np}}.
	\end{align*}
% 	where $A_3=\Big(\frac{\underset{\cc}{\max}\ \ell}{A_{2}}\Big)^{\frac{n+1}{n}}(n+1)$.
	Since $p<n+1$, we have that
	\begin{align*}
		\int_{\cc}\hat{s}^{p}\du \leq \frac{A_{3}^{\frac{np}{n+1-p}}}{\underset{\cc}{\min}\ f}\Big(\int_{\cc}f\du\Big)^{\frac{n+1}{n+1-p}}.
	\end{align*}
	Combining with \eqref{l4}, we get the positive upper bound of $R_{0}$, which implies the positive upper bound of $\hat{s}$, see also in \cite{Xia13}*{Lemma 4.2}.
	
	\noindent
	{\bf Step \uppercase\expandafter{\romannumeral2}$: \bf\hat{s}\geq m_{1}$}.
	
	Since $|\hat{\Sigma}|$ is enclosed by a rescaled Wulff shape $\cc$ with radius $R(\Sigma)$, we have that
	\begin{align}\label{l8}
		|\hat{\Sigma}|\leq %\frac
		\frac{1}{n+1}{\widehat{|\cc|}R^{n+1}(\Sigma)}.\ %{n+1}.
	\end{align}
	It follows from \eqref{l7}, \eqref{l5} and \eqref{l8} that
	\begin{align}\label{l9}
		R_{0}\geq\Big(\int_{\cc}A_{1}^{p}\du\Big)^{\frac{1}{n+1-p}}\Big(\underset{\cc}{\min}\ f\Big)^{\frac{1}{n+1-p}}\left(\widehat{|\cc|}\right)^{\frac{-1}{n+p-1}},
	\end{align}
	and
	\begin{align}\label{l10}
		(n+1)|\hat{\Sigma}|\geq R_{0}^{p}\ \underset{\cc}{\min}f \cdot\int_{\cc}A_{1}^{p}\du.\
	\end{align}
%	From \eqref{support}, we have
%	$$\hat{s}=\sup_{x\in\widehat{\Sigma}}G(\mathcal{T}^{-1}\xi)\(\mathcal{T}^{-1}\xi,x\)=\frac{\<X(\xi), \nu(X({\xi}))\>}{F(\nu(X({\xi}))}=\frac{u(\xi)}{F(\nu(X({\xi}))}.\ $$
	Combining \eqref{l9} with \eqref{l10}, we conclude that the volume $|\hat{\Sigma}|$ has a uniform positive lower bound. % and $\hat{s}$ has a uniform positive upper bound $R_{0}$, which implies that $u$ also has a uniform positive lower bound, see also \cite{M-W-W1}*{Lemma 3.3}, since $F$ is even on $\mathbb{S}^n$ and $f$ is even on $\cc$ for case $p>1$. Therefore, the lower bound of the anisotropic support function $\hat{s}$ is determined by the lower bound of the isotropic support function $u$ and the upper bound of $F$ in case $p>1$.
	 Besides, we can calculate that
	 \begin{align*}
		\mathcal{V}_{1,\omega_0}(\Sigma)\overset{\eqref{equ:Vk=Vkkklll}}{=}&\frac{1}{n+1}\int_{\Sigma}1+\omega_0G(\nu_{{F}})(\nu_{{F}},E_{n+1}^F) ~d\mu_F
	\\
	= &\frac{1}{n+1}\int_{\cc}
	\frac{F(\nu(\xi))+\omega_{0}\<\nu(\xi), E_{n+1}^{F}\>}{K(\widetilde{\nu_{{F}}}^{-1}(\xi))} ~d\mathcal{H}^n(\xi)
	\\
	\overset{\eqref{equ:K^-1=f}}{=}&\frac{1}{n+1}\int_{\cc}
	\left(F(\nu(\xi))+\omega_{0}\<\nu(\xi), E_{n+1}^{F}\>\right) f ~d\mathcal{H}^n(\xi).
	\end{align*}
	Since $|\nu(\xi)|\leq 1$ for any $\xi\in\cc$, we know $F(\nu(\xi))+\omega_{0}\<\nu(\xi), E_{n+1}^{F}\>$ has a uniform upper bound, then $\mathcal{V}_{1,\omega_0}(\Sigma)$  also has a uniform upper bound.
	From \eqref{equ:Vol<r} and \eqref{l10}, we know $r(\Sigma)$ has a uniform positive lower bound.
	
	For any $\xi \in \cc$, let $t(\xi)>0$ be the number such that $t(\xi) \xi \in \Sigma$. It follows from the definition of $r(\Sigma)$ that $r(\Sigma) \leq \underset{\xi \in \cc}{\inf}t(\xi)$. Consequently,	for any $\xi \in \cc$ we have
	\begin{align*}
		\hat{s}(\xi)=&\sup _{X \in \Sigma} G(\mathcal{T}^{-1}\xi)(\mathcal{T}^{-1}\xi, X)
		\\
		\geq& G(\mathcal{T}^{-1}\xi)\left(\mathcal{T}^{-1}\xi, t(\xi) \left(\mathcal{T}^{-1}\xi+\omega_0 E_{n+1}^F\right)\right)
		\\
		=&t(\xi)(1+\omega_0G(\mathcal{T}^{-1}\xi)(\mathcal{T}^{-1}\xi,E_{n+1}^F))
		\\
		 \geq& r(\Sigma)c.
	\end{align*}
	%where we use the fact $1+\omega_0G(\mathcal{T}^{-1}\xi)(\mathcal{T}^{-1}\xi,E_{n+1}^F)=\frac{F(\nu(\xi))+\omega_{0}\<E_{n+1}^F,\nu(\xi)\>}{F(\nu(\xi))}>0$, which yields the uniform lower bound  $1+\omega_0G(\mathcal{T}^{-1}\xi)(\mathcal{T}^{-1}\xi,E_{n+1}^F)\geq c>0$.	
	Therefore, the lower bound of the anisotropic support function $\hat{s}$ is determined by the lower bound of $r(\Sigma)$.
	 %In case  $p=1$,  combining with Proposition \ref{prop:John ellipsoid} and , we have the lower bound of $r$, then similar to \cite{Xia13}, we obtain he lower bound of $\hat{s}$.
	 %\begin{align*}
	 %	\hat{s}=\sup_{x\in{\Sigma}}G(\mathcal{T}^{-1}\xi)\(\mathcal{T}^{-1}\xi,x\)\leq
	 %\end{align*}
\end{proof}

\subsection{$C^{1}$ estimate}
Inspired by the role of spherical caps in the classical isotropic capillary Minkowski problem in the Euclidean upper half-space, we naturally consider that for the anisotropic capillary Minkowski problem, the target model is no longer the Wulff shape $\mathcal{W}$, but rather the translated Wulff shape $\w$. Therefore, we establish the anisotropic geometric quantities of the translated Wulff shape $\w$ and reparameterize Eq. \eqref{equ:p pde-inv-Gauss} using these new anisotropic geometric quantities.

\begin{flushleft}
\textbf{Reparameterize the problem by the translated Wulff shape $\widetilde{\W}$}.
\end{flushleft}

For a given Wulff shape $\W$ and a constant  $\omega_0 \in(-F(E_{n+1}), F(-E_{n+1}))$, we can define a new Wulff shape $\widetilde{\W}:=\rm \mathcal{T}(\W)\subset\mathbb{R}^{n+1}$, where $\mathcal{T}$ is a translation  transformation on $\mathbb{R}^{n+1}$ defined by
$$\mathcal{T}(y)=y+\omega_0E^F_{n+1},\quad \forall y\in\mathbb{R}^{n+1}.$$
Since $\omega_0 \in(-F(E_{n+1}), F(-E_{n+1}))$, we have $F^0(O-\omega_0E^{F}_{n+1})=F^0(-\omega_0E^{F}_{n+1})<1$, then the translated Wulff shape $\widetilde{\W}=\{y\in\mathbb{R}^{n+1}:F^0(y-\omega_0E^{F}_{n+1})=1\}$ still encloses the origin $O$ in its interior.
\begin{figure}[h]
	\centering
	\label{fig:2}
	\begin{tikzpicture}[x=0.75pt,y=0.75pt,yscale=-1,xscale=1]
		%uncomment if require: \path (0,433); %set diagram left start at 0, and has height of 433
		
		%Shape: Polygon Curved [id:ds225334415266911]
		\draw   (96.2,78.6) .. controls (125.2,67.6) and (229,37) .. (260.2,61.6) .. controls (291.4,86.2) and (326.4,138.2) .. (284.2,158.6) .. controls (242,179) and (178,158) .. (135.2,141.6) .. controls (92.4,125.2) and (67.2,89.6) .. (96.2,78.6) -- cycle ;
		%Straight Lines [id:da252061447758283]
		\draw  [dash pattern={on 4.5pt off 4.5pt}]  (220,21) -- (355.2,152.6) ;
		%Straight Lines [id:da033731654504067565]
		\draw  [dash pattern={on 4.5pt off 4.5pt}]  (174,45) -- (310.2,180.6) ;
		%Straight Lines [id:da23208520360134832]
		\draw    (224,95) -- (241.59,77.41) ;
		\draw [shift={(243,76)}, rotate = 135] [fill={rgb, 255:red, 0; green, 0; blue, 0 }  ][line width=0.08]  [draw opacity=0] (12,-3) -- (0,0) -- (12,3) -- cycle    ;
		%Straight Lines [id:da8984375749856972]
		\draw    (270,71) -- (287.59,53.41) ;
		\draw [shift={(289,52)}, rotate = 135] [fill={rgb, 255:red, 0; green, 0; blue, 0 }  ][line width=0.08]  [draw opacity=0] (12,-3) -- (0,0) -- (12,3) -- cycle    ;
		%Straight Lines [id:da024612131866353915]
		\draw [color={rgb, 255:red, 208; green, 2; blue, 27 }  ,draw opacity=1 ]   (224,95) -- (268.23,71.93) ;
		\draw [shift={(270,71)}, rotate = 152.45] [fill={rgb, 255:red, 208; green, 2; blue, 27 }  ,fill opacity=1 ][line width=0.08]  [draw opacity=0] (12,-3) -- (0,0) -- (12,3) -- cycle    ;
		%Straight Lines [id:da6850033206784607]
		\draw    (232.2,145.6) -- (224.32,96.97) ;
		\draw [shift={(224,95)}, rotate = 80.79] [fill={rgb, 255:red, 0; green, 0; blue, 0 }  ][line width=0.08]  [draw opacity=0] (12,-3) -- (0,0) -- (12,3) -- cycle    ;
		%Straight Lines [id:da0695681044714207]
		\draw    (232.2,145.6) -- (269.1,72.78) ;
		\draw [shift={(270,71)}, rotate = 116.87] [fill={rgb, 255:red, 0; green, 0; blue, 0 }  ][line width=0.08]  [draw opacity=0] (12,-3) -- (0,0) -- (12,3) -- cycle    ;
		%Shape: Polygon Curved [id:ds013622522926809655]
		\draw   (52.6,105.3) .. controls (81.6,94.3) and (185.4,63.7) .. (216.6,88.3) .. controls (247.8,112.9) and (282.8,164.9) .. (240.6,185.3) .. controls (198.4,205.7) and (134.4,184.7) .. (91.6,168.3) .. controls (48.8,151.9) and (23.6,116.3) .. (52.6,105.3) -- cycle ;
		%Straight Lines [id:da13895548840802063]
		\draw  [dash pattern={on 4.5pt off 4.5pt}]  (232.2,145.6) -- (290.2,91.6) ;
		%Shape: Brace [id:dp0034199174630360485]
		\draw   (232.2,145.6) .. controls (235.47,149.93) and (238.77,149.96) .. (242.1,146.69) -- (269.99,119.31) .. controls (274.74,114.64) and (278.75,113.96) .. (282.02,117.29) .. controls (278.75,113.96) and (279.5,109.96) .. (284.26,105.29)(282.11,107.39) -- (288.11,101.5) .. controls (293.44,98.23) and (293.47,94.93) .. (291.2,91.6) ;
		
		% Text Node
		%\draw (721,21) node    {$0$};
		% Text Node
		%	\draw (701,71) node    {$0$};
		% Text Node
		\draw (215,110) node [anchor=north west][inner sep=0.75pt]   [align=left] {$z$};
		% Text Node
		\draw (244,93) node [anchor=north west][inner sep=0.75pt]   [align=left] {$\tilde{z}$};%=z+\omega_0E_{n+1}^F$};
	% Text Node
	\draw (32,91) node [anchor=north west][inner sep=0.75pt]   [align=left] {$\W$};
	% Text Node
	\draw (60,35) node [anchor=north west][inner sep=0.75pt]   [align=left] {$\widetilde{\W}=\W+\omega_0E_{n+1}^F$};
	% Text Node
	\draw (226,70) node [anchor=north west][inner sep=0.75pt]   [align=left] {$x$};
	% Text Node
	\draw (282,40) node [anchor=north west][inner sep=0.75pt]   [align=left] {$x$};
	% Text Node
	\draw (280,65) node [anchor=north west][inner sep=0.75pt]   [align=left] {\color{red}$\omega_0E_{n+1}^F$};
	% Text Node
	\draw (221,146) node [anchor=north west][inner sep=0.75pt]   [align=left] {$O$};
	% Text Node
	\draw (270,120) node [anchor=north west][inner sep=0.75pt]   [align=left] {$\widetilde{F}(x)$};
	% Text Node
	\draw (295,178) node [anchor=north west][inner sep=0.75pt]   [align=left] {$T_z\W=T_x\mathbb{S}^n$};
	% Text Node
	\draw (343,154) node [anchor=north west][inner sep=0.75pt]   [align=left] {$T_{\tilde{z}}\widetilde{\W}=T_x\mathbb{S}^n$};

\end{tikzpicture}
\caption{$\W$ and $\widetilde{\W}$} \label{fig:0}
\end{figure}
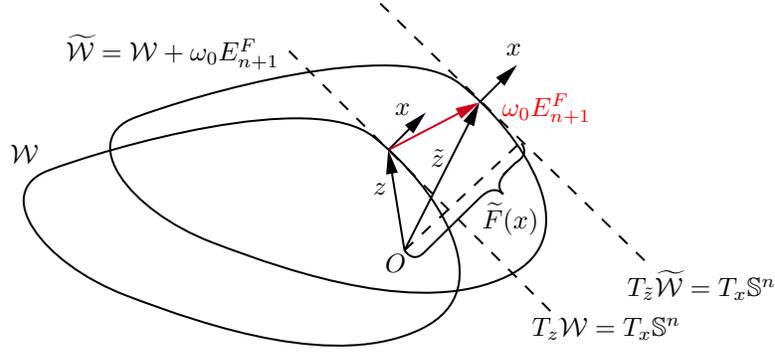 \\

For any $z\in\W$, we have $\tilde{z}:=z+\omega_0E_{n+1}^F\in \widetilde{\W}$. If $x$ is the unit outward normal vector of $\W\subset\mathbb{R}^{n+1}$ at point $z$, then $x$ is also a unit outward normal vector of $\widetilde{\W}\subset\mathbb{R}^{n+1}$ at point $\tilde{z}$,  since $T_z\W=T_{\tilde{z}}\widetilde{\W}$ (see Fig. \ref{fig:0}). Then
the support function $\widetilde{F}$ of $\widetilde{\W}$ satisfies
\begin{align*}
\widetilde{F}(x)&=\<\tilde{z},x\>=\<z,x\>+\<x,\omega_0E_{n+1}^F\>
\\
&=F(x)+\omega_0\<E_{n+1}^F,x\>, \quad\forall x\in \mathbb{S}^n.
\end{align*}
We extend it to a $1$-homogeneous function on $\mathbb{R}^{n+1}$ as
$$\widetilde{F}(y)=F(y)+\omega_0\<E_{n+1}^F,y\>, \quad\forall y\in \mathbb{R}^{n+1}.$$
%\tikzset{every picture/.style={line width=0.75pt}} %set default line width to 0.75pt
Then $D\widetilde{F}=DF+\omega_0E_{n+1}^F, D^2\widetilde{F}=D^2F$.
Obviously the translation map does not change the principal curvature radii of $\W$ and $\widetilde{\W}$, i.e., ${\widetilde{A_F}}(x):=D^2\widetilde{F}|_{T_x\mathbb{S}^n}(x)=D^2F|_{T_x\mathbb{S}^n}(x)=A_F(x)$, for $x\in\mathbb{S}^n$.
We denote $\widetilde{F}^{0}$ as the  dual Minkowski norm of $\widetilde{F}$. The new metric $\widetilde{G}$ and $(0,3)$-tensor $\widetilde{Q}$ with respect to $\widetilde{F}^0$ are respectively constructed by \eqref{equ:G} and \eqref{equ:Q} in Section \ref{subsec 2.3}. The relationships between $G$ and $\widetilde{G}$, $Q$ and $\widetilde{Q}$  can be found in Appendix of \cite{Arxiv1}.

We use a tilde $(\ \widetilde{}\ )$ to indicate the related anisotropic geometric quantities with respect to the new Wulff shape $\widetilde{\W}$ in this subsection.  For  $X\in\Sigma$, denote  anisotropic Gauss map with respect to $\widetilde{\W}$ as: $$\widetilde{\nu_{{F}}}(X)=D\widetilde{F}(\nu(X))=DF(\nu)+\omega_0E_{n+1}^F=\mathcal{T}(\nu_{{F}}(X)).$$
Anisotropic support function with respect to $\widetilde{\W}$ is:
\begin{align}\label{equ:SF-Translation invariant}
\widetilde{s}(X)=\frac{\<X,\nu\>}{\widetilde{F}(\nu)}=\frac{\<X,\nu\>}{F(\nu)+\omega_0\<E_{n+1}^F,\nu\>}=\bar{s}(X).
\end{align}
Anisotropic Weingarten matrix with respect to $\widetilde{\W}$ is:
$$\widetilde{S_F}(X)=\widetilde{A_F}\operatorname{d}\nu={A_F}\operatorname{d}\nu=S_F(X),$$
that means the anisotropic principal curvatures (or principal curvature radii) of $\Sigma$ with respect to $\W$  are the same as anisotropic principal curvatures (or principal curvature radii) of $\Sigma$ with respect to $\widetilde{\W}$, i.e.,  $\widetilde{\kappa^F}(X)=\kappa^F(X), \tilde{\tau}(X)=\tau(X)$.

On the boundary of  $\omega_0$-capillary hypersurface, since $$\<\widetilde{\nu_F},-E_{n+1}\>=\<{\nu_F}+\omega_0E_{n+1}^F,-E_{n+1}\>=\<{\nu_F},-E_{n+1}\>-\omega_0=0,$$ we have the following proposition.
\begin{proposition}
%If we solve the anisotropic free boundary problem, then we solve the  anisotropic  $\omega_0$-capillary boundary problem for general $\omega_0$.
The anisotropic $\omega_0$-capillary boundary problem for Wulff shape $\W$ can be naturally solved based on the results of the corresponding free boundary problem for the translated Wulff shape $\widetilde{\W}$.
\end{proposition}

Let $\Sigma$ be a strictly convex anisotropic  $\omega_0$-capillary hypersurface, we know  $\widetilde{\W}$ is also a strictly convex hypersurface, then $\widetilde{\nu_{{F}}}$ is an everywhere nondegenerate diffeomorphism. We use it to reparameterize $\Sigma$:
$$X:\widetilde{\W}\rightarrow\Sigma,\quad X(\tilde{z})=X(\widetilde{\nu_F}^{-1}(\tilde{z})),\  \tilde{z}\in \widetilde{\W}.$$
Then from \eqref{equ:SF-Translation invariant} we have
\begin{align}\label{equ:s-trandlation}
%	\tilde{s}\circ \mathcal{T}(z)=\tilde{u}\circ{\nu_F}^{-1}=\bar{u}\circ{\nu_F}^{-1}=\frac{s(z)}{1+G(z)(z,E_{n+1}^F)}
\tilde{s}\circ \mathcal{T}(z)=\frac{\hat{s}(z)}{1+\omega_0G(z)(z,E_{n+1}^F)}
, \quad z\in\W.
\end{align}
If we still write $\kappa^F=\kappa^F\circ \nu_{{F}}^{-1},\widetilde{\kappa^F}=\widetilde{\kappa^F}\circ \widetilde{\nu_{{F}}}^{-1},\tau=\tau\circ\nu_{{F}}^{-1}$, and $\tilde{\tau}=\tilde{\tau}\circ\widetilde{\nu_{{F}}}^{-1}$, then
$$\widetilde{\kappa^F}\circ \mathcal{T}=\kappa^F, \quad\tilde{\tau}\circ \mathcal{T}=\tau.$$
We denote $\tilde{s}_i=\tilde{\nabla}_{\tilde{e}_i}\tilde{s},\, \tilde{s}_{ij}=\tilde{\nabla}_{\tilde{e}_i}\tilde{\nabla}_{\tilde{e}_j}\tilde{s},\, \tilde{Q}_{ijk}=\tilde{Q}(\tilde{z})(\tilde{e}_i,\tilde{e}_j,\tilde{e}_k)$, where $\tilde{\nabla}$ is the covariant derivative on $\widetilde{\W}$ with respect to metric $\tilde{g}=\widetilde{G}(\tilde{z})|_{T_{\tilde{z}}\widetilde{\W}}$, and $\{\tilde{e}_i\}_{i=1}^n$ are standard orthogonal frame of $(\widetilde{\W},\tilde{g},\tilde{\nabla})$.
From \cite{Xia-2017-convex,Xia13}, the anisotropic principal curvatures at $X(\tilde{z})\in\Sigma$ are the reciprocals of  the eigenvalues of
\begin{align}\label{equ:tau_ij}
\tilde{\tau}_{ij}(\tilde{z}):=\tilde{s}_{ij}-\frac{1}{2}\tilde{Q}_{ijk}\tilde{s}_k+\tilde{g}_{ij}\tilde{s}>0, \quad \tilde{z}\in\widetilde{\W}.
\end{align}
Then Eq. \eqref{equ:p pde-inv-Gauss} can be rewritten as
\begin{equation}\label{rewrite}
\renewcommand{\arraystretch}{1.5}
\left\{
\begin{array}{rll}
	\det \left(\tilde{s}_{ij}-\frac{1}{2}\tilde{Q}_{ijk}\tilde{s}_{k}+\tilde{s}\tilde{g}_{ij}\right)&=\tilde{f}\tilde{s}^{p-1},  &\text{in\ } \cc,
	\\
	 \<{\tilde{\nabla}} \tilde{s},E_{n+1}\>&=0,&  \text{on}\ \partial\cc,\
\end{array}	
\right.
\end{equation}
where $\tilde{f}=f\ell^{p-1}$.
It is convenient to use $\tilde{s}$ to establish $C^{1}$ and $C^{2}$ estimates.

\begin{lemma}[\cite{Arxiv1}, Lemma 5.10]\label{pwf2}
For $\tilde{z} \in \pa \cc  $, we take standard orthogonal frame $\{\tilde{e}_{i}\}_{i=1}^{n} \in T_{\tilde{z}}\cc $ with respect to $\tilde{g}$, which satisfies $\tilde{e}_{\al} \in T_{\tilde{z}}\pa\cc  $ for $\al =1, \cdots, n-1$, and $\tilde{g}(\tilde{e}_{\al}, \tilde{e}_{n})=0$ for any $\tilde{e}_{\al}$. Then we have
$$\tilde{\nabla}_{\tilde{e}_{n}}\s=0,\ and\ \  \tt_{\al n}=0,\ \ \ \ \ on\ \  \pa(\cc  \cap  \overline{\mathbb{R}_+^{n+1}}),$$
and
$$\tt_{\al \al, n}=\tilde{Q}_{\al \ga n}(\tt_{\al \ga}-\delta_{\al \ga}\tt_{nn}) \ \ on\ \  \pa(\cc  \cap  \overline{\mathbb{R}_+^{n+1}}).$$
\end{lemma}\label{g1}
\begin{lemma}[\cite{Arxiv1}, Lemma 5.11]\label{g3}
For any $\tilde{z} \in  \pa\cc$ and $\tilde{e}_{\al} \in T_{\tilde{z}}(\pa\cc)$, we have
\begin{align}\label{equ:tildeQ-Q}
	\tilde{Q}_{\al \al n}=\frac{1}{1+\omega_{0}G(z)(z, E_{n+1}^{F})}\Big(Q_{\al \al n}-\frac{\omega_{0}G(z)(\tilde{e}_{\al}, \tilde{e}_{\al})}{F(\nu)\<\mu, E_{n+1}\>}\Big).\
\end{align}
\end{lemma}
\begin{remark}\label{r4.1}
Clearly, the result of Lemma \ref{g3} describes an important fact: Condition \eqref{condition} is equivalent to $\tilde{Q}_{\al\al n}<0$, which is crucial for our proof of the a priori estimates. Meanwhile, the inequality $\tilde{Q}_{\alpha\alpha n} < 0$ means that the matrix $\tilde{Q}_{\al \be n}$ is negative definite. Consequently, from the following proposition, we see that this is equivalent that the boundary $\pa\cc \subset \cc$ is anisotropically strictly convex.\end{remark}

\begin{proposition}\label{g2}
The following boundary conditions on $\pa\cc $:
$$\tt_{\al n}=0,\ \ \tilde{Q}_{\al \be n}<0 ,$$
for $\forall \al, \be =1,\cdots, n-1$, is equivalent to  the following matrix is positive definite:
$$\left(\hat{h}_{\al \be}^{\pa\cc }\right)>0,$$
where $\hat{h}_{\al \be}^{\pa\cc }:=\tilde{g}(\tilde{e}_{\al}, \tilde{\nabla}_{\tilde{e}_{\be}}\tilde{e}_{n})$. This shows that $\pa\cc  \subset \cc $ is anisotropically strictly convex.
\end{proposition}
\begin{proof}
\begin{align*}
	\tt_{\al n}|_{\pa\cc }=0&\Leftrightarrow \s_{\al n}-\frac{1}{2}\tilde{Q}_{\al \be n}\s_{\be}=0\\
	&\Leftrightarrow \tilde{\nabla}_{\tilde{e}_{\al}}\tilde{g}(\tilde{\nabla}\s, \tilde{e}_{n})-\tilde{g}(\tilde{\nabla}\s, \tilde{\nabla}_{\tilde{e}_{\al}}\tilde{e}_{n})=\frac{1}{2}\tilde{Q}_{\al \be n}\s_{\be}\\
	&\Leftrightarrow -\tilde{g}(\tilde{\nabla}\s, \tilde{\nabla}_{\tilde{e}_{\al}}\tilde{e}_{n})=\frac{1}{2}\tilde{Q}_{\al \be n}\s_{\be}\\
	&\Leftrightarrow \hat{h}_{\al \be}^{\pa\cc }=-\frac{1}{2}\tilde{Q}_{\al \be n}>0,\
\end{align*}
where we use the fact that $\s_{n}=0$\ on\  $\pa\cc $.\
\end{proof}

The $C^{1}$ estimate can be derived without using Eq. \eqref{rewrite}.
\begin{lemma}\label{C1}
Consider an admissible solution $\hat{s}$ to Eq. \eqref{equ:p pde-inv-Gauss}. Let the exponent $p$ and the contact constant $\omega_0$ satisfy one of the following conditions:
	\begin{enumerate}
		\item$p=1$ or $p\geq n+1$,and $\omega_0\in\bigl(-F(E_{n+1}),\,F(-E_{n+1})\bigr)$ together with Condition \eqref{condition};
		\item$1<p<n+1$ and $\omega_0\in\bigl(-F(E_{n+1}),\,0\bigr]$ together with Condition \eqref{condition}.
	\end{enumerate}
Then there exists a constant C depending on $n,\omega_{0},\Vert F\Vert_{C^{3}(\W)}$, such that
$$|\tilde{\nabla}\s|\leq C.\ $$
\end{lemma}
\begin{proof}

	Suppose that $|\tilde{\nabla}\s|\gg1$, otherwise, we are done. From Lemma \ref{C02} and
	$\tilde{s}\circ \mathcal{T}(z)=\frac{\hat{s}(z)}{F(\nu)+\omega_0\<\nu,E_{n+1}^F\>}$ for $z\in\W,$  we know that $\tilde{s}$ is also bounded above and below, and we still denote these bounds by $m_1$ and $m_2$.
Suppose $\Psi =\log |\tilde{\nabla}\s|^{2}-\log(2m_{2}-\s)$ attains its maximum at $\xi_{0} \in \pa\cc $, the maximum principle on the boundary implies
\begin{align*}
	0\leq \Psi_{n}=\frac{2\s_{kn}\s_{k}}{ |\tilde{\nabla}\s|^{2}}+\frac{\s_{n}}{2m_{2}-\s}=\frac{2\s_{\al n}\s_{\al}}{ |\tilde{\nabla}\s|^{2}}=\frac{2(\tilde{\tau}_{\al n}+\frac{1}{2}\tilde{Q}_{\al \ga n}\s_{\ga})\s_{\al}}{ |\tilde{\nabla}\s|^{2}}<0,
\end{align*}
which yields a contradiction. Then the maximum point $\xi_{0} \in \cc \setminus\pa \cc $. Assume at this point $\tilde{\tau}$ is diagonal and w.l.o.g. $|\s_{1}|\geq \frac{1}{\sqrt{n}}|\tilde{\nabla}\s|>0$. At $\xi_{0}$, we have
\begin{align*}
	0=\s_{1}\Psi_{1}&=\frac{2\s_{1}\s_{k}\s_{k1}}{ |\tilde{\nabla}\s|^{2}}+\frac{\s_{1}^{2}}{2m_{2}-\s}\\
	&=\frac{2\s_{1}\s_{k}(\tilde{\tau}_{1k}+\frac{1}{2}\tilde{Q}_{1kl}\s_{l}-\delta_{1k}\s)}{|\tilde{\nabla}\s|^{2}}+\frac{\s_{1}^{2}}{2m_{2}-\s}\\
	&\geq -C|\tilde{\nabla}\s|+\frac{1}{n}\frac{|\tilde{\nabla}\s|^{2}}{2m_{2}-m_{1}},
\end{align*}
which implies that
$$|\tilde{\nabla}\s|\leq C.\ $$

\end{proof}

In order to deal with the case $p=n+1$, we proceed to establish the logarithmic gradient estimate for solutions to Eq. \eqref{rewrite}. % Our choice of testfunctions is inspired by the ideas developed in \cite{M-W-W1,Ma-Qiu,Ma-Xu}. %, in the context of boundary value problems.
\begin{lemma}\label{logC1}
Let $p\geq n+1$. Let $\omega_0$ satisfy $\omega_0 \in (-F(E_{n+1}), F(-E_{n+1}))$ and Condition \eqref{condition}. Suppose $\tilde{s}$ is admissible  solution to Eq. \eqref{rewrite}, then there exists a constant C depending on $n,\omega_{0},\Vert F\Vert _{C^{3}(\W)}$, such that
\begin{align}\label{equ:c1-log}
	\underset{\cc}{\max}|\tilde{\nabla} \log \tilde{s}|\leq C.\
\end{align}
Furthermore, if $n+1\leq p \leq n+2$, the constant $C$ in \eqref{equ:c1-log} is independent of $p$.
\end{lemma}
\begin{proof}
Define $v=: \log \tilde{s}$, from \eqref{rewrite}, then $v$ satisfies
\begin{equation}
	\renewcommand{\arraystretch}{1.5}
	\left\{
	\begin{array}{rll}
		\det \left(v_{ij}+v_{i}v_{j}-\frac{1}{2}\tilde{Q}_{ijk}v_{k}+\delta_{ij}\right)&=e^{p_{0}v}\tilde{f}, & \text{in\ } \cc,
		\\
		\tilde{\nabla}_{e_{n}}v&=0, & \text{on}\ \partial\cc,
	\end{array}	
	\right.
\end{equation}
where $p_{0}:=p-n-1$. Consider the function
$$\Phi:=\log|\tilde{\nabla}v|^{2}+e^{4(\lambda+v)}|\tilde{\nabla}v|^{2},$$
for some positive constant $\lambda \in \R$. Suppose that $|\tilde{\nabla} \log \tilde{s}|\gg 1$; otherwise, we are done.
Assume that $\Phi$ attains its maximum at some point, say $\xi_{0} \in \cc $. We may assume that $|\tilde{\nabla}v(\xi_0)|$ is sufficiently large,  otherwise, the proof is complete. We divide the proof into cases: either $\xi_{0} \in \pa\cc $ or $\xi_{0}\in \cc \setminus\pa \cc $.

Case \uppercase\expandafter{\romannumeral1}. $\xi_{0} \in \pa\cc $. The maximal condition implies
\begin{align*}
	0\leq \Phi_{n}&=\frac{2v_{k}v_{kn}}{|\tilde{\nabla}v|^{2}}+4e^{4(\lambda+v)}|\tilde{\nabla}v|^{2}v_{n}+e^{4(\lambda+v)}2v_{k}v_{kn}\\
	&=\Big(\frac{1}{|\tilde{\nabla}v|^{2}}+e^{4(\lambda+v)}\Big)2v_{\be}v_{\be n}\\
	&=\Big(\frac{1}{|\tilde{\nabla}v|^{2}}+e^{4(\lambda+v)}\Big)\frac{\tilde{Q}_{\ga \be n}\s_{\ga}\s_{\be}}{\s^{2}}<0.\
\end{align*}
This yields a contradiction. Hence $\xi_{0} \in  \cc \setminus\pa \cc $.

Case \uppercase\expandafter{\romannumeral2}. $\xi_{0} \in \cc \setminus\pa \cc $.  Assume at this point $W$ is diagonal where $W_{ij}=v_{ij}+v_{i}v_{j}-\frac{1}{2}\tilde{Q}_{ijk}v_{k}+\delta_{ij}$ and w.l.o.g. $|v_{1}|\geq\frac{1}{\sqrt{n}}|\tilde{\nabla}v|>0$. At $\xi_{0}$, we have
\begin{align*}
	0=v_{1}\Phi_{1}&=\frac{2v_{1}v_{k}v_{k1}}{|\tilde{\nabla}v|^{2}}+4e^{4(\lambda+v)}|\tilde{\nabla}v|^{2}v_{1}^{2}+e^{4(\lambda+v)}2v_{1}v_{k}v_{k1}\\
	&=\frac{2v_{1}v_{k}(W_{1k}-v_{1}v_{k}+\frac{1}{2}\tilde{Q}_{1kl}v_{l}-\delta_{1k})}{|\tilde{\nabla}v|^{2}}+4e^{4(\lambda+v)}|\tilde{\nabla}v|^{2}v_{1}^{2}\\
	&+e^{4(\lambda+v)}2v_{1}v_{k}(W_{1k}-v_{1}v_{k}+\frac{1}{2}\tilde{Q}_{1kl}v_{l}-\delta_{1k})\\
	&\geq 2e^{4(\lambda+v)}v_{1}^{2}|\tilde{\nabla}v|^{2}-C|\tilde{\nabla}v|^{3},
\end{align*}
which implies that
$$|\tilde{\nabla}v|\leq C.\ $$
\end{proof}

\subsection{$C^{2}$ estimate}
In this subsection, we establish the a priori $C^{2}$ estimate for the admissible solution of Eq. \eqref{rewrite}.
\begin{lemma}\label{C2}
Consider an admissible solution $\hat{s}$ to Eq. \eqref{equ:p pde-inv-Gauss}. Let the exponent $p$ and the contact constant $\omega_0$ satisfy one of the following conditions:
	\begin{enumerate}
		\item$p=1$ or $p\geq n+1$,and $\omega_0\in\bigl(-F(E_{n+1}),\,F(-E_{n+1})\bigr)$ together with Condition \eqref{condition};
		\item$1<p<n+1$ and $\omega_0\in\bigl(-F(E_{n+1}),\,0\bigr]$ together with Condition \eqref{condition}.
	\end{enumerate}

Then there exists a constant $C$ depending on $n,p,\omega_{0},\Vert F\Vert_{C^{4}(\W)},\Vert\tilde{f}\Vert_{C^{2}(\cc)}$, such that
\begin{align}\label{equ:c2}
	\underset{\cc }{\max}|\tilde{\nabla}^{2} \tilde{s}|\leq C.
\end{align}
Furthermore, if $n+1\leq p \leq n+2$, the constant $C$ in \eqref{equ:c2} is independent of $p$.
\end{lemma}
\begin{proof}
We consider the following auxiliary function
$$P(\xi)=\log \sigma_{1}+e^{\be(m_{2}-\s)},$$
for $\xi\in \cc $. Suppose that $P(\xi)$ attains its maximum at some point $\xi_{0} \in \cc $. We divide the proof into two cases according to whether $\xi_{0}$ is an interior point or not.

Case \uppercase\expandafter{\romannumeral1}.\ $\xi_{0} \in \cc \setminus\pa \cc $.  In this case, we choose an orthonormal frame  $\{e_i\}_{i=1}^n$ around $\xi_{0}$, such that $\tilde{\tau}$ is diagonal. Denote
\begin{align}\label{c2-1}
	\mathcal{F}(\tilde{\tau}_{ij}):= \log \det({\tilde{\tau}_{ij}})= \log (\tilde{f}\s^{p-1})= f^{*},
\end{align}
and $\tilde{\tau}^{ij}$ be the inverse matrix of $\tilde{\tau}_{ij}$. Then,
\begin{align}
	\f^{ij}=\tt^{ij},\ \ \ \f^{ij,kl}=-\tt^{ik}\tt^{jl}.\
\end{align}
Clearly, $\f$ is also diagonal at $\xi_{0}$. By taking the first and second covariant derivatives of \eqref{c2-1}, we obtain
\begin{align}\label{c2-2}
	\f^{ij}\tt_{ij,k}=f^{*}_{,k},\ \ \ \f^{ij}\tt_{ij,kk}+ \f^{ij,rs}\tt_{ij,k}\tt_{rs,k}=f^{*}_{,kk}.
\end{align}

Then maximal condition implies that
\begin{equation}\label{c2-3}
	P_{i}=\frac{\sigma_{1,i}}{\sigma_{1}}-\be e^{\be(m_{2}-\s)}\s_{i}=0,
\end{equation}
and
\begin{align}
	0\geq\f^{ij}P_{ij}=\frac{\sum\limits_{i,k}^{n}\f^{ii}\tt_{kk,ii}}{\sigma_{1}}-\frac{\sum\limits_{i,k}^{n}\f^{ii}\tt_{kk,i}^{2}}{\sigma_{1}^{2}}+\sum\limits_{i}^{n}\f^{ii}e^{\be(m_{2}-\s)}(\be^{2}\s_{i}^{2}-\be\s_{ii}).\
\end{align}
Nest we estimate the terms $\sum\limits_{i,k}^{n}\f^{ii}\tt_{kk,ii}$, $\sum\limits_{i,k}^{n}\f^{ii}\tt_{kk,i}^{2}$ and $-\sum\limits_{i}^{n}\f^{ii}e^{\be(m_{2}-\s)}\be\s_{ii}$.\
The Ricci identity on $(\cc , \tilde{g})$ gives
\begin{align*}
\tt_{kk,ii}=\tt_{ii,kk}+\frac{1}{2}(\tilde{Q}_{iip}\tt_{kk,p}-\tilde{Q}_{kkp}\tt_{ii,p})+\tilde{g}_{ii}\tt_{kk}-\tilde{g}_{kk}\tt_{ii}+\tilde{Q}\ast \tilde{Q}\ast \tt+\tilde{\nabla}\tilde{Q}\ast \tt,
\end{align*}
where the notation $\ast$ denotes scalar contraction of two tensors by $\tilde{g}$, see also in \cite{W-X}*{Lemma A.1}. We see from the definition of $\tilde{Q}$ only depend on $F$ and $\omega_{0}$.
Using the above Ricci identity, \eqref{c2-2} and \eqref{c2-3}, we deduce
\begin{align*}
\f^{ii}\tt_{kk,ii}&=\frac{\Delta \tilde{f}}{\tilde{f}}-\frac{|\tilde{\nabla}f|^{2}}{\tilde{f}^{2}}+(p-1)\frac{\Delta \s}{s}-(p-1)\frac{|\tilde{\nabla}\s|^{2}}{\s^{2}}\\
&+\tt^{ii}\tt^{jj}\tt_{ij,k}^{2}+\frac{1}{2}\f^{ii}\tilde{Q}_{iik}\sigma_{1}\be e^{\be(m_{2}-\s)}\s_{k}-\frac{1}{2}\tilde{Q}_{kk\ell}f^{*}_{,\ell}\\
&+\f^{ii}\sigma_{1}+\f^{ii}(\tilde{Q}\ast \tilde{Q}\ast \tt+\tilde{\nabla}\tilde{Q}\ast \tt)-n^2\\
&\geq \tt^{ii}\tt^{jj}\tt_{ij,k}^{2}-C_{1}\f^{ii}(\sigma_{1}+1)+\frac{1}{2}\f^{ii}\tilde{Q}_{iik}\sigma_{1}\be e^{\be(m_{2}-\s)}\s_{k}.\
\end{align*}
It follows that
\begin{align*}
\frac{\f^{ii}\tt_{kk,ii}}{\sigma_{1}}\geq \frac{ \tt^{ii}\tt^{jj}\tt_{ij,k}^{2}}{\sigma_{1}}+\frac{1}{2}\f^{ii}\tilde{Q}_{ii\ell}\be e^{\be(m_{2}-\s)}\s_{\ell}-C_{2}(\f^{ii}+1).\
\end{align*}
Observe that
\begin{equation}\label{niu}
\tt_{kk,i}+\frac{1}{2}\q_{kk\ell}\tt_{i\ell}=\tt_{ik,k}+\frac{1}{2}\q_{ik\ell}\tt_{k\ell,}
\end{equation}
we have that
\begin{align*}
\f^{ii}\tt_{kk,i}^{2}&\leq \f^{ii}\Big(\tt_{ik,k}^{2}+|\q \ast \q| \sigma_{1}^{2}+2|\q|\sigma_{1}|\tt_{ik,k}|\Big)\\
&\leq \sum\limits_{i,j,k}^{n}\tt^{ii}(\tt_{jj}\tt^{jj}\tt_{ij,k}^{2})+C_{3}\f^{ii}\sigma_{1}^{2}+C_{4}\f^{ii}\sigma_{1}^{2}\be e^{\be(m_{2}-\s)}|\s_{i}|\\
&\leq \sum\limits_{l}^{n}\tt_{ll}\sum\limits_{i,j,k}^{n}\tt^{ii}\tt^{jj}\tt_{ij,k}^{2}\!+\!C_{3}\f^{ii}\sigma_{1}^{2}+\frac{1}{2}C_{5}\be^{\frac{3}{2}}e^{\be(m_{2}\!-\!\s)}\f^{ii}\sigma_{1}^{2}\s_{i}^{2}\!+\!\frac{1}{2}C_{6}\be^{\frac{1}{2}}e^{\be(m_{2}\!-\!\s)}\f^{ii}\sigma_{1}^{2},
\end{align*}
where the second inequality follows from \eqref{niu} and \eqref{c2-3}. Therefore, we get that
\begin{align*}
\frac{\f^{ii}\tt_{kk,i}^{2}}{\sigma_{1}^{2}}\leq  \frac{ \tt^{ii}\tt^{jj}\tt_{ij,k}^{2}}{\sigma_{1}}+C_{3}\f^{ii}+\frac{1}{2}C_{5}\be^{\frac{3}{2}}e^{\be(m_{2}-\s)}\f^{ii}\s_{i}^{2}+\frac{1}{2}C_{6}\be^{\frac{1}{2}}e^{\be(m_{2}-\s)}\f^{ii}.
\end{align*}
We also have
\begin{align*}
-\sum\limits_{i}^{n}\f^{ii}e^{\be(m_{2}-\s)}\be\s_{ii}&=-e^{\be(m_{2}-\s)}\be \f^{ii}(\tt_{ii}+\frac{1}{2}\q_{ii\ell}\s_{\ell}-\s)\\
&\geq m_{1}e^{\be(m_{2}-\s)}\be\f^{ii}-\frac{1}{2}\be e^{\be(m_{2}-\s)}\f^{ii } \q_{ii\ell}\s_{\ell}-C_{6}.
\end{align*}
Finally, we get that
\begin{align*}
0\geq\f^{ij}P_{ij}\geq e^{\be(m_{2}-\s)}\f^{ii}\s_{i}^{2}\Big(\be^{2}\!-\!\frac{C_{5}\be^{\frac{3}{2}}}{2}\Big)\!+\!e^{\be(m_{2}-\s)}\f^{ii}\Big(m_{1}\be\!-\!\frac{C_{6}\be^{\frac{1}{2}}}{2}\!-\!C_{3}\Big)\!-\!C_{6},
\end{align*}
and, choosing $\be$ large enough, we obtain
\begin{align}\label{c2-4}
\sum\limits_{i}^{n}\f^{ii}\leq C_{7},
\end{align}
where $C_{7}$ depends on $p, n, m_{1}, m_{2}, \Vert\tilde{f}\Vert_{C^{2}}, \Vert\s\Vert_{C^{1}}, \Vert F\Vert_{C^{5}}$.\

Recall the following elementary inequality (see \cite{Y-S-T}):
\begin{align*}
\frac{\sum\limits_{i}^{n}\lambda_{i}}{\prod\limits_{i}^{n}\lambda_{i}}\leq \Big(\sum\limits_{i}^{n}\lambda_{i}^{-1}\Big)^{n-1}\ \ \ for\  \lambda_{i}>0,\ \ \forall i=1,\cdots, n .
\end{align*}
In view of \eqref{c2-1} and \eqref{c2-4}, we have
$$\sigma_{1} \leq C_{7}^{n-1}\prod\limits_{i}^{n}\tt_{ii}\leq C.\ $$

Case \uppercase\expandafter{\romannumeral2}.\ $\xi_{0} \in \pa \cc $. We choose an orthonormal frame $\{e_{i}\}_{i=1}^{n}$ around $\xi_{0}$ satisfying $e_{n}= \frac{\mu_{F}}{|\mu_{F}|_{\tilde{g}}}$ and denote $\tt_{nn}=\max \{\tt_{ii}| i=1,\cdots, n\}>0$. From \eqref{c2-2}, we have
\begin{align*}
f^{*}_{,n}=\f^{ii}\tt_{ii,n}&=\f^{nn}\tt_{nn,n}+\f^{\al \al}\tt_{\al \al,n}\\
&=\f^{nn}\tt_{nn,n}+\f^{\al \al}\q_{\al \al n}(\tt_{\al \al}-\tt_{nn}),\
\end{align*}
which implies that
\begin{align*}
\tt_{nn,n}=\frac{f^{*}_{,n}}{\f^{nn}}-\frac{\f^{\al \al}}{\f^{nn}}\q_{\al \al n}(\tt_{\al \al}-\tt_{nn}).
\end{align*}
%If the maximum of $\Omega$ is attained at a boundary point, then at $\xi_{0}$ we have
Since $0\leq \sigma_1\cdot\tilde{\nabla}_nP(\xi_0)=\sigma_{1,n}$, at $\xi_0$ we have
\begin{align}
0\leq \f^{nn}\sigma_{1,n}&\leq f^{*}_{,n}-\f^{\al \al}\q_{\al \al n}(\tt_{\al \al}-\tt_{nn})+\f^{nn}\q_{\al \al n}(\tt_{\al \al}-\tt_{nn})
\nonumber\\
&=f^{*}_{,n}+(-\q_{\al \al n})(\tt_{\al \al}-\tt_{nn})(\f^{\al \al}-\f^{nn})
\nonumber
\\
&=f^{*}_{,n}+(-\q_{\al \al n})\sum\limits_{i}(\tt_{ii}-\tt_{nn})(\f^{ii}-\f^{nn})\nonumber\\
&\leq f^{*}_{,n}+(-\q_{\al \al n})(n+1-\f^{nn}\sigma_{1}-\tt_{nn}\sum\limits_{i}\f^{ii})\nonumber\\
&\leq f^{*}_{,n}+(-\q_{\al \al n})(n+1-\f^{nn}\sigma_{1})  .\label{equ:pf=c2-1}
\end{align}
% From \eqref{condition} and \eqref{equ:tildeQ-Q}, we have $(n-1)$ matrix $(-\tilde{Q}_{\alpha\beta n})>0$ on $T_{\partial\cc}\partial\cc\times T_{\partial\cc}\partial\cc$.
 Therefore,
\begin{align}
\sigma_{1}\leq \frac{C_{8}}{\f^{nn}}+\frac{n+1}{\f^{nn}}.\
\end{align}
Next, we show that $\f^{nn}$ cannot be very small, we can see that
$$0\leq \f^{nn}\sigma_{1,n}\leq  f^{*}_{,n}+(-\q_{\al \al n})\Big(n+1-\f^{nn}\sigma_{1}-\frac{\sum\limits_{i}\f^{ii}}{\f^{nn}}\Big).\ $$
The arithmetic-geometric mean inequality implies
\begin{align}
	\label{equ:pf=c2-2}
	\sum\limits_{i=1}^{n}\f^{ii}\geq n(\det\tt_{ij})^{-\frac{1}{n}}\geq n\big(\underset{\cc}{\min} \tilde{f}^{-1}(\tilde{s}^{1-p})\big)^{\frac{1}{n}}\geq C_{9}.\
\end{align}
We can clearly see that if $\f^{nn}$ is too small, this contradicts the fact that the maximum is attained at the boundary point.
No matter where the maximum point is attained, for the admissible solution $\tilde{s}$, we have $$0 < 2\sigma_{2} = \left( \sum\limits_{i=1}^{n} \tilde{\tau}_{ii} \right)^{2} - \sum\limits_{i,j=1}^{n} |\tilde{\tau}_{ij}|^{2} ,$$
then we obtain  \eqref{equ:c2}.

If $p\in(n+1,n+2)$, from \eqref{equ:pf=c2-1} and \eqref{equ:pf=c2-2}, we can see  constant $C$ in \eqref{equ:c2} is independent of $p$.

Thus we have
completed the proof of Lemma \ref{C2}.
\end{proof}
\begin{proposition}\label{Prop:Cka-estamate}
Consider an admissible solution $\hat{s}$ to Eq. \eqref{equ:p pde-inv-Gauss}. Let the exponent $p$ and the contact constant $\omega_0$ satisfy one of the following conditions:
	\begin{enumerate}
		\item$p=1$ or $p\geq n+1$,and $\omega_0\in\bigl(-F(E_{n+1}),\,F(-E_{n+1})\bigr)$ together with Condition \eqref{condition};
		\item$1<p<n+1$ and $\omega_0\in\bigl(-F(E_{n+1}),\,0\bigr]$ together with Condition \eqref{condition}.
	\end{enumerate}
 For each integral $k\geq 1$ and $\ga \in (0, 1)$, there exists a constant $C$, depending on $p, n, k, \ga, \Vert f\Vert_{C^{k+1}(\cc)}$, and $\Vert F\Vert_{C^{k+3}(\cc)}$, such that
\begin{equation}\label{liu}
	\Vert\hat{s}\Vert_{C^{k+1, \ga}(\cc)}\leq C.\
\end{equation}
Furthermore, if $n+1< p< n+2$, the constant $C$ in \eqref{liu} is independent of $p$.
\end{proposition}
\begin{proof}
Combining Lemma \ref{C01}, Lemma \ref{C02}, Lemma \ref{C1} and Lemma \ref{C2}, we obtain
$$\Vert\hat{s}\Vert_{C^{2}(\cc)}\leq C.\ $$
Using the fully nonlinear second order uniformly elliptic equations with oblique derivative boundary condition theory (see, e.g., \cite{M-S}*{Theorem 1.1} ), one can get the $C^{2, \ga}$ Schauder estimate and higher  order estimate as \eqref{liu}. Hence we complete the proof.
\end{proof}

\section{Proof of Theorem \ref{thm:p=1} ($p=1$)}\label{section5}
In this section, we use the method of continuity as in
\cite{Cheng-Yau,Xia13,Lutwak1995,Ni} to complete the proof of
Theorem \ref{thm:p=1}.

%Set $\tau_{i j}[S]=S_{i j}-\frac{1}{2} Q_{i j k} S_k+S \delta_{i j}$ and $\sigma_n\left(u_{i j}\right)=\operatorname{det}\left(u_{i j}\right)$.
Let $L_{\hat{s}}$ be the linearized operator of $\hat{s}\mapsto \operatorname{det}\left(\tau_{i j}[\hat{s}]\right)$, namely,
\begin{align}\label{equ:Ls-p=1}
	L_{\hat{s}}(v):=\mathcal{G}^{i j}\left(v_{i j}+v \mathring{g}_{i j}-\frac{1}{2}{Q}_{ijk}v_k\right)=\mathcal{G}^{i j}\cdot \tau_{ij}[v],
\end{align}
for any $v \in C^{\infty}(\mathcal{C}_{\omega_{0}})$. Here   $\mathcal{G}^{i j}:=\frac{\partial \det({\tau[\hat{s}]})}{\partial  {\tau}_{i j}}$, $v_k=\mathring{\nabla}_kv,v_{ij}=\mathring{\nabla}_i\mathring{\nabla}_jv$,  $\tau_{ij}[v]=v_{ij}-\frac{1}{2}Q_{ijk}v_k+v\mathring{g}_{ij}$, and $\mathring{\nabla}$ is the Levi-Civita connection of $( \mathcal{C}_{\omega_0},\mathring{g})$ with respect to anisotropic metric $\mathring{g}$.
Hence, for constants $\left\{a_i\right\}_{i=1}^{n}$, the function $L(\xi)=\sum_{i=1}^n a_i G(\mathcal{T}^{-1}\xi)\left(\mathcal{T}^{-1}\xi, E_i\right)$ satisfies
$$
L_{i j}-\frac{1}{2} Q_{i j k} L_k+L \delta_{i j}=0 , \text{\ in\ } \mathcal{C}_{\omega_0},
$$
and
$$
\mathring{\nabla}_{\mu_{F}}L=\frac{\omega_0}{F(\nu)\<E_{n+1},\mu\>}L, \text{\ on\ } \partial\mathcal{C}_{\omega_0}.
$$
For the proof details, one can refer to \cite{Xia13}*{Eq. (4.1)} and \cite{Arxiv1}*{Lemma 5.1}.
Thus, if $\hat{s}$ is an anisotropic capillary convex function solving \eqref{equ:p=1 pde-inv-Gauss}, then for any constant $L$, $\hat{s}+L$ is also a solution. By Proposition 3.5 of \cite{arxiv2}, $\hat{s}$ and $\hat{s}+L$ correspond uniquely to anisotropic capillary convex bodies, which we denote by $\hat\Sigma_1$ and $\hat\Sigma_2$, respectively. These bodies are related by a translation: $\Sigma_2 = \Sigma_1 + \sum_{i=1}^{n} a_i E_i$ for some coefficients $a_i$. If the origin is not an interior point of $\widehat{\partial\Sigma_1} \subset \mathbb{R}^n$, we may translate $\Sigma_1$ by a suitable vector $\sum_{i=1}^{n} a_i E_i \in \mathbb{R}^n$, so that the translated domain $\widehat{\partial\Sigma_2}$ contains the origin in its interior.
\begin{remark}\label{rk:O-inter}
	Such an observation allows us to assume, without loss of generality, that the anisotropic capillary hypersurface ${\Sigma
}$ determined by the solution of Problem \eqref{equ:p=1 pde-inv-Gauss} satisfies that $\widehat{\partial\Sigma}$  contains the origin as an interior point.
\end{remark} %, after which its anisotropic support function $\hat{s}$ becomes positive.
 %  to satisfy the following
%orthogonal condition:
%$$
%\int_{\mathcal{C}_{\omega_{0}}} G(\mathcal{T}^{-1}\xi)\left(\mathcal{T}^{-1}\xi, E_i\right) \cdot \hat{s}(\xi) ~d \mu_F=0 \quad \forall i=1,2, \ldots, n .
%$$

%This orthogonal condition means that the origin lies in the interior of the convex body enclosed by $\partial\Sigma\in \mathbb{R}^n=\partial\overline{\mathbb{R}^{n+1}_+}$.

%Define
%\begin{align}
%	\mathcal{S}=\left\{
%	t\in[0,1]:\det()
%	\right\}
%\end{align}

Let
$$
f_t:=(1-t) +t {f}, \quad \text { for } \quad 0 \leq t \leq 1.\
$$
Consider the problem
\begin{equation}\label{rewrite-0}
	\renewcommand{\arraystretch}{1.5}
	\left\{
	\begin{array}{rll}
		\det %\left(\hat{s}_{ij}-\frac{1}{2}Q_{ijk}\hat{s}_{k}+\hat{s}\hat{g}_{ij}\right)&=f_t\hat{s}^{p-1},  &\text{in\ } \mathcal{C}_{\omega_0},
		\left(\tau[\hat{s}]\right)&=f_t,  &\text{in\ }  \mathcal{C}_{\omega_0},
		\\
%	\label{equ:boundary equ s p=}
		\mathring{\nabla}_{\mu_F}\hat{s}&=\frac{\omega_0}{F(\nu)\<\mu,E_{n+1}\>}\hat{s}, & \text{on}\  \partial\mathcal{C}_{\omega_0}.
	\end{array}	
	\right.
\end{equation}
%$$
%\begin{aligned}
%\operatorname{det}\left(\nabla^2 h+h \sigma\right) & =h^{p-1} f_t, & & \text { in } \mathcal{C}_\theta \\
%\nabla_\mu h & =\cot \theta h, & & \text { on } \partial \mathcal{C}_\theta
%\end{aligned}
%$$
%Where on the boundary, $e_n$ is  orthogonal to the tangent plane $T(\partial\mathcal{C}_{\omega_{0}})$  with respect to the metric $\mathring{g}$.
Define the set
\begin{align*}
	\mathcal{H}:=\left\{\hat{s} \in C^{4}\left( \mathcal{C}_{\omega_{0}}\right): \mathring{\nabla}_{\mu_F}\hat{s}=\frac{\omega_0}{F(\nu)\<\mu,E_{n+1}\>}\hat{s} \text { on } \partial \mathcal{C}_{\omega_0}, \text{\ and\ }\tau[\hat{s}]>0\text { in }  \mathcal{C}_{\omega_0}\right\}.
\end{align*}
We denote
$$
\mathcal{I}:=\{t \in[0,1]: \text { Eq. \eqref{rewrite-0} has a  solution in } \mathcal{H}\},
$$
We rewrite Eq. \eqref{rewrite-0} as
\begin{align}\label{equ:Gtau=sp-1ft}
	\mathcal{G}({\tau[\hat{s}]}):=\operatorname{det}({\tau[\hat{s}]})=
 f_t,
\end{align}
and the linearized operator $L_{\hat{s}}$ of Eq. \eqref{equ:Gtau=sp-1ft} is defined by \eqref{equ:Ls-p=1}.
 %Since $\hat{s}\in\mathcal{H}$ %is a anisotropic  capillary  convex function (see Definition \ref{defn-convex capillary fun}),
%Suppose $\hat{s}$ satisfies $\tau[\hat{s}]>0$,
%then $\mathcal{G}^{i j}$ is a positively definite matrix.
Next, the following lemma implies that the kernel of operator $L_{\hat{s}}$ contains only the functions in span $\left\{
G( (\xi))(\mathcal{T}^{-1}\xi,E_{i})
\right\}_{i=1}^n$ with $\xi\in\mathcal{C}_{\omega_0}$, see \cite{arxiv2}*{Lemma 4.1}.
\begin{lemma}
	\label{lemma:ker-A}
	Let $f(\xi)\in C^2(\mathcal{C}_{\omega_0})$ be a function such that $\tau[f]=0$. Then
	\begin{align}\label{equ:f=1-n+1}
		f=\sum_{\alpha=1}^{n+1}a_{\alpha}G(\mathcal{T}^{-1}\xi)(\mathcal{T}^{-1}\xi,E_{\alpha}),
	\end{align}
	for some constants $a_1,\cdots,a_{n+1}$.
	Furthermore, if $f$ also satisfies the boundary condition \eqref{robin}, then $a_{n+1}=0$, that is
	\begin{align*}%\label{equ:f=1-n+n}
		f=\sum_{i=1}^{n}a_{i}G(\mathcal{T}^{-1}\xi)(\mathcal{T}^{-1}\xi,E_{i}),
	\end{align*}
	for some constants $a_1,\cdots,a_{n}$.
	Here $\left\{E_i\right\}_{i=1}^n$ are the horizontal coordinate unit vectors of $\overline{\mathbb{R}_{+}^{n+1}}$.\
\end{lemma}
The above lemma shows that: %the kernels of $L_{\hat{s}}$ are only those given by \eqref{equ:f=1-n+1}.
 %Denote $z=\mathcal{T}^{-1}(\xi)\in\W$ then we have
\begin{lemma}\label{lemma:KerL-p=1}
%	Let $p =1$. % and $[-\tilde{Q}_{n\alpha\beta}]_{(n-1)(n-1)}$ is strictly positive definite on $T_{\xi}(\partial \mathcal{C}_{\omega_0})$.
	Suppose $v \in \mathcal{H}$ and $L_{\hat{s}}(v)=0$, where $\hat{s}_{ij}-\frac{1}{2}Q_{ijk}\hat{s}_{k}+\mathring{g}_{ij}\hat{s} > 0$. Then $$v(\xi) =\sum_{i=1}^{n} a_iG(\mathcal{T}^{-1}\xi)(\mathcal{T}^{-1}\xi,E^i),$$
	for some constants $a_1,\cdots,a_n$.
\end{lemma}

We now prove that $L_{\hat{s}}$ is a self-adjoint operator, a key property for the subsequent analysis.

%By \cite{Arxiv1}*{Lemma 3.9}, we obtain:
\begin{lemma}
	\label{lemma:Lwv=Lvw}
	For any $\hat{s},v,w\in \mathcal{H}$, there holds
	\begin{align*}
		\int_{ \mathcal{C}_{\omega_0}}wL_{\hat{s}}v~d\mu_F=	\int_{ \mathcal{C}_{\omega_0}}vL_{\hat{s}}w~d\mu_F.
	\end{align*}
\end{lemma}
\begin{proof}
Suppose $d\mu_{{F}}=\varphi ~d\mu_{\mathring{g}}$, then $$\varphi_i=-\frac{1}{2}\varphi \sum_{k=1}^{n}{Q}_{ikk},\ $$
see \cite{Xia13}*{Lemma 2.8}. The \cite{Xia13}*{Eq. (5.10)} implies that
	\begin{align}
		\label{equ:Xia-13-5.10-p=1}
		\mathring{\nabla}_j\mathcal{G}^{ij}\cdot v_i\cdot\varphi+\mathcal{G}^{ij}\cdot\left(
		v_i\varphi_j+\frac{1}{2}Q_{ijk}  v_k \varphi
		\right)
		=0.\
	\end{align}
	By applying integration by parts yields	
	\begin{align}
		& \int_{\mathcal{C}_{\omega_0}} w L_{\hat{s}}(v) ~d \mu_F=\int_{\mathcal{C}_{\omega_0}} w \mathcal{G}^{ij}\left(v_{i j}-\frac{1}{2} Q_{i j k} v_k+v \mathring{g}_{i j}\right) \varphi ~d\mu_{\mathring{g}} \nonumber
		\\
		=&- \int_{\mathcal{C}_{\omega_0}} \mathcal{G}^{ij} w_i v_j \varphi ~d\mu_{\mathring{g}}+\int_{\mathcal{C}_{\omega_0}} \mathcal{G}^{ij} \mathring{g}_{i j} w v \varphi ~d\mu_{\mathring{g}}+\int_{\partial\mathcal{C}_{\omega_0}}\mathcal{G}^{ij}\cdot w\cdot v_j \varphi\frac{(\mu_{F})^{i}}{|\mu_{F}|_{\mathring{g}}}~ds
		\nonumber
		\\
		& -\int_{\mathcal{C}_{\omega_0}} w\left\{\left(\mathcal{G}^{ij}\right)_j v_i \varphi+\mathcal{G}^{ij}\left(v_i \varphi_j+\frac{1}{2} Q_{i j k} v_k \varphi\right)\right\} ~d\mu_{\mathring{g}}
		\nonumber
		\\
		=&\int_{\mathcal{C}_{\omega_0}} \mathcal{G}^{ij}(\mathring{g}_{i j} w v- w_i v_j )~d \mu_F
		+\int_{\partial\mathcal{C}_{\omega_0}}\mathcal{G}^{nn}  w\varphi
		v\frac{\omega_{0}}{|\mu_{{F}}|_{\mathring{g}} F(\nu)\<\mu,E_{n+1}\>}~ds,\  %\mathring{\nabla}_{\mu_{F}}v ~ds\ ,
	\label{equ:pf-Lemma-vLw=wLv}
	\end{align}
where we use the fact $\tau_{\al n}=0$, see \cite{Arxiv1}*{Remark 3.1}. Then we complete the proof of Lemma \ref{lemma:Lwv=Lvw}.
	\end{proof}
The following is an immediate consequence of Lemma \ref{lemma:Lwv=Lvw}. It also generalizes the necessary condition  \eqref{equ:int-fE1--En}.
\begin{corollary}
	\label{cor:s-in-(kerL)*}
	For any $\hat{s}\in\mathcal{H}$, %we have
	 there holds
	 \begin{align*}
	 	\int_{\mathcal{C}_{\omega_0}}G(z)(z,E_i)\det\left(\tau[\hat{s}]\right)~d\mu_F=0,\quad 1\leq i\leq n.
	 \end{align*}
\end{corollary}
\begin{proof}[\textbf{Proof of Theorem \ref{thm:p=1}}]
	%We use the degree theory for nonlinear elliptic equation with nonlinear oblique boundary conditions, developed by	Y. Y. Li \cite{LiYanYan-degree-oblique} to prove the existance of solutions.
	First we show the existence of Eq. \eqref{equ:p=1 pde-inv-Gauss}. For $t=0, \hat{s}=\ell \in \mathcal{H}$ is a solution, i.e., $0 \in \mathcal{I}$, so the set $\mathcal{I}$ is nonempty. Next we will apply the implicit function theorem to show the openness of $\mathcal{I}$. %while the closedess follows from Proposition \ref{Prop:Cka-estamate}.	
	Let $W^{m,2}$ be the Sobolev space of $(\mathcal{C}_{\omega_{0}},\mathring{g})$. Choose $m$ sufficiently large such that $W^{m,2}(\mathcal{C}_{\omega_{0}})\subset C^{4}(\mathcal{C}_{\omega_{0}})$. Consider $L_{\hat{s}}$ as a bounded linear map from $W^{m+2,2}(\mathcal{C}_{\omega_{0}})$ to $W^{m,2}(\mathcal{C}_{\omega_{0}})$. Lemma \ref{lemma:Lwv=Lvw} and Lemma \ref{lemma:KerL-p=1} tell that
	\begin{align*}
	\mathrm{Image}(L_{\hat{s}})=\mathrm{Ker}(L_{\hat{s}}^*)^{\perp}=\left(\mathrm{Span}\{G(z)(z,E_1),\cdots,G(z)(z,E_n)\}\right)^{\perp},
	\end{align*}
	which implies for any $f\in W^{m,2}(\mathcal{C}_{\omega_0})$  satisfying \eqref{equ:int-fE1--En}, we have $f\in\mathrm{ Image}(L_{\hat{s}}) $, that is, $L_{\hat{s}}$ is surjective. The standard implicit function theorem yields that $\mathcal{G}$ is locally
	invertible around $\hat{s}$, which in turn gives the openness of $\mathcal{I}$.	
	
	While the closedess follows from Proposition \ref{Prop:Cka-estamate}, we conclude that $\mathcal{I}=[0,1]$. Hence we prove the existence of an admissible solution to the anisotropic capillary Minkowski problem.
	
	We now turn to the uniqueness part.
	%For the uniqueness part,
	 Applying Lemma \ref{lemma:KerL-p=1} and Lemma \ref{lemma:Lwv=Lvw} and following the argument in \cite{Xia13}, we find that for two anisotropic capillary hypersurfaces $\Sigma$ and $\Sigma'$ with the same Gauss-Kronecker curvature, $$\Sigma(\xi)-\Sigma'(\xi)=\sum_{i=1}^na_iE^i\in\partial\overline{\mathbb{R}^{n+1}_+},$$ which means $\Sigma$ is unique up to horizontal translations.
\end{proof}

\section{Proof of Theorem \ref{thm:p>1} ($p>1$)}\label{section6}
We are now ready to prove Theorem \ref{thm:p>1}. Let
$$
f_t:=(1-t) +t \tilde{f}, \quad \text { for } \quad 0 \leq t \leq 1 ,
$$
where $\tilde{f}=f\ell^{p-1}$.
Consider the problem
\begin{equation}\label{rewrite-1}
	\renewcommand{\arraystretch}{1.5}
	\left\{
	\begin{array}{lll}
		\det \left(\tilde{s}_{ij}-\frac{1}{2}\tilde{Q}_{ijk}\tilde{s}_{k}+\tilde{s}\tilde{g}_{ij}\right)&=f_t\tilde{s}^{p-1},  &\text{in\ } \mathcal{C}_{\omega_0},
		\\
		\ \ \ \ \ \ \ \ \ \ \ \ \ \ \ \ \ \ \ \ \ \  \tilde{\nabla}_{e_n} \s&=0, & \text{on}\ \partial\mathcal{C}_{\omega_0}.
	\end{array}	
	\right.
\end{equation}
%$$
%\begin{aligned}
%\operatorname{det}\left(\nabla^2 h+h \sigma\right) & =h^{p-1} f_t, & & \text { in } \mathcal{C}_\theta \\
%\nabla_\mu h & =\cot \theta h, & & \text { on } \partial \mathcal{C}_\theta
%\end{aligned}
%$$

Define the set
$$
\mathcal{H}:=\left\{h \in C^{4, \alpha}\left(\mathcal{C}_{\omega_{0}}\right): \tilde{\nabla}_{e_n} h=0 \text { on } \partial \mathcal{C}_{\omega_0} \right\},
$$
and when $p>n+1$, we denote
$$
\mathcal{I}:=\{t \in[0,1]: \text { Eq. \eqref{rewrite-1} has a positive solution in } \mathcal{H}\},
$$
when $1<p<n+1$, we denote
$$
\mathcal{I}:=\{t \in[0,1]: \text { Eq. \eqref{rewrite-1} has a positive, anisotropic capillary even solution in } \mathcal{H}\} .
$$
We rewrite Eq. \eqref{rewrite-1} as
\begin{align}\label{equ:Gtau=sp-1ft-1}
	\mathcal{G}(\tilde{\tau}):=\operatorname{det}(\tilde{\tau})=\tilde{s}^{p-1} f_t,
\end{align}
and the linearized operator $L_{\tilde{s}}$ of Eq. \eqref{equ:Gtau=sp-1ft-1} is
$$
L_{\tilde{s}}(v):=\mathcal{G}^{i j}\left(v_{i j}+v \tilde{g}_{i j}-\frac{1}{2}\tilde{Q}_{ijk}v_k\right)-(p-1) f_t h^{p-2} v,
$$
where $v_k=\tilde{\nabla}_kv,v_{ij}=\tilde{\nabla}_i\tilde{\nabla}_jv$, and  $\mathcal{G}^{i j}:=\frac{\partial \mathcal{G}(\tilde{\tau})}{\partial \tilde{\tau}_{i j}}$. Since $\tilde{s}$ is a anisotropic  capillary  convex function (see Definition \ref{defn-convex capillary fun}), we know that $\mathcal{G}^{i j}$ is a positively definite matrix. Next, we adopt the technique in \cite{Guan-Xia-2018-Lp-Christoffel}*{Section 4.3} and show that the kernel of operator $L_{\tilde{s}}$ is trivial when $p \neq n+1$.
\begin{lemma}\label{lemma:kerL=0}
	Let $p \neq n+1>1$. % and $[-\tilde{Q}_{n\alpha\beta}]_{(n-1)(n-1)}$ is strictly positive definite on $T_{\xi}(\partial \mathcal{C}_{\omega_0})$.
	Suppose $v \in \mathcal{H}$,  $v \in \operatorname{Ker}\left(L_{\tilde{s}}\right)$, $\tilde{\tau}[v]>0$, $\tilde{\tau}[\tilde{s}]>0$, and $\tilde{s}>0$, then $v \equiv 0$.
\end{lemma}
\begin{proof}
	Suppose $v \in \operatorname{Ker}\left(L_{\tilde{s}}\right)$, then
\begin{align}
	\label{equ:MW-4.3}
	\mathcal{G}^{i j}\cdot \left(v_{i j}+v \tilde{g}_{i j}-\frac{1}{2}\tilde{Q}_{ijk}v_k\right)-(p-1) \tilde{s}^{-1} \operatorname{det}\left(\tilde{\tau}\right) v=0 .
\end{align}
Multiplying \eqref{equ:MW-4.3} with $\tilde{s}$ and integrating over $\mathcal{C}_{\omega_0}$  we get
\begin{align}
 (p-1) \int_{\mathcal{C}_{\omega_{0}}}
v\cdot \operatorname{det}\left(\tilde{\tau}\right) d \mu_F=\int_{\mathcal{C}_{\omega_{0}}} \tilde{s}\cdot \mathcal{G}^{i j}\cdot \left(v_{i j}+v \tilde{g}_{i j}-\frac{1}{2}\tilde{Q}_{ijk}v_k\right)  d \mu_{{F}}. \label{equ:p-=1}
\end{align}
Similarity, denote $d\mu_F=\varphi d\mu_{\tilde{g}}$ in $(\cc,\tilde{g})$,  \eqref{equ:Xia-13-5.10-p=1} implies
\begin{align}
	\label{equ:Xia-13-5.10}
	\tilde{\nabla}_j\mathcal{G}^{ij}\cdot v_i\cdot\varphi=-\mathcal{G}^{ij}\cdot\left(
	v_i\varphi_j+\frac{1}{2}\tilde{Q}_{ij}^k  v_k \varphi
	\right)
	=\frac{1}{2}\mathcal{G}^{ij}\cdot\left(\tilde{Q}_{jk}^kv_i-\tilde{Q}_{ij}^k  v_k
	\right)\cdot  \varphi.\
\end{align}
By applying Lemma \ref{pwf2}, we have $0=\tilde{\tau}_{\alpha n}=\left(
\tilde{s}_{\alpha n}+\tilde{s} \tilde{g}_{\alpha n}-\frac{1}{2}\tilde{Q}_{\alpha nk}\tilde{s}_k
\right)$ for $e_{\alpha}\in T(\partial\mathcal{C}_{\omega_{0}})$. Together with boundary condition, we have
\begin{align}
	\label{equ:pf-boundary-int=0}
	\int_{\partial\mathcal{C}_{\omega_0}}\mathcal{G}^{ij}\cdot \tilde{s}\cdot v_i\cdot \tilde{g}_{jn}\varphi	-\mathcal{G}^{ij}\cdot \tilde{s}_j\cdot v\cdot \tilde{g}_{in}\varphi~ds=0.
\end{align}
Integrating by parts twice leads to
\begin{align*}
	&\int_{\mathcal{C}_{\omega_0}}\tilde{s}\cdot \mathcal{G}^{ij}\cdot v_{ij} ~d\mu_F=\int_{\mathcal{C}_{\omega_0}}\tilde{s}\cdot \mathcal{G}^{ij}\cdot v_{ij}\cdot  \varphi ~d\mu_{\tilde{g}}
	\\
	=&\int_{\mathcal{C}_{\omega_0}}-\tilde{s}_j\cdot \mathcal{G}^{ij}\cdot v_{i}\cdot  \varphi
	-\tilde{s}\cdot \tilde{\nabla}_j\mathcal{G}^{ij}\cdot v_{i}\cdot  \varphi
	-\tilde{s}\cdot \mathcal{G}^{ij}\cdot v_{i}\cdot  \varphi_j
	~d\mu_{\tilde{g}}
	\\
	&+\int_{\partial\mathcal{C}_{\omega_0}}\mathcal{G}^{ij}\cdot \tilde{s}\cdot v_i\cdot \tilde{g}_{jn}\varphi~ds
	\\
	=&\int_{\mathcal{C}_{\omega_0}}
	\tilde{s}_{ij}\cdot \mathcal{G}^{ij}\cdot v\cdot  \varphi
	+\tilde{s}_j\cdot \tilde{\nabla}_i\mathcal{G}^{ij}\cdot v\cdot  \varphi
	+\tilde{s}_j\cdot \mathcal{G}^{ij}\cdot v\cdot  \varphi_i
	\\
	&-\tilde{s}\cdot \tilde{\nabla}_j\mathcal{G}^{ij}\cdot v_{i}\cdot  \varphi
	-\tilde{s}\cdot \mathcal{G}^{ij}\cdot v_{i}\cdot  \varphi_j
	~d\mu_{\tilde{g}}
	\\
	&+\int_{\partial\mathcal{C}_{\omega_0}}\mathcal{G}^{ij}\cdot \tilde{s}\cdot v_i\cdot \tilde{g}_{jn}\varphi
	-\mathcal{G}^{ij}\cdot \tilde{s}_j\cdot v\cdot \tilde{g}_{in}\varphi~ds
	\\
	\overset{\eqref{equ:Xia-13-5.10},\eqref{equ:pf-boundary-int=0}}{=}&\int_{\mathcal{C}_{\omega_0}}
	\tilde{s}_{ij}\cdot \mathcal{G}^{ij}\cdot v
	+\frac{1}{2}\cdot \mathcal{G}^{ij}\cdot\left(\tilde{Q}_{ik}^k\tilde{s}_j-\tilde{Q}_{ij}^k  \tilde{s}_k
	\right)\cdot v
	-\frac{1}{2}\tilde{s}_j\cdot \mathcal{G}^{ij}\cdot v\cdot  \tilde{Q}_{ik}^k
\\
&	-\tilde{s}\cdot \frac{1}{2}\mathcal{G}^{ij}\cdot\left(\tilde{Q}_{jk}^kv_i-\tilde{Q}_{ij}^k  v_k
	\right)
	+\frac{1}{2}\tilde{s}\cdot \mathcal{G}^{ij}\cdot v_{i}\cdot  \tilde{Q}_{jk}^k
	~d\mu_F
	\\
	=&\int_{\mathcal{C}_{\omega_0}}
	\tilde{s}_{ij}\cdot \mathcal{G}^{ij}\cdot v
	-\frac{1}{2}\cdot \mathcal{G}^{ij}\cdot\tilde{Q}_{ij}^k  \tilde{s}_k
	\cdot v
	+\tilde{s}\cdot \frac{1}{2}\mathcal{G}^{ij}\cdot\tilde{Q}_{ij}^k  v_k
	~d\mu_F.\
\end{align*}
Putting it into
\eqref{equ:p-=1} yields
\begin{align*}
	&(p-1) \int_{\mathcal{C}_{\omega_{0}}}
	v\cdot \operatorname{det}\left(\tilde{s}_{i j}+\tilde{s} \tilde{g}_{i j}-\frac{1}{2}\tilde{Q}_{ijk}\tilde{s}_k\right) d \mu_F
	\\
	=&\int_{\mathcal{C}_{\omega_{0}}} v\cdot \mathcal{G}^{i j}\cdot \left(\tilde{s}_{i j}+\tilde{s} \tilde{g}_{i j}-\frac{1}{2}\tilde{Q}_{ijk}\tilde{s}_k\right)  d \mu_{{F}}
	\\
	=&n\int_{\mathcal{C}_{\omega_{0}}}
	v\cdot \operatorname{det}\left(\tilde{s}_{i j}+\tilde{s} \tilde{g}_{i j}-\frac{1}{2}\tilde{Q}_{ijk}\tilde{s}_k\right) d \mu_F.\
\end{align*}
Since $p\neq n+1$, we have
\begin{align}
	\label{equ:pf-p-neq-n+1-exist-1}
	\int_{\mathcal{C}_{\omega_{0}}}
	v\cdot \operatorname{det}\left(\tilde{s}_{i j}+\tilde{s} \tilde{g}_{i j}-\frac{1}{2}\tilde{Q}_{ijk}\tilde{s}_k\right) d \mu_F=0.
\end{align}
As a special form of \cite{Arxiv1}*{Theorem 1.1}, we have
\begin{align*}
	&\left(\int_{\cc} v \operatorname{det}\left(\tilde{s}_{i j}+\tilde{s} \tilde{g}_{i j}-\frac{1}{2}\tilde{Q}_{ijk}\tilde{s}_k\right) d \mu_F\right)^2
	\\
	\geq& \int_{\cc} \tilde{s} \operatorname{det}\left(\tilde{s}_{i j}+\tilde{s} \tilde{g}_{i j}-\frac{1}{2}\tilde{Q}_{ijk}\tilde{s}_k\right) d \mu_F\cdot  \int_{\cc} v \mathcal{G}^{i j}\left(v_{i j}+v \tilde{g}_{i j}-\frac{1}{2}\tilde{Q}_{ijk}v_k\right) d \mu_F.
\end{align*}
Together with \eqref{equ:pf-p-neq-n+1-exist-1}, implies
\begin{align}
		\label{equ:pf-p-neq-n+1-exist-2}
	\int_{\cc} v \mathcal{G}^{i j}\left(v_{i j}+v \tilde{g}_{i j}-\frac{1}{2}\tilde{Q}_{ijk}v_k\right) d \mu_F \leq 0.
\end{align}
Similar to \eqref{equ:pf-p-neq-n+1-exist-1}, by multiplying \eqref{equ:p-=1} with $v$ and integrating over $\mathcal{C}_{\omega_0}$, we obtain
\begin{align}
		\label{equ:pf-p-neq-n+1-exist-3}
	&(p-1) \int_{\cc} \tilde{s}^{-1} v^2 \operatorname{det}\left(\tilde{s}_{i j}+\tilde{s} \tilde{g}_{i j}-\frac{1}{2}\tilde{Q}_{ijk}\tilde{s}_k\right) d \mu_F\nonumber
	\\
	=&\int_{\cc} v \mathcal{G}^{i j}\left(v_{i j}+v \tilde{g}_{i j}-\frac{1}{2}\tilde{Q}_{ijk}v_k\right) d \mu_F.
\end{align}

Due to $p>1$ and combining \eqref{equ:pf-p-neq-n+1-exist-2} and \eqref{equ:pf-p-neq-n+1-exist-3}, we derive
$$
\int_{\cc} \tilde{s}^{-1} v^2 \operatorname{det}\left(\tilde{s}_{i j}+\tilde{s} \tilde{g}_{i j}-\frac{1}{2}\tilde{Q}_{ijk}\tilde{s}_k\right) d \mu_F=0,
$$
which implies $v \equiv 0$.
\end{proof}
%We are now prove Theorem \ref{thm:p>1}.
\begin{proof}[\textbf{Proof of Theorem  \ref{thm:p>1}}]
 First, we show the existence parts of (1) and (3). Let $p>1$ and $p \neq n+1$. Consider Eq. \eqref{rewrite-1}. For $t=0, \tilde{s}\equiv 1 \in \mathcal{H}$ is a solution, i.e., $0 \in \mathcal{I}$, so the set $\mathcal{I}$ is nonempty. The closeness of $\mathcal{I}$ follows from Proposition \ref{Prop:Cka-estamate}, while the openness of $\mathcal{I}$ follows from the implicit function theorem, Lemma \ref{lemma:kerL=0}, and Fredholm's alternative theorem. Then it follows that $1 \in \mathcal{I}$, namely, we obtain a positive solution to Eq. \eqref{rewrite}.

We move on to prove (2), namely $p=n+1$. If $n+1< p< n+2$, the function $\bar{u} :=\frac{\tilde{s}}{\underset{\cc}{\min}\tilde{s}}$ satisfies
\begin{align*}
\det(\bar{u}_{ij}-\frac{1}{2}\tilde{Q}_{ijk}\bar{u}_{k}+\tilde{g}_{ij}\bar{u})&=(\underset{\cc}{\min}\tilde{s})^{p-(n+1)}\tilde{f}\bar{u}^{p-1},\ \ in\  \cc,\\
\tilde{\nabla}_{e_{n}}\bar{u}&= 0, \ \ \ \ \ \ \ \ \ \ \ \ \ \ \ \ \ \ \ \ \ \ \ \ \ \ \ on\  \pa \cc.
\end{align*}
Lemma \ref{C01} implies that $(\underset{\cc}{\min}\tilde{s})^{p-(n+1)}$ has a uniformly positive upper and low bounds. From Lemma \ref{logC1} and Lemma \ref{C2}, we have
$$1\leq \bar{u}\leq C,\ \ \ |\tilde{\nabla}\bar{u}|\leq C,\ \ \ \ \Vert\bar{u}\Vert_{C^{2}(\cc)}\leq C.$$
Therefore, applying the Evans-Krylov theorem and Schauder theory, we get that
\begin{equation}\label{pwf4}
\Vert\bar{u}\Vert_{C^{3,\ga}(\cc)}\leq C_{1},
\end{equation}
where $C_{1}$ independent of $p$.

For a fixed positive constant $\varepsilon \in(0,1)$, consider the approximating equation:
\begin{equation}\label{equ:rewrite=p=n+1}
	\renewcommand{\arraystretch}{1.5}
	\left\{
	\begin{array}{ll}
		\det \left(\tilde{s}_{ij}-\frac{1}{2}\tilde{Q}_{ijk}\tilde{s}_{k}+\tilde{s}\tilde{g}_{ij}\right)&=\tilde{f}\tilde{s}^{n+\varepsilon},  \text{\ in\ } \mathcal{C}_{\omega_0},
		\\
		\ \ \ \ \ \ \ \ \ \ \ \ \ \ \ \ \ \ \ \ \ \ {\tilde{\nabla}}_{e_n} \tilde{s}&=0,  \ \ \ \ \ \ \text{\ on}\ \partial\mathcal{C}_{\omega_0}.\
	\end{array}	
	\right.
\end{equation}
From the preceding discussion, we conclude that there exists a positive function $\tilde{s}_{\varepsilon} \in C^{3, \ga}\left(\mathcal{C}_\theta\right)$ that solves Eq. \eqref{equ:rewrite=p=n+1}. Denote
$$
\bar{u}_{\varepsilon}:=\frac{\tilde{s}_{\varepsilon}}{\underset{\cc}{\min}\tilde{s}_{\varepsilon}},
$$
we know ${\bar{u}_{\varepsilon}}$ satisfies the equation:
\begin{equation}
	\renewcommand{\arraystretch}{1.5}
	\left\{
	\begin{array}{ll}
		\det \left(\tilde{s}_{ij}-\frac{1}{2}\tilde{Q}_{ijk}\tilde{s}_{k}+\tilde{s}\tilde{g}_{ij}\right)&=(\underset{\cc}{\min}\tilde{s}_{\varepsilon})^{\varepsilon}\tilde{f}\tilde{s}^{n+\varepsilon},  \text{\ in\ } \mathcal{C}_{\omega_0},
		\\
		\ \ \ \ \ \ \ \ \ \ \ \ \ \ \ \ \ \ \ \ \ \ {\tilde{\nabla}}_{e_n} \tilde{s}&=0,  \ \ \ \ \ \ \ \ \ \ \ \ \ \ \ \ \ \text{\ on}\ \partial\mathcal{C}_{\omega_0}.\
	\end{array}	
	\right.
\end{equation}
\eqref{pwf4} implies
$$
\left\|\bar{u}_{\varepsilon}\right\|_{C^{3, \ga}(\cc)} \leq C,
$$
where the constant $C$ is independent of $\varepsilon$, and it only depends on $ n, k, \al, \Vert f\Vert_{C^{k+1}(\cc)}$, and $\Vert F\Vert_{C^{k+3}(\cc)}$. Hence, there exists a subsequence $\varepsilon_j \rightarrow 0$, such that $\bar{u}_{\varepsilon_j} \rightarrow \tilde{s}$ in $C^{2, \ga_0}\left(\mathcal{C}_{\omega_0}\right)$ for any $0<\ga_0< \ga$, and Lemma \ref{C01} implies that $\left(\underset{\cc}{\min}\tilde{s}_{\varepsilon_j}\right) \rightarrow \tilde{\eta}$ for some positive constant $\tilde{\eta}$ as $j \rightarrow+\infty$. Thus we conclude that the pair $(\tilde{s}, \tilde{\eta})$ solve

\begin{equation}\label{equ:rewrite=p=n+1-r-s}
	\renewcommand{\arraystretch}{1.5}
	\left\{
	\begin{array}{rll}
\det \left(\tilde{s}_{ij}-\frac{1}{2}\tilde{Q}_{ijk}\tilde{s}_{k}+\tilde{s}\tilde{g}_{ij}\right) & =\tilde{\eta} \tilde{f} \tilde{s}^n, &  \text { in } \mathcal{C}_{\omega_0}, \\
	\tilde{\nabla}_{e_n} \tilde{s} & =0, &  \text { on } \partial \mathcal{C}_{\omega_0} .
\end{array}
\right.
\end{equation}Therefore, we prove the existence part for the cases $p=n+1$.
The uniqueness proof of Theorem \ref{thm:p>1} is divided into three cases:
\begin{itemize}
    \item Case $p = n+1$: We establish the result following the ideas in \cite[Lemma 8]{Guan-Lin-preprint} and \cite[Section 3]{Chou-Wang}.
    \item Cases $1 < p < n+1$ and $p > n+1$: These can be handled analogously to the arguments in \cite[Section 4.3]{Guan-Xia-2018-Lp-Christoffel} and \cite[Section 3.2]{Mei-Wang-Weng-arXiv2504.14392} (see also \cite[Theorem 5.2]{Ding-Li-Tran}). The details are similar and thus omitted.
\end{itemize}

We now verify the uniqueness for the case $p = n+1$.
 %When $f$ is positive and continuous, an analytic proof can be given as follows.
 First, suppose that $ (\eta_{1}, \hat{s}_1 )$ and  $(\eta_{2}, \hat{s}_2)$  are two solutions of \eqref{equ:rewrite=p=n+1-r-s}. Denote $\Sigma_1$ and $\Sigma_2$ are the anisotropic capillary hypersurfaces determined by capillary convex functions $\hat{s}_1$ and $\hat{s}_2$, respectively. %for $p=n+1$ and $f_t=f$, .
 By multiplying $\hat{s}_2$ with a suitable constant, we may assume $\Sigma_2$ is contained inside $\Sigma_1$, with some point touching $\Sigma_1$. At this point, say, $\xi_0$, we have $\hat{s}_1(\xi_0)=\hat{s}_2(\xi_0)$, and
$$
\begin{aligned}
	\eta_1 f\left(\xi_0\right) \hat{s}_1^n\left(\xi_0\right) & =\det(\tau[\hat{s}_1]) \\
	& \geqslant \det(\tau[\hat{s}_2]) \\
	& =\eta_2 f\left(\xi_0\right) \hat{s}_2^n\left(\xi_0\right).
\end{aligned}
$$
Hence $\eta_1 \geqslant \eta_2$. By symmetry we get $\eta_1=\eta_2$.
Next, assume $\hat{s}_1$ and $\hat{s}_2$ are two different solutions of \eqref{equ:rewrite=p=n+1-r-s} with the same $\eta_0$. By multiplying $\hat{s}_2$ with a suitable constant we may assume the set $E=\left\{\xi \in \mathcal{C}_{\omega_0}: \hat{s}_1(\xi)>\right. \left.\hat{s}_2(\xi)\right\}$ is an open set  in the interior of $\mathcal{C}_{\omega_{0}}$. We can always do this when $\hat{s}_2$ is not a constant multiple of $\hat{s}_1$.
Let $\Omega$ be a connected component of $E$.
Then both $\hat{s}_1$ and $\hat{s}_2$ satisfy

\begin{equation*}
	\left\{\begin{array}{lll}
		\det(\tau[\hat{s}_i])=\gamma f \hat{s}_i^n, & \text { in } \Omega, \\ \hat{s}_1=\hat{s}_2, & \text { on } \partial \Omega, \\ \hat{s}_1>\hat{s}_2, & \text { in } \Omega .
	\end{array}
	\right.
\end{equation*}
But this is impossible by the comparison principle in \cite{Gilgarg-Trudinger-PDE-Book}*{Theorem 17.4}.
 %This completes the proof.
\end{proof}
%Inspired by \cite{Chou-Wang} (Existence)

%\cite{} (Unique)

\section*{Reference}
%\begin{bibdiv}

%\end{bibdiv}


\begin{biblist}
\bib{Al1}{article}{
	author={Alexandroff, A. D.},
	title={Theorie der gemisteten Volumina von konvexen K\"{o}rdern, II: Neue Ungleichungen zwischen den gemischten Volumina und ihre Anwendungen},
	journal={Mat. Sbornik N.S.},
	volume={2},
	date={1937},
	pages={1205–1238 (Russian)},
}

\bib{Al2}{article}{
	author={Alexandroff, A. D.},
	title={On the theory of mixed volumes. III. Extension of two theorems of Minkowski on convex polyhedra to arbitrary convex bodies},
	journal={Mat. Sbornik N.S.},
	volume={3},
	date={1938},
	pages={27–46 (Russian; German summary)},
}

\bib{Al3}{article}{
	author={Alexandroff, A. D.},
	title={On the surface area measure of convex bodies},
	journal={Mat. Sbornik N.S.},
	volume={6},
	date={1939},
	pages={167–174 (Russian; German summary)},
}

\bib{Al4}{article}{
	author={Alexandroff, A. D.},
	title={Smoothness of the convex surface of bounded Gaussian curvature},
	journal={C. R. (Doklady) Acad. Sci. URSS (N.S.)},
	volume={36},
	date={1942},
	pages={195--199},
	review={\MR{7626}},
}
\bib{And01}{article}{%MR1871390,
	title={Volume-Preserving Anisotropic Mean Curvature Flow},
	author={Andrews, B.},
	%journal={Indiana University Mathematics Journal},
	JOURNAL = {Indiana Univ. Math. J.},
	FJOURNAL = {Indiana University Mathematics Journal},
	volume={50},
	number={2},
	pages={783-827},
	year={2001},
}%12
\bib{Lucwak2013}{article}{
	author={B\"{o}r\"{o}czky, K. J.},
	author={Lutwak, E.},
	author={Yang, D.},
	author={Zhang, G.},
	title={The logarithmic Minkowski problem},
	journal={J. Amer. Math. Soc.},
	volume={26},
	date={2013},
	number={3},
	pages={831--852},
	issn={0894-0347},
	review={\MR{3037788}},
	doi={10.1090/S0894-0347-2012-00741-3},
}



\bib{Cal}{article}{
	author={Calabi, E.},
	title={Improper affine hyperspheres of convex type and a generalization
		of a theorem by K. J\"{o}rgens},
	journal={Michigan Math. J.},
	volume={5},
	date={1958},
	pages={105--126},
	issn={0026-2285},
	review={\MR{106487}},
}

\bib{C-N-S}{article}{
	author={Caffarelli, L.},
	author={Nirenberg, L.},
	author={Spruck, J.},
	title={The Dirichlet problem for nonlinear second-order elliptic
		equations. I. Monge-Amp\`ere equation},
	journal={Comm. Pure Appl. Math.},
	volume={37},
	date={1984},
	number={3},
	pages={369--402},
	issn={0010-3640},
	review={\MR{739925}},
	doi={10.1002/cpa.3160370306},
}
\bib{Chen-Li-Zhu-2017}{article}{
	author={Chen, S.},
	author={Li, Q.},
	author={Zhu, G.},
	title={On the $L_p$ Monge-Amp\`ere equation},
	journal={J. Differential Equations},
	volume={263},
	date={2017},
	number={8},
	pages={4997--5011},
	issn={0022-0396},
	review={\MR{3680945}},
	doi={10.1016/j.jde.2017.06.007},
}
\bib{Chen-anisotropic-Lp}{article}{
	author={Chen, Z.},
	title={{\it A priori} bounds, existence, and uniqueness of smooth
		solutions to an anisotropic $L_p$ Minkowski problem for log-concave
		measure},
	journal={Adv. Nonlinear Stud.},
	volume={23},
	date={2023},
	number={1},
	pages={Paper No. 20220068, 21},
	issn={1536-1365},
	review={\MR{4599957}},
	doi={10.1515/ans-2022-0068},
}


\bib{Cheng-Yau}{article}{
	author={Cheng, S. Y.},
	author={Yau, S. T.},
	title={On the regularity of the solution of the $n$-dimensional Minkowski
		problem},
	journal={Comm. Pure Appl. Math.},
	volume={29},
	date={1976},
	number={5},
	pages={495--516},
	issn={0010-3640},
	review={\MR{423267}},
	doi={10.1002/cpa.3160290504},
}


\bib{Chou-Wang}{article}{
	author={Chou, K. S.},
	author={Wang, X.},
	title={The $L_p$-Minkowski problem and the Minkowski problem in
		centroaffine geometry},
	journal={Adv. Math.},
	volume={205},
	date={2006},
	number={1},
	pages={33--83},
	issn={0001-8708},
	review={\MR{2254308}},
	doi={10.1016/j.aim.2005.07.004},
}

\bib{Ding-Li-arxiv-2024}{article}{
    AUTHOR = {Ding, S.},
    author={Li, G.},
     TITLE = {A type of anisotropic flows and dual {O}rlicz
              {C}hristoffel-{M}inkowski type equations},
   JOURNAL = {J. Differential Equations},
  FJOURNAL = {Journal of Differential Equations},
    VOLUME = {445},
      YEAR = {2025},
     PAGES = {Paper No. 113586, 34},
      ISSN = {0022-0396,1090-2732},
   MRCLASS = {53E99 (35K55)},
  MRNUMBER = {4925888},
       DOI = {10.1016/j.jde.2025.113586},
       URL = {https://doi.org/10.1016/j.jde.2025.113586},
}


\bib{Ding-Li-Tran}{article}{
	author={Ding, S.},
	author={Li, G.},
	title={A class of inverse curvature flows and $L^p$ dual
		Christoffel-Minkowski problem},
	journal={Trans. Amer. Math. Soc.},
	volume={376},
	date={2023},
	number={1},
	pages={697--752},
	issn={0002-9947},
	review={\MR{4510121}},
	doi={10.1090/tran/8793},
}


\bib{Arxiv1}{article}{
	title={Anisotropic mean curvature type flow and capillary Alexandrov-Fenchel inequalities},
	author={Ding, S.},
	author={Gao, J.},
	author={Li, G.},
	year={2025},
	eprint={arXiv:2408.10740v2 },
	
}
\bib{F-J}{article}{
	author={Fenchel, W.},
	author={Jessen, B.},
	title={Mengenfunktionen und konvexe K\"{o}rper},
	language={German},
	journal={Danske Vid. Selskab. Mat.-fys. Medd.},
	date={1938},
	number={16},
	pages={1–31},
}

%\bib{F-R}{book}{
	%AUTHOR = {Finn, R.},
	%TITLE = {Equilibrium capillary surfaces},
	%SERIES = {Grundlehren der mathematischen Wissenschaften [Fundamental
		%Principles of Mathematical Sciences]},
	%VOLUME = { 284 },
	%PUBLISHER = {Springer-Verlag, New York},
	%YEAR = {1986},
	%PAGES = {xvi+245},
	%ISBN = {0-387-96174-7},
	%MRCLASS = {49-02 (49F10 53-02 53A10 58E12)},
	%MRNUMBER = {816345},
	%MRREVIEWER = {Helmut\ Kaul},
	%DOI = {10.1007/978-1-4613-8584-4},
	%URL = {https://doi.org/10.1007/978-1-4613-8584-4},
%}

\bib{Arxiv}{article}{
	title={Generalized Minkowski formulas and rigidity results for anisotropic capillary hypersurfaces},
	author={Gao, J.},
	author={Li, G.},
	year={2024},
	eprint={2401.12137},
	%month=feb
}
\bib{arxiv2}{article}{
	title={Alexandrov-Fenchel inequalities for convex anisotropic capillary hypersurfaces in the half-space},
	author={Gao, J.}
	author= {Li, G.},
	year={2025},
	eprint={2507.04796},
	archivePrefix={arXiv},
	primaryClass={math.DG},
	url={https://arxiv.org/abs/2507.04796},
}
\bib{Ghomi}{article}{
	author={Ghomi, M.},
	title={Gauss map, topology, and convexity of hypersurfaces with
		nonvanishing curvature},
	journal={Topology},
	volume={41},
	date={2002},
	number={1},
	pages={107--117},
	issn={0040-9383},
	review={\MR{1871243}},
	doi={10.1016/S0040-9383(00)00028-8},
}
\bib{Gilgarg-Trudinger-PDE-Book}{book}{
	author={Gilbarg, D.},
	author={Trudinger, N.},
	title={Elliptic partial differential equations of second order},
	series={Classics in Mathematics},
	note={Reprint of the 1998 edition},
	publisher={Springer-Verlag, Berlin},
	date={2001},
	pages={xiv+517},
	isbn={3-540-41160-7},
	review={\MR{1814364}},
}









\bib{Guan-Lin-preprint}{article}{
	author={Guan, P.},
	author={Lin, C.},
	title={On equation  $\det(u_{ij}+\delta_{ij}u)=u^pf$ on $\mathbb{S}^{n}$},
	pages={preprint},
}

\bib{Guan-Xia-2018-Lp-Christoffel}{article}{
	author={Guan, P.},
	author={Xia, C.},
	title={$L^p$ Christoffel-Minkowski problem: the case $1<p<k+1$},
	journal={Calc. Var. Partial Differential Equations},
	volume={57},
	date={2018},
	number={2},
	pages={Paper No. 69, 23},
	issn={0944-2669},
	review={\MR{3776359}},
	doi={10.1007/s00526-018-1341-y},
}
\bib{Guo-Xia}{article}{
	author={Guo, J.},
	author={Xia, C.},
	title={Stable anisotropic capillary hypersurfaces in a half-space},
	year={2024},
	eprint={2301.03020},
	archivePrefix={arXiv},
	primaryClass={math.DG}
}
\bib{Lutwak2010}{article}{
	author={Haberl, C.},
	author={Lutwak, E.},
	author={Yang, D.},
	author={Zhang, G.},
	title={The even Orlicz Minkowski problem},
	journal={Adv. Math.},
	volume={224},
	date={2010},
	number={6},
	pages={2485--2510},
	issn={0001-8708},
	review={\MR{2652213}},
	doi={10.1016/j.aim.2010.02.006},
}

\bib{H-H-I}{article}{
	title={Capillary $L_p$ Minkowski Flows},
	author={Hu, J.},
	author={Hu, Y.},
	author={Ivaki, M.},
	%author={Jinrong Hu and Yingxiang Hu and Mohammad N. Ivaki},
	year={2025},
	eprint={2509.06110},
	archivePrefix={arXiv},
	primaryClass={math.AP},
	url={https://arxiv.org/abs/2509.06110},
}
\bib{H-I-S}{article}{
	title={Capillary Christoffel-Minkowski problem},
	author={Hu, Y.},
	author={Ivaki, M.},
	author={Scheuer, J.},
	year={2025},
	eprint={2504.09320},
	archivePrefix={arXiv},
	primaryClass={math.DG},
	url={https://arxiv.org/abs/2504.09320},
}







\bib{Lutwak2005}{article}{
	author={Hug, D.},
	author={Lutwak, E.},
	author={Yang, D.},
	author={Zhang, G.},
	title={On the $L_p$ Minkowski problem for polytopes},
	journal={Discrete Comput. Geom.},
	volume={33},
	date={2005},
	number={4},
	pages={699--715},
	issn={0179-5376},
	review={\MR{2132298}},
	doi={10.1007/s00454-004-1149-8},
}
\bib{Jia-Wang-Xia-Zhang2023}{article}{
	AUTHOR = {Jia, X.},
	AUTHOR = {Wang, G.},
	AUTHOR = {Xia, C.},
	AUTHOR = {Zhang, X.},
	title={Alexandrov's theorem for anisotropic capillary hypersurfaces in the half-space},
	journal={Arch. Ration. Mech. Anal. },
	VOLUME = {247},
	YEAR = {2023},
	NUMBER = {2},
	PAGES = {19-25},
}

\bib{Le1}{article}{
	author={Lewy, H.},
	title={On the existence of a closed convex surface realizing a given Riemannian metric},
	journal={Proc. Nat. Acad. Sci. USA},
	volume={24},
	date={1938},
	number={2},
	pages={104–106},
}


\bib{Le2}{article}{
	author={Lewy, H.},
	title={On differential geometry in the large. I. Minkowski's problem},
	journal={Trans. Amer. Math. Soc.},
	volume={43},
	date={1938},
	number={2},
	pages={258--270},
	issn={0002-9947},
	review={\MR{1501942}},
	doi={10.2307/1990042},
}
\bib{Li-Xu}{article}{%i2022hyperbolicpsumhorosphericalpbrunnminkowski,
	title={Hyperbolic $p$-sum and Horospherical $p$-Brunn-Minkowski theory in hyperbolic space},
	author={Li, H.},
	author={Xu, B.},
	year={2022},
	eprint={2211.06875},
	archivePrefix={arXiv},
	primaryClass={math.MG},
	url={https://arxiv.org/abs/2211.06875},
}
\bib{Lions-Trudinger-1986}{article}{
	author={Lions, P.-L.},
	author={Trudinger, N. S.},
	author={Urbas, J. I. E.},
	title={The Neumann problem for equations of Monge-Amp\`ere type},
	journal={Comm. Pure Appl. Math.},
	volume={39},
	date={1986},
	number={4},
	pages={539--563},
	issn={0010-3640},
	review={\MR{840340}},
	doi={10.1002/cpa.3160390405},
}

\bib{M-S}{article}{
	AUTHOR = {Lieberman, G.},
	AUTHOR = {Trudinger, N.},
	TITLE = {Nonlinear oblique boundary value problems for nonlinear
		elliptic equations},
	JOURNAL = {Trans. Amer. Math. Soc.},
	FJOURNAL = {Transactions of the American Mathematical Society},
	VOLUME = {295},
	YEAR = {1986},
	NUMBER = {2},
	PAGES = {509--546},
	ISSN = {0002-9947,1088-6850},
	MRCLASS = {35J65},
	MRNUMBER = {833695},
	MRREVIEWER = {Pierre-Louis\ Lions},
	DOI = {10.2307/2000050},
	URL = {https://doi.org/10.2307/2000050},
}
\bib{Lutwak1993}{article}{
	author={Lutwak, E.},
	title={The Brunn-Minkowski-Firey theory. I. Mixed volumes and the
		Minkowski problem},
	journal={J. Differential Geom.},
	volume={38},
	date={1993},
	number={1},
	pages={131--150},
	issn={0022-040X},
	review={\MR{1231704}},
}
\bib{Lutwak1995}{article}{
	author={Lutwak, E.},
	author={Oliker, V.},
	title={On the regularity of solutions to a generalization of the
		Minkowski problem},
	journal={J. Differential Geom.},
	volume={41},
	date={1995},
	number={1},
	pages={227--246},
	issn={0022-040X},
	review={\MR{1316557}},
}

\bib{Lutwark2012}{article}{
	author={Lutwak, E.},
	author={Yang, D.},
	author={Zhang, G.},
	title={The Brunn-Minkowski-Firey inequality for nonconvex sets},
	journal={Adv. in Appl. Math.},
	volume={48},
	date={2012},
	number={2},
	pages={407–413},
}




\bib{Ma-Qiu}{article}{
	author={Ma, X.},
	author={Qiu, G.},
	title={The Neumann problem for Hessian equations},
	journal={Comm. Math. Phys.},
	volume={366},
	date={2019},
	number={1},
	pages={1--28},
	issn={0010-3616},
	review={\MR{3919441}},
	doi={10.1007/s00220-019-03339-1},
}
%\bib{Ma-Xu}{article}{
%	author={Ma, X.},
%	author={Xu, J.},
%	title={Gradient estimates of mean curvature equations with Neumann
%		boundary value problems},
%	journal={Adv. Math.},
%	volume={290},
%	date={2016},
%	pages={1010--1039},
%	issn={0001-8708},
%	review={\MR{3451945}},
%	doi={10.1016/j.aim.2015.10.031},
%}

%\bib{Mei-Wang-Weng-Lp_Minkowski}{article}
%{AUTHOR = {Mei, Xinqun and Wang, Guofang and Weng, Liangjun and Xia, 		Chao}}

\bib{Xia-arxiv}{article}{
	AUTHOR = {Mei, X.},
       AUTHOR = {Wang, G.},
	AUTHOR = {Weng, L.},
	AUTHOR = {Xia, C.},
	TITLE = {Alexandrov-{F}enchel inequalities for convex hypersurfaces in
		the half-space with capillary boundary {II}},
	JOURNAL = {Math. Z.},
	FJOURNAL = {Mathematische Zeitschrift},
	VOLUME = {310},
	YEAR = {2025},
	NUMBER = {4},
	PAGES = {Paper No. 71, 17},
	ISSN = {0025-5874,1432-1823},
	MRCLASS = {52A39 (53C24 58J50)},
	MRNUMBER = {4911815},
	DOI = {10.1007/s00209-025-03781-z},
	URL = {https://doi.org/10.1007/s00209-025-03781-z},
}

\bib{Mei-Wang-Weng-arXiv2504.14392}{article}{
	title={Convex capillary hypersurfaces of prescribed curvature problem},
	author={Mei, X.},
	author={Wang, G.},
	year={2025},
	eprint={2504.14392},
	archivePrefix={arXiv},
	primaryClass={math.DG},
	url={https://arxiv.org/abs/2504.14392},
}

\bib{M-W-W}{article}{
	AUTHOR = {Mei, X.},
	AUTHOR = {Wang, G.},
	AUTHOR = {Weng, L.},
	TITLE = {The capillary {M}inkowski problem},
	JOURNAL = {Adv. Math.},
	FJOURNAL = {Advances in Mathematics},
	VOLUME = {469},
	YEAR = {2025},
	PAGES = {Paper No. 110230, 29},
	ISSN = {0001-8708,1090-2082},
	MRCLASS = {53C65 (35J60 35J66 35J96 53C42 53C45)},
	MRNUMBER = {4884106},
	MRREVIEWER = {Changyu\ Ren},
	DOI = {10.1016/j.aim.2025.110230},
	URL = {https://doi.org/10.1016/j.aim.2025.110230},
}

\bib{M-W-W1}{article}{
	title={The capillary $L_p$-Minkowski problem},
	author={Mei, X.},
	author={Wang, G.},
	author={Weng, L.},
	year={2025},
	eprint={2505.07746},
	archivePrefix={arXiv},
	primaryClass={math.DG},
	url={https://arxiv.org/abs/2505.07746},
}

\bib{Mi1}{article}{
	author={Minkowski, H.},
	title={Allgemeine Lehrs\"{a}tze \"{u}ber die konvexen Polyeder},
	language={German},
	journal={Nachr. Ges. Wiss. Göttingen},
	date={1897},
	number={4},
	pages={198–219},
	issn={0025-5831},
}


\bib{Mi2}{article}{
	author={Minkowski, H.},
	title={Volumen und Oberfl\"{a}che},
	language={German},
	journal={Math. Ann.},
	volume={57},
	date={1903},
	number={4},
	pages={447--495},
	issn={0025-5831},
	review={\MR{1511220}},
	doi={10.1007/BF01445180},
}

\bib{Ni}{article}{
	author={Nirenberg, L.},
	title={The Weyl and Minkowski problems in differential geometry in the
		large},
	journal={Comm. Pure Appl. Math.},
	volume={6},
	date={1953},
	pages={337--394},
	issn={0010-3640},
	review={\MR{58265}},
	doi={10.1002/cpa.3160060303},
}
\bib{P-M}{article}{
    AUTHOR = {De Philippis, G.},
    author= {Maggi, F.},
     TITLE = {Regularity of free boundaries in anisotropic capillarity
              problems and the validity of {Y}oung's law},
   JOURNAL = {Arch. Ration. Mech. Anal.},
  FJOURNAL = {Archive for Rational Mechanics and Analysis},
    VOLUME = {216},
      YEAR = {2015},
    NUMBER = {2},
     PAGES = {473--568},
      ISSN = {0003-9527,1432-0673},
   MRCLASS = {35R30 (28A75 35B65)},
  MRNUMBER = {3317808},
MRREVIEWER = {Antoine\ Henrot},
       DOI = {10.1007/s00205-014-0813-2},
       URL = {https://doi.org/10.1007/s00205-014-0813-2},
}

\bib{Po2}{article}{
	author={Pogorelov, A. V.},
	title={A regular solution of the $n$-dimensional Minkowski problem},
	language={Russian},
	journal={Dokl. Akad. Nauk SSSR},
	volume={199},
	pages={785--788},
	issn={0002-3264},
	translation={
		journal={Soviet Math. Dokl.},
		volume={12},
		date={1971},
		pages={1192--1196},
		issn={0197-6788},
	},
	review={\MR{284956}},
}
\bib{Po1}{article}{
	author={Pogorelov, A. V.},
	title={Regularity of a convex surface with given Gaussian curvature},
	language={Russian},
	journal={Mat. Sbornik N.S.},
	volume={31(73)},
	date={1952},
	pages={88--103},
	review={\MR{52807}},
}	





\bib{book-convex-body}{book}{
	author={Schneider, R.},
	title={Convex bodies: the Brunn-Minkowski theory},
	series={Encyclopedia of Mathematics and its Applications},
	volume={151},
	edition={Second expanded edition},
	publisher={Cambridge University Press, Cambridge},
	date={2014},
	pages={xxii+736},
	isbn={978-1-107-60101-7},
	review={\MR{3155183}},
}
%\bib{Shenfeld19}{article}{
%	author={Shenfeld, Y.},
%	author={van Handel, R.},
%	title={Mixed volumes and the Bochner method},
%	journal={Proc. Amer. Math. Soc.},
%	volume={147},
%	date={2019},
%	number={12},
%	pages={5385--5402},
%	issn={0002-9939},
%	review={\MR{4021097}},
%	doi={10.1090/proc/14651},
%}
%\bib{Shenfeld22}{article}{
%	author={Shenfeld, Y.},
%	author={van Handel, R.},
%	title={The extremals of Minkowski's quadratic inequality},
%	journal={Duke Math. J.},
%	volume={171},
%	date={2022},
%	number={4},
%	pages={957--1027},
%	issn={0012-7094},
%	review={\MR{4393790}},
%	doi={10.1215/00127094-2021-0033},
%}


\bib{Wang-Weng-Xia}{article}{
	TITLE = {Alexandrov-{F}enchel inequalities for convex hypersurfaces in
		the half-space with capillary boundary},
	JOURNAL = {Math. Ann.},
	FJOURNAL = {Mathematische Annalen},
	VOLUME = {388},
	YEAR = {2024},
	NUMBER = {2},
	PAGES = {2121--2154},
	ISSN = {0025-5831},
	MRCLASS = {53E40 (35K96 53C21 53C24)},
	MRNUMBER = {4700391},
	DOI = {10.1007/s00208-023-02571-4},
	URL = {https://doi.org/10.1007/s00208-023-02571-4},
	author={Wang, G.},
	author={Weng, L.},
	author={Xia, C.},
	%year={2023},
	%pages={	arXiv:2206.04639}
	%https://doi.org/10.1007/s00208-023-02571-4		
}

\bib{W-Z}{article}{
	title={The capillary Orlicz-Minkowski problem},
	author={Wang, X.},
	author={Zhu, B.},
	year={2025},
	eprint={2509.10859},
	archivePrefix={arXiv},
	primaryClass={math.DG},
	url={https://arxiv.org/abs/2509.10859},
}

\bib{W-X}{article}{
	AUTHOR = {Wei, Y.},
	AUTHOR = {Xiong, C.},
	TITLE = {A fully nonlinear locally constrained anisotropic curvature
		flow},
	JOURNAL = {Nonlinear Anal.},
	FJOURNAL = {Nonlinear Analysis. Theory, Methods \& Applications. An
		International Multidisciplinary Journal},
	VOLUME = {217},
	YEAR = {2022},
	PAGES = {Paper No. 112760, 29},
	ISSN = {0362-546X,1873-5215},
	MRCLASS = {53E10 (53C21)},
	MRNUMBER = {4361845},
	MRREVIEWER = {Xiaolong\ Li},
	DOI = {10.1016/j.na.2021.112760},
	URL = {https://doi.org/10.1016/j.na.2021.112760},
}

\bib{Xia13}{article}{
	author={Xia, C.},
	%author={},
	title={On an anisotropic Minkowski problem},
	journal={Indiana Univ. Math. J.},
	volume={62},
	number={5},
	pages={1399–1430},
	year={2013},	
}%13

\bib{Xia-phd}{book}{%MR3188548,
	AUTHOR = {Xia, C.},
	TITLE = {On a Class of Anisotropic Problem},
	SERIES = {PhD Thesis},
	PUBLISHER = {Albert-Ludwigs University Freiburg},
	YEAR = {2012},
}





%\bib{kkk-2023}{article}{%MR4582024,
	%	AUTHOR = {Kwong, Kwok-Kun},
	%	TITLE = {Weighted isoperimetric inequalities in warped product
		%		manifolds},
	%	JOURNAL = {Differential Geom. Appl.},
	%	FJOURNAL = {Differential Geometry and its Applications},
	%	VOLUME = {89},
	%	YEAR = {2023},
	%	PAGES = {Paper No. 102011, 28},
	%	ISSN = {0926-2245},
	%	MRCLASS = {53C23 (49Q20 53C15)},
	%	MRNUMBER = {4582024},
	%	DOI = {10.1016/j.difgeo.2023.102011},
	%	URL = {https://doi.org/10.1016/j.difgeo.2023.102011},
	%}


\bib{Xia-2017-convex}{article}{%MR3667425,
	AUTHOR = {Xia, C.},
	TITLE = {Inverse anisotropic curvature flow from convex hypersurfaces},
	JOURNAL = {J. Geom. Anal.},
	FJOURNAL = {Journal of Geometric Analysis},
	VOLUME = {27},
	YEAR = {2017},
	NUMBER = {3},
	PAGES = {2131--2154},
	ISSN = {1050-6926},
	MRCLASS = {53C44 (52A20)},
	MRNUMBER = {3667425},
	MRREVIEWER = {Alina Stancu},
	DOI = {10.1007/s12220-016-9755-2},
	URL = {https://doi.org/10.1007/s12220-016-9755-2},
}

\bib{Xia2017}{article}{%MR3667582
	title={Inverse anisotropic mean curvature flow and a Minkowski type inequality},
	author={Xia, C.},
	%journal={Advances in Mathematics},
	JOURNAL = {Adv. Math.},
	FJOURNAL = {Advances in Mathematics},
	volume={315},
	pages={102-129},
	year={2017},
}




\bib{Y-S-T}{article}{
	AUTHOR = {Yau, S. T.},
	TITLE = {On the {R}icci curvature of a compact {K}\"ahler manifold and
		the complex {M}onge-{A}mp\`ere equation. {I}},
	JOURNAL = {Comm. Pure Appl. Math.},
	FJOURNAL = {Communications on Pure and Applied Mathematics},
	VOLUME = {31},
	YEAR = {1978},
	NUMBER = {3},
	PAGES = {339--411},
	ISSN = {0010-3640,1097-0312},
	MRCLASS = {53C55 (32C10 35J60)},
	MRNUMBER = {480350},
	MRREVIEWER = {Robert\ E.\ Greene},
	DOI = {10.1002/cpa.3160310304},
	URL = {https://doi.org/10.1002/cpa.3160310304},
}



\bib{Zhu}{article}{
	author={Zhu, G.},
	title={The $L_p$ Minkowski problem for polytopes for $0<p<1$},
	journal={J. Funct. Anal.},
	volume={269},
	date={2015},
	number={4},
	pages={1070--1094},
	issn={0022-1236},
	review={\MR{3352764}},
	doi={10.1016/j.jfa.2015.05.007},
}

\bib{Zhu-2017}{article}{
	author={Zhu, G.},
	title={The $L_p$ Minkowski problem for polytopes for $p<0$},
	journal={Indiana Univ. Math. J.},
	volume={66},
	date={2017},
	number={4},
	pages={1333--1350},
	issn={0022-2518},
	review={\MR{3689334}},
	doi={10.1512/iumj.2017.66.6110},
}	










%%%%%%%%%%%%%%%%%%%%%%%%%%%%%%%%%%%%%%%%%%%%%%%%%%%%%%%%%%%%%%%%%%%%%%%%%%%%%%%%%%%%%%%%%%%%%%%%%%%%%%%%%%%%%%%%%%%%%%%%%%%%%%%%%%%%%%%%%%%%%%%%%%%%%%%%%%%%%%%%%%%%%%%%%%%%	




















%%%%%%%%%%%%%%%%%%%%%%%%%%%%%%%%%%%%%%%%%%%%%%%%%%%%%%%%%%%%%%%%%%%%%%%%%%%%%%%%%%%%%%%%%%%%%%%%%%%%%%%%%%%%%%%%%%%%%%%%%%%%%%%%%%%%%%%%%%%%%%%%%%%%%%%%%%%%%%%%%%%%%%




\end{biblist}
\end{document}